\documentclass[11pt]{article}
\usepackage{amsmath,amssymb,graphicx,enumerate,bbm,hyperref,amsthm}

\newcommand{\mathd}{\mathrm{d}}
\newcommand{\nin}{\not\in}
\newcommand{\tmem}[1]{{\em #1\/}}
\newcommand{\tmmathbf}[1]{\ensuremath{\boldsymbol{#1}}}
\newcommand{\tmop}[1]{\ensuremath{\operatorname{#1}}}
\newcommand{\tmtextit}[1]{{\itshape{#1}}}
\theoremstyle{plain}
\newtheorem{theorem}{Theorem}[section]
\newtheorem{proposition}[theorem]{Proposition}

\newtheorem{lemma}[theorem]{Lemma}

\theoremstyle{remark}
\newtheorem{remark}[theorem]{Remark}

\theoremstyle{definition}
\newtheorem{definition}[theorem]{Definition}

\numberwithin{equation}{section}

\begin{document}

\title{Slow dynamics for the dilute Ising model\\in the phase coexistence
region}\author{Marc Wouts\thanks{Universit\'e Paris 13, CNRS, UMR 7539 LAGA,
99 avenue Jean-Baptiste Cl\'ement, F-93430 Villetaneuse, France. E-mail:
{\tmem{wouts@math.univ-paris13.fr}}}}\maketitle

\begin{abstract}
  In this paper we consider the Glauber dynamics for a disordered
  ferromagnetic Ising model, in the region of phase coexistence. It was
  conjectured several decades ago that the spin autocorrelation decays as a negative power of time {\cite{HF87PhysRevB}}. We confirm this behavior by establishing a corresponding
  lower bound in any dimensions $d \geqslant 2$, together with an upper bound when $d=2$. Our approach is deeply connected to the Wulff construction for the dilute Ising model. We consider initial phase profiles with a reduced surface tension on their boundary and prove that, under mild conditions, those profiles are separated from the (equilibrium) pure plus phase by an energy barrier.
\end{abstract}

{\tableofcontents}

\section{Introduction and definitions}

\subsection{Introduction}

For many years the Ising model and the corresponding Glauber dynamics have been
a very active research field. In the 1990's the asymptotics of the spectral gap of the Ising
model with free boundary condition were connected to surface tension, see
{\cite{Tho89CMP,Mar99LNM}} and references therein. The inversion time of the infinite
volume Ising model phase under a small field was then related to Wulff energies
{\cite{Sch94CMP,SS98CMP}}. In the last decade, precise estimates were achieved
for the spectral gap and mixing time {\cite{BM02JSP}}. Recently impressive
moves towards evidence of Lifshitz behavior and mean-curvature displacement of interfaces were achieved in {\cite{MT10CMP,LS10,LMST10,CMST10}}.

The focus of the present paper is on the consequences of the presence of
disorder on the dynamics, in the phase coexistence region. Since~{\cite{Mar99LNM}} (and references therein)
it has been known that dilution in the Ising model triggers slow,
non-exponential relaxation to equilibrium in the Griffiths phase.
Here we focus on the {\tmem{phase coexistence}} region, which means
that, at equilibrium, the system can be either in the plus or the minus phase.
This setting was considered already in {\cite{HF87PhysRevB}} where heuristic
discussions suggested that autocorrelation decays as a negative power of time.
In the present paper we turn these heuristics into rigorous proofs. Previous stages
of this project were the adaptation of the coarse graining
and of the Wulff construction to the disordered setting, see {\cite{Wou08SPA,Wou09CMP}} respectively.

Our main result is a lower bound on the autocorrelation 
(Theorem \ref{thm-lwb-A}) which validates the heuristics of {\cite{HF87PhysRevB}}. That is to say, when both the initial configuration includes a droplet of the minus phase and the surface tension on the boundary of the droplet is smaller than its quenched value, the system must cross an energy barrier before the droplet can disappear. Interestingly we show that the energy gap can be computed on
continuous evolutions of the droplet, a slight improvement in comparison
with the usual scheme of computing the bottleneck as the maximum gap in energy
over all intermediate magnetization, as in \cite{Mar99LNM,BI04AHP}. 

We also present in this work an upper bound on the autocorrelation when $d=2$ (Theorem \ref{thm-upb-A}) together with some consequences of our estimates on the typical spectral gap and mixing times in finite volume (Theorems \ref{thm-Trel}). 

We would like to mention that, although is it the case here, we do not expect that
dilution always slows down relaxation. Indeed, for the infinite volume dilute Ising system subject to a small positive external field, the dilution has a catalyst effect on the transition from the minus to the (equilibrium) plus phase, cf.~\cite{BGW12}.

The organization of the paper is as follows: in the remaining part of the
current Section, we define the dilute Ising model and the Glauber dynamics. 
We also introduce the necessary tools and technical assumptions. Then in Section~\ref{sec:results} we present our main
results. Heuristics and proofs are given in Section~\ref{sec:proofs}.

\subsection{The dilute Ising model}

The canonical vectors of $\mathbbm{R}^d$ are denoted by $(\tmmathbf{e}_i)_{i = 1
\ldots d}$. For any $x = \sum_{i = 1}^n x_i \tmmathbf{e}_i = (x_1, \ldots,
x_d) \in \mathbbm{R}^d$ we consider the following norms:
\begin{equation}
  \|x\|_1 = \sum_{i = 1}^d |x_i | \text{, \ \ \ } \|x\|_2 = \left( \sum_{i =
  1}^d x_i^2 \right)^{1 / 2} \text{ \ \ and \ \ } \|x\|_{\infty} = \max_{i =
  1}^d |x_i |. \label{eq-norms}
\end{equation}
Given $x, y \in \mathbbm{Z}^d$ we say that $x, y$ are nearest neighbors (which
we denote $x \sim y$) if they are at Euclidean distance $1$, i.e. if $\|x -
y\|_2 = 1$. To any domain $\Lambda \subset \mathbbm{Z}^d$ we associate the
edge sets
\begin{eqnarray}
  E (\Lambda) & = & \left\{ \left\{ x, y \right\}: x, y \in \Lambda \text{ \
  and \ } x \sim y \right\} \\
  \text{and \ \ } E^w (\Lambda) & = & \left\{ \left\{ x, y \right\}: x \in
  \Lambda, y \in \mathbbm{Z}^d \text{ \ and \ } x \sim y \right\} . 
\end{eqnarray}
We consider in this paper the dilute Ising model on $\mathbbm{Z}^d$ for $d
\geqslant 2$. It is defined in two steps: first, the couplings between
adjacent spins are represented by a random sequence $J = (J_e)_{e \in E
(\mathbbm{Z}^d)}$ of law $\mathbbm{P}$, such that the $(J_e)_{e \in E
(\mathbbm{Z}^d)}$ are independent, identically distributed in $[0, 1]$ under
$\mathbbm{P}$. For convenience we write $\mathcal{J}= [0, 1]^{E
(\mathbbm{Z}^d)}$ the set of possible realizations of $J$. Given $\Lambda
\subset \mathbbm{Z}^d$ a finite domain, $J \in \mathcal{J}$ and a spin
configuration $\sigma \in \Sigma^+_{\Lambda} = \left\{ \sigma: \mathbbm{Z}^d \rightarrow \{\pm 1\}: \sigma_z = 1, \forall z \notin \Lambda \right\}$,
we let
\begin{equation}
  H_{\Lambda}^{J, +} (\sigma) = - \sum_{e =\{x, y\} \in E^w (\Lambda)} J_e
  \sigma_x \sigma_y
\end{equation}
the Hamiltonian with plus boundary condition on $\Lambda$. The dilute Ising
model on $\Lambda$ with plus boundary condition, given a realization $J$ of
the couplings, is the probability measure $\mu^{J, +}_{\Lambda}$ on
$\Sigma^+_{\Lambda}$ that satisfies
\begin{equation}
  \mu^{J, +}_{\Lambda} (\{\sigma\}) = \frac{1}{Z^{J, +}_{\Lambda, \beta}} \exp
  \left( - \frac{\beta}{2} H_{\Lambda}^{J, +} (\sigma) \right) \text{, \ \ }
  \forall \sigma \in \Sigma^+_{\Lambda}
\end{equation}
where $\beta \geqslant 0$ is the inverse temperature and $Z^{J, +}_{\Lambda,
\beta}$ is the partition function
\begin{equation}
  Z_{\Lambda, \beta}^{J, +} = \sum_{\sigma \in \Sigma^+_{\Lambda}} \exp \left(
  - \frac{\beta}{2} H_{\Lambda}^{J, +} (\sigma) \right) .
\end{equation}

Consider
\begin{equation}
  m_{\beta} = \lim_{N \rightarrow \infty} \mathbbm{E} \mu^{J,
  +}_{\hat{\Lambda}_N, \beta}  \left( \sigma_0 \right)
\end{equation}
the magnetization in the thermodynamic limit, where $\hat{\Lambda}_N$ is the
symmetric box $\hat{\Lambda}_N =\{- N, \ldots, N\}^d$ and $\mathbbm{E}$ the
expectation associated with $\mathbbm{P}$. When $m_{\beta} > 0$ the boundary
condition has an influence even if it is arbitrary far away from the origin.
In particular the decreasing limit $\mu^{J, +}_{\beta} = \lim_N \downarrow
\mu^{J, +}_{\hat{\Lambda}_N, \beta}$ does not coincide with $\mu^{J,
-}_{\beta} = \lim_N \uparrow \mu^{J, -}_{\hat{\Lambda}_N, \beta}$ ($\mathbb P$-almost surely). The two
infinite volume measures $\mathbbm{E} \mu^{J, +}_{\beta}$ and $\mathbbm{E}
\mu^{J, -}_{\beta}$ are respectively the plus and minus phases.

It was shown in {\cite{ACCN87JPhys}} that the dilute Ising model undergoes a
phase transition at low temperature when the random interactions percolate. In
our settings, this means that the critical inverse temperature
\begin{equation}
  \beta_c = \inf \left\{ \beta \geqslant 0: m_{\beta} > 0 \right\},
  \label{Bc}
\end{equation}
which is never smaller than $\beta_c^{\tmop{pure}}$ -- the critical inverse
temperature for the pure Ising model corresponding to $J \equiv 1$ -- is finite if and only
if $\mathbbm{P}(J_e > 0) > p_c (d)$ where $p_c (d)$ is the threshold for bond
percolation on $\mathbbm{Z}^d$.

\subsection{The Fortuin-Kasteleyn representation}

The Ising model has a percolation-like representation which is very convenient
for formulating two of the fundamental concepts for the study of equilibrium phase
coexistence: renormalization and surface tension. We call
\[ \Omega = \left\{ \omega: E (\mathbbm{Z}^d) \rightarrow \{0, 1\} \right\}
\]
the set of cluster configurations on $E (\mathbbm{Z}^d)$, and for any $\omega
\in \Omega$ and $E \subset E (\mathbbm{Z}^d)$ we call $\omega_{|E}$ the
restriction of $\omega$ to $E$, defined by
\[ (\omega_{|E})_e = \left\{ \begin{array}{ll}
     \omega_e & \text{if } e \in E\\
     0 & \text{else.}
   \end{array} \right. \]
The set of cluster configurations on $E$ is $\Omega_E =\{\omega_{|E}, \omega
\in \Omega\}$. Given a parameter $q \geqslant 1$ and an inverse temperature
$\beta \geqslant 0$, a realization of the random couplings $J: E
(\mathbbm{Z}^d) \rightarrow [0, 1]$, a finite edge set $E \subset E
(\mathbbm{Z}^d)$ and a boundary condition $\pi \in \Omega_{E^c}$ we consider
the random cluster model $\Phi_{E, \beta}^{J, \pi, q}$ on $\Omega_E$ defined
by
\begin{equation}
  \Phi_{E, \beta}^{J, \pi, q} \left( \{\omega\} \right) = \frac{1}{Z_{E,
  \beta}^{J, \pi, q}} \prod_{e \in E} p_e^{\omega_e} (1 - p_e)^{1 - \omega_e}
  \times q^{C^{\pi}_E (\omega)} \text{, \ \ \ } \forall \omega \in \Omega_E
  \label{eq-def-FK}
\end{equation}
where $p_e = 1 - \exp (- \beta J_e)$, $C^{\pi}_E (\omega)$ is the number of
clusters of the set of vertices in $\mathbbm{Z}^d$ attained by $E$ under the
wiring $\omega \vee \pi$ such that $(\omega \vee \pi)_e = \max (\omega_e,
\pi_e)$, and $Z_{E, \beta}^{J, \pi, q}$ is the renormalization constant that
makes of $\Phi_{E, \beta}^{J, \pi, q}$ a probability measure.

For convenience we use the same notation for the probability measure $\Phi_{E,
\beta}^{J, \pi, q}$ and for its expectation. Most of the time we will take
either $\pi = f$, where $f$ is the free boundary condition: $f_e = 0, \forall
e \in E^c$, or $\pi = w$ where $w$ is the wired boundary condition: $w_e = 1,
\forall e \in E^c$. When the parameters $q$ and $\beta$ are clear from the
context we omit them. Given $\mathcal{R}$ a compact subset of $\mathbbm{R}^d$
(usually a rectangular parallelepiped) we denote by $\Phi_{\mathcal{R}}^{J,
\pi}$ the measure $\Phi_{E ( \dot{\mathcal{R}} \cap \mathbbm{Z}^d)}^{J, \pi}$
on the cluster configurations on $E ( \dot{\mathcal{R}} \cap \mathbbm{Z}^d)$,
where $\dot{\mathcal{R}}$ stands for the interior of $\mathcal{R}$. In
particular, for any $g, h: \Omega \rightarrow \mathbbm{R}$ the quantities
$\Phi_{\mathcal{R}_1}^{J, \pi} (g)$ and $\Phi_{\mathcal{R}_2}^{J, \pi} (h)$
are independent under $\mathbbm{P}$ when $\mathcal{R}_1 \cap \mathcal{R}_2 = \emptyset$.

The connection between the dilute Ising model $\mu^{J, +}_{\Lambda, \beta}$
and the random-cluster model was made explicit in {\cite{ES88PhysRevD}}.
Consider the joint probability measure
\[ \Psi_{\Lambda, \beta}^{J, +} \left( \left\{ (\sigma, \omega) \right\}
   \right) = \frac{\tmmathbf{1}_{\left\{ \sigma \prec \omega
   \right\}}}{\tilde{Z}^{J, +}_{\Lambda, \beta}} \prod_{e \in E^w (\Lambda)}
   \left( p_e \right)^{\omega_e}  \left( 1 - p_e \right)^{1 - \omega_e},
   \text{ \ } \forall (\sigma, \omega) \in \Sigma^+_{\Lambda} \times \Omega_{E
   (\Lambda)} \]
where $p_e = 1 - \exp (- \beta J_e)$, $\sigma \prec \omega$ is the event that
$\sigma$ and $\omega$ are compatible, namely that $\omega_e = 1 \Rightarrow
\sigma_x = \sigma_y, \forall e =\{x, y\} \in E^w (\Lambda)$, and
$\tilde{Z}^{J, +}_{\Lambda, \beta}$ is the corresponding normalizing factor.
Then,
\begin{enumerate}[i.]
  \item The marginal of $\Psi_{\Lambda, \beta}^{J, +}$ on the variable
  $\sigma$ is the Ising model $\mu^{J, +}_{\Lambda}$,
  
  \item Its marginal on the variable $\omega$ is the random-cluster model
  $\Phi_{E (\Lambda), \beta}^{J, w, 2}$ with wired boundary condition $w$ and
  parameter $q = 2$.
  
  \item Conditionally on $\omega$, the spin $\sigma$ of each connected
  component of $\Lambda$ for $\omega$ (from now on {\tmem{cluster}}) is
  constant, and equal to $+ 1$ if the cluster is connected to $\Lambda^c$. The
  spin of all clusters not touching $\Lambda^c$ are independent and equal to
  $+ 1$ with a probability $1 / 2$.
  
  \item Conditionally on $\sigma$, the edges are open (i.e. $\omega_e = 1$ for
  $e =\{x, y\}$) independently, with respective probabilities $p_e
  \delta_{\sigma_x, \sigma_y}$.
\end{enumerate}

Furthermore, the distribution $\Phi_{E, \beta}^{J, \pi, q}$ increases with
$\beta, J, \pi$, it satisfies the FKG inequality and the DLR equation,
cf.~{\cite{ACCN88JSP}}.

\subsection{Slab percolation}

\label{sec-renorm}We say that {\tmem{slab percolation}} holds under
$\text{$\mathbb{E} \Phi^{J, f, q}_{E (\mathbbm{Z}^d), \beta}$}$ when either
\begin{eqnarray*}
  d \geqslant 3 & \text{and} & \exists H \in \mathbbm{N}^{\star}, \inf_{N \in
  \mathbbm{N}^{\star}} \inf_{x, y \in S_{N, H}} \mathbb{E} \Phi^{J, f,
  q}_{S_{N, H}, \beta} (x \overset{\omega}{\leftrightarrow} y) > 0,\\
  \text{or \ } d = 2 & \text{and } & \exists \kappa: \mathbbm{N}^{\star}
  \mapsto \mathbbm{N}^{\star}, \lim_{N \rightarrow \infty} \mathbb{E} \Phi^{J,
  f, q}_{S_{N, \kappa (N)}, \beta} (\exists \text{ a horizontal crossing for
  } \omega) > 0
\end{eqnarray*}
where $S_{N, H} =\{1, \ldots, N\}^{d - 1} \times \{1, \ldots, H\}$ is the slab
of height $H$. The critical threshold for slab percolation is
\begin{equation}
  \hat{\beta}_c = \inf \left\{ \beta \geqslant 0: \text{slab percolation
  occurs under } \text{$\mathbb{E} \Phi^{J, f, q}_{E (\mathbbm{Z}^d), \beta}$}
  \right\} . \label{eq-betasp}
\end{equation}
We believe that $\hat{\beta}_c$ and $\beta_c$ coincide, where $\beta_c$ is the
critical inverse temperature for the dilute Ising model defined at (\ref{Bc}).
We also consider
\begin{eqnarray}
  \mathcal{N} & = & \left\{ \beta \geqslant 0: \lim_{N \rightarrow \infty}
  \mathbb{E} \Phi^{J, f, q}_{\hat{\Lambda}_N, \beta} \neq \lim_{N \rightarrow
  \infty} \mathbb{E} \Phi^{J, w, q}_{\hat{\Lambda}_N, \beta} \right\}
  \label{N}, 
\end{eqnarray}
the set of inverse temperatures at which the infinite volume random media
random cluster measure is not unique. It was shown in {\cite{Wou08SPA}},
Theorem 2.3, that $\mathcal{N}$ is at most countable. Under the assumptions
$\beta > \hat{\beta}_c$ and $\beta \nin \mathcal{N}$, one can use a
renormalization procedure (Theorem 5.1 in {\cite{Wou08SPA}}) which gives a
precise meaning to the notion of plus and minus phases and is hence a
fundamental tool for the study of equilibrium phase
coexistence.

\subsection{Surface tension}

Surface tension is another essential tool for the study of equilibrium
phase coexistence. In the context of the dilute Ising model, it is a random
quantity since it depends on the couplings $J$. We recall here some important definitions and results
from~{\cite{Wou09CMP}}. Let $S^{d - 1}$ be the set of unit vectors of $\mathbbm{R}^d$. Given
$\tmmathbf{n} \in S^{d - 1}$ we let
\[ \mathbbm{S}_{\tmmathbf{n}} = \left\{ \sum_{k = 1}^{d - 1} [- 1 /
   2, 1/2]\tmmathbf{u}_k ; (\tmmathbf{u}_1, \ldots, \tmmathbf{u}_{d - 1},
   \tmmathbf{n}) \text{ is an orthonormal basis of } \mathbbm{R}^d \right\} \]
(where $\sum$ stands for the Minkowski addition)
be the set of $d - 1$ dimensional hypercubes of side-length $1$, centered at
$0$, orthogonal to $\tmmathbf{n}$. Finally, given $\mathcal{S} \in
\mathbbm{S}_{\tmmathbf{n}}$, $x \in \mathbbm{R}^d$ and $L, H > 0$ we denote
\begin{equation}
  \mathcal{R}_{x, L, H} (\mathcal{S}, \tmmathbf{n}) = x + L\mathcal{S}+ [- H,
  H]\tmmathbf{n} \label{eq-def-R}
\end{equation}
the rectangular parallelepiped centered at $x$, with basis $x + L\mathcal{S}$
and extension $2 H$ in the direction $\tmmathbf{n}$. The discrete version of
$\mathcal{R}$ is $\hat{\mathcal{R}} = \dot{\mathcal{R}} \cap \mathbbm{Z}^d$
and the inner discrete boundary of $\mathcal{R}$ is
\[ \partial \hat{\mathcal{R}} = \left\{ y \in \hat{\mathcal{R}}: \exists z
   \in \mathbbm{Z}^d \setminus \hat{\mathcal{R}}, z \sim y \right\} . \]
For any $\mathcal{R}$ as in (\ref{eq-def-R}) we decompose $\partial
\hat{\mathcal{R}}$ into its \tmtextit{upper} and \tmtextit{lower} parts
$\partial^+  \hat{\mathcal{R}} =\{y \in \partial \hat{\mathcal{R}}: (y - x)
\cdot \tmmathbf{n} \geqslant 0\}$ and $\partial^-  \hat{\mathcal{R}} =\{y \in
\partial \hat{\mathcal{R}}: (y - x) \cdot \tmmathbf{n}< 0\}$. Then we call
\begin{equation}
  \mathcal{D}_{\mathcal{R}} = \left\{ \text{$\omega \in \Omega$}: \partial^+ 
  \hat{\mathcal{R}}  \overset{\omega}{\nleftrightarrow} \partial^- 
  \hat{\mathcal{R}} \right\} \label{eq-def-DR}
\end{equation}
the event of disconnection between the upper and lower parts of $\partial
\hat{\mathcal{R}}$, and
\begin{equation}
  \tau^J_{\mathcal{R}} = - \frac{1}{L^{d - 1}} \log \Phi^{J, w}_{\mathcal{R}}
  \left( \mathcal{D}_{\mathcal{R}} \right) \label{eq-def-tauJ} .
\end{equation}
the surface tension in $\mathcal{R}$. Surface tension is sub-additive and has
a typical {\tmem{quenched }}value
\begin{eqnarray}
  \tau^q_{\beta} (\tmmathbf{n}) & = & \lim_{N \rightarrow \infty}
  \tau^J_{\mathcal{R}_{0, N, \delta N} (\mathcal{S}, \tmmathbf{n})} \text{ \
  in $\mathbbm{P}$-probability}  \label{def-tauq}
\end{eqnarray}
that does not depend on $\delta > 0$ nor on $\mathcal{S} \in
\mathbbm{S}_{\tmmathbf{n}}$ (Theorem 1.3 in {\cite{Wou09CMP}}). It is positive
for any $\beta > \hat{\beta}_c$ (Proposition 1.5 in the same reference). We
denote by $J^{\min}$ and $J^{\max}$ the extremal values of the support of $J$.
Since $\tau^J_{\mathcal{R}}$ increases with $J$, surface tension can be as low
as $\tau^{\min}_{\mathcal{R}}$ that corresponds to the constant couplings $J
\equiv J^{\min}$, and as large as $\tau^{\max}_{\mathcal{R}}$ when $J \equiv
J^{\max}$. According to the convergence in (\ref{def-tauq}),
$\tau^{\min}_{\mathcal{R}}$ and $\tau^{\max}_{\mathcal{R}}$ also converge when
$\mathcal{R}=\mathcal{R}_{0, N, \delta N} (\mathcal{S}, \tmmathbf{n})$ with $N
\rightarrow \infty$ and we call their respective limits $\tau^{\min}
(\tmmathbf{n})$ and $\tau^{\max} (\tmmathbf{n})$. Surface tension can deviate
from $\tau^q_{\beta} (\tmmathbf{n})$. Upper large deviations happen at a
volume order (Theorem~1.4 in \cite{Wou09CMP}) and are irrelevant to surface
phenomenon like phase coexistence. Lower deviations under $\tau > \tau^{\min}
(\tmmathbf{n})$ occur according to the rate function
\begin{eqnarray}
  I_{\tmmathbf{n}} (\tau) & = & \lim_N - \frac{1}{N^{d - 1}} \log \mathbbm{P}
  \left( \tau^J_{\mathcal{R}_{0, N, \delta N} (\mathcal{S}, \tmmathbf{n})}
  \leqslant \tau \right),  \label{def-In}
\end{eqnarray}
see Theorem 1.6 in {\cite{Wou09CMP}}. The set
\begin{eqnarray}
  \mathcal{N}_I & = & \left\{ \beta \geqslant 0: \exists \tmmathbf{n} \in
  S^{d - 1} \text{ and } r > 0 \text{ such that } I_{\beta, \tmmathbf{n}}
  (\tau^q_{\beta} (\tmmathbf{n}) - r) = 0 \right\} \label{NI}
\end{eqnarray}
is at most countable, see Corollary~1.9 therein.

\subsection{Magnetization profiles}

In the following, $\mathcal{L}^d$ stands for the Lebesgue measure on
$\mathbbm{R}^d$ and $\mathcal{H}^{d - 1}$ for the $d - 1$ dimensional
Hausdorff measure. The $L^1$-distance between two Borel measurable functions
$u, v: [0, 1]^d \rightarrow \mathbbm{R}$ is
\[ \|u - v\|_{L^1} = \int_{[0, 1]^d} |u - v|d\mathcal{L}^d, \]
and the set $L^1$ is
\[ \left\{ u: [0, 1]^d \rightarrow \mathbbm{R} \text{ Borel measurable},
   \|u\|_{L^1} < \infty \right\} . \]
In order that $L^1$ be a Banach space for the $L^1$-norm, we identify $u: [0,
1]^d \rightarrow \mathbbm{R}$ with the class of functions $\{v: \|u -
v\|_{L^1} = 0\}$ that coincide with $u$ on a set of full measure. We also
denote by $\mathcal{V}(u, \delta)$ the neighborhood of radius $\delta > 0$ in
$L^1$ around $u \in L^1$. Given a Borel set $U \subset \mathbbm{R}^d$, we call
\[ \chi_U \text{: } x \in \mathbbm{R}^d \mapsto \left\{ \begin{array}{ll}
     1 & \text{if } x \nin U\\
     - 1 & \text{if } x \in U
   \end{array} \right. \]
the phase profile corresponding to $U$, and call $\mathcal{P}(U)$ the
perimeter of $U$ (as defined in Chap.~3 of {\cite{AFP00Book}}). The set of bounded
variation profiles is
\[ \tmop{BV} = \left\{ u = \chi_U: U \subset (0, 1)^d \text{ is a Borel set
   and } \mathcal{P}(U) < \infty \right\} . \]
Bounded variations profiles $u = \chi_U \in \tmop{BV}$ have a
\tmtextit{reduced boundary} $\partial^{\star} u$ and an outer normal
$\tmmathbf{n}_.^u: \partial^{\star} u \rightarrow S^{d - 1}$ with, in
particular, $\mathcal{H}^{d - 1} (\partial^{\star} u) =\mathcal{P}(U)$. As the
outer normal $\tmmathbf{n}_.^u$ defined on $\partial^{\star} u$ is Borel
measurable, we can consider integrals of the kind
\begin{equation}
  \mathcal{F}^q (u) = \int_{\partial^{\star} u} \tau^q (\tmmathbf{n}_x^u)
  d\mathcal{H}^{d - 1} (x) \label{eq-def-Ftau} \text{, \ \ \ \ } \forall u \in
  \tmop{BV}
\end{equation}
that define the quenched surface energy of a given profile. When $u = \chi_U
\in \tmop{BV}$ we also denote by $\mathcal{F}^q (U)$ the surface energy of
$u$.

\subsection{Initial configuration and gap in surface energy}

As described with further detail in the heuristics (Section~\ref{sec:heur}), our strategy for controlling the Glauber dynamics is to start from some
metastable {\tmem{initial configuration}}, from which a positive gap in surface energy must be overcome before the system can reach the pure plus phase. This metastable configuration is characterized, on the one side, by an initial phase profile $u_0 \in \tmop{BV}$, and on the second side, by a reduced surface tension $\tau^r$ on the boundary of $u_0$. So a so-called {\it initial configuration} has actually two microscopic counterparts. First, to $u_0$ will correspond a set of initial spin configurations for the Glauber dynamics, while to the reduced surface tension $\tau^r$ we will associate a {\it dilution} event on the couplings $J$.

Before we can define the set of
initial configurations IC at (\ref{IC}), we need still a few more definitions.

\begin{definition}
  \label{def-u-regular}We say that a profile $u = \chi_U$ is regular if
  \begin{enumerate}[i.]
    \item $U$ is open and at positive distance from the boundary $\partial [0,
    1]^d$ of the unit cube,
    
    \item $\partial U$ is $d - 1$ rectifiable and
    
    \item for small enough $r > 0$, $[0, 1]^d \setminus \left( \partial U + B
    (0, r) \right)$ has exactly two connected components.
  \end{enumerate}
\end{definition}

We recall that $E \subset \mathbbm{R}^d$ is a $d - 1$ rectifiable set if there
exists a Lipschitz function mapping some bounded subset of $\mathbbm{R}^{d
- 1}$ onto $E$ (Definition 3.2.14 in {\cite{Fed69Book}}). It is the case in
particular of the boundary of Wulff crystals (Theorem 3.2.35 in
{\cite{Fed69Book}}) and of bounded polyhedral sets. It follows from
Proposition 3.62 in {\cite{AFP00Book}} that any $u = \chi_U$ regular belongs
to $\tmop{BV}$ and that $\partial U = \partial^{\star} u$ up to a
$\mathcal{H}^{d - 1}$-negligible set. Finally, we call
\begin{equation}
  \tmop{IC} = \left\{ \begin{array}{l}
    (u_0, \tau^r) \in \tmop{BV} \times \mathcal{C}([0, 1]^d, \mathbbm{R}):
    \text{ $u_0$ is regular and there is}\\
    \text{$\varepsilon > 0$ such that, for all } x \in \partial^{\star} u_0,
    \tau^{\min} (\tmmathbf{n}_x^{u_0}) + \varepsilon < \tau^r (x) \leqslant
    \tau^q (\tmmathbf{n}_x^{u_0}) .
  \end{array} \right\} \label{IC}
\end{equation}
Given an initial configuration $(u_0, \tau^r) \in \tmop{IC}$, we define the
{\tmem{reduced surface energy }}as
\begin{equation}
  \mathcal{F}^r \left( u \right) = \int_{\partial^{\star} u_0 \cap
  \partial^{\star} u} \tau^r (x) d\mathcal{H}^{d - 1} (x) +
  \int_{\partial^{\star} u \setminus \partial^{\star} u_0} \tau^q
  (\tmmathbf{n}_x) d\mathcal{H}^{d - 1} (x) \label{eq-def-Fr}
\end{equation}
which is obviously smaller than $\mathcal{F}^q \left( u \right)$. The reduced surface
energy of the initial phase profile $u_0$ is
\begin{equation}
  \mathcal{F}^r \left( u_0 \right) = \int_{\partial^{\star} u_0} \tau^r (x)
  d\mathcal{H}^{d - 1} (x) 
\end{equation}
while the cost of dilution is
\begin{equation}
  \mathcal{I}^r (u_0) = \int_{\partial^{\star} u_0} I_{\tmmathbf{n}_x} (\tau^r
  (x)) d\mathcal{H}^{d - 1} (x) \label{eq-def-Ir} .
\end{equation}

For any $\varepsilon > 0$, we call $\mathcal{C}_{\varepsilon} (u_0)$ the set
of sequences of phase profiles that evolve from $u_0$ to $\tmmathbf{1}$ (pure
plus phase) by jumps with $L^1$ norm less than $\varepsilon$, that is
\begin{equation}
  \mathcal{C}_{\varepsilon} (u_0) = \left\{ (v_i)_{i = 0 \ldots k}:
  \begin{array}{l}
    k \in \mathbbm{N} \text{ ; } v_0 = u_0 \text{ and } v_k =\tmmathbf{1}\\
    \forall i \in \{0, \ldots, k - 1\}, v_i \in \tmop{BV} \text{ and } \|v_{i
    + 1} - v_i \|_{L^1} \leqslant \varepsilon
  \end{array} \right\} . \label{Cepsu0}
\end{equation}
Finally we define the {\tmem{gap in surface energy}} or {\it energy barrier} for removing the droplet
$u_0$, under the reduced surface tension $\tau^r$: this is
\begin{equation}
  \mathcal{K}^r (u_0) = \lim_{\varepsilon \rightarrow 0^+} \inf_{v \in
  \mathcal{C}_{\varepsilon} (u_0)} \max_i \mathcal{F}^r (v_i) -\mathcal{F}^r
  (u_0) \label{eq-def-Kr-disc} .
\end{equation}

\subsection{The Glauber dynamics}

The Glauber dynamics is characterized by a family of \tmtextit{transition
rates} $c^J (x, \sigma)$ at which the configuration $\sigma$ changes to
$\sigma^x$ defined by
\[ \sigma^x_y = \left\{ \begin{array}{ll}
     \sigma_y & \text{if } y \neq x\\
     - \sigma_y & \tmop{if} y = x.
   \end{array} \right. \]
In other words, $c^J (x, \sigma)$ is also the rate at which the spin at $x$
flips. We make standard assumptions on the transition rates, namely:

\begin{description}
  \item[Finite range] There exists $r < \infty$, the \tmtextit{range of
  interaction}, such that $c^J (x, \sigma)$ is independent of $\sigma (y)$
  when $d (x, y) > r$, and of $J_e$ when $d (e, x) > r$.
  
  \item[Rates are bounded] The rates are uniformly bounded from below and
  above: there are $c_m, c_M \in (0, \infty)$ such that
  \[ c_m \leqslant c^J (x, \sigma) \leqslant c_M \text{, \ \ } \forall x \in
     \mathbbm{Z}^d, J \in \mathcal{J}, \sigma \in \Sigma . \]
  \item[Detailed balance] The rates satisfy the \tmtextit{detailed balance
  condition}: for all $\sigma \in \Sigma$ and $x \in \mathbbm{Z}^d$, for all
  $J \in \mathcal{J}$, the product
  \[ c^J (x, \sigma) \times \exp \left( \frac{\beta}{2} \sum_{y \sim x} J_{\{x, y\}}
     \sigma_x \sigma_y \right) \]
  does not depend on $\sigma_x$.
  
  \item[Translation invariance] If, for some $z \in \mathbbm{Z}^d$ one has
  \[ J'_e = J_{z + e}, \forall e \in E (\mathbbm{Z}^d) \text{ \ \ and \ \ }
     \sigma'_x = \sigma_{x + z}, \forall x \in \mathbbm{Z}^d, \]
  then $c^{J'} (x, \sigma') = c^J (x + z, \sigma)$.
  
  \item[Attractivity] Given any $\sigma, \sigma' \in \Sigma$ with $\sigma
  \leqslant \sigma'$, the equality $\sigma_x = \sigma'_x$ implies
  \[ \sigma'_x c^J (x, \sigma') \leqslant \sigma_x c^J (x, \sigma) . \]
\end{description}

Two important examples of the Glauber dynamics are the Metropolis dynamics,
often used in computer simulations, for which
\begin{eqnarray*}
  c^J (x, \sigma) & = & \max \left( 1, \exp \left( - \beta \sum_{y \sim x}
  J_{\{x, y\}} \sigma_x \sigma_y \right) \right)
\end{eqnarray*}
and the heat-bath dynamics
\begin{eqnarray*}
  c^J (x, \tau) & = & \mu^J_{\{x\}} \left( \sigma_x = - \tau_x | \sigma_y =
  \tau_y, \forall y \neq x \right) .
\end{eqnarray*}
Given the transition rates, one can proceed to a graphical construction of the
dynamics, as follows: equip each site $x \in \mathbbm{Z}^d$ with a Poisson
process valued on~$\mathbbm{R}^+$, with intensity $c_M$ (we recall that $c_M$ is a uniform bound on the rates of the dynamics). Consider now the time
$t$ growing from $0$. When the Poisson process at $x$ has a point at $t$, flip
the spin at position $x$ with probability $c^J (x, \sigma_t) / c_M$. Because
the flip rates are bounded, the determination of $\sigma_t (x)$ involves only
finitely many sites and therefore the dynamics is well defined, even in the
infinite domain $\mathbbm{Z}^d$.

We call $\tmmathbf{P}^J$ the law of the Markov process $(\sigma (t))_{t
\geqslant 0}$ associated to this dynamics, and $\tmmathbf{P}^J_{\sigma (0)}$
the law conditioned on the initial configuration $\sigma (0)$. It is
convenient to introduce the semi-group $T^J$ defined by
\begin{equation}
  \left[ T^J (t) f \right] (\sigma) =\tmmathbf{E}^J_{\sigma} \left( f (\sigma
  (t)) \right) .
\end{equation}
The detailed balance condition makes the generator
\begin{equation}
  (L^J f) (\sigma) = \sum_{x \in \mathbbm{Z}^d} c^J (x, \sigma) (f (\sigma^x)
  - f (\sigma))
\end{equation}
self-adjoint in $L^2 (\mu^{J, +})$, and ensures that the Gibbs measure
$\mu^{J, +}$ is reversible for the dynamics. One way of quantifying the
approach to equilibrium of the dynamics in infinite volume is therefore the
averaged autocorrelation
\begin{eqnarray}
  A^{\lambda} (t) & = & \mathbbm{E} \left( \left[ \tmop{Var}_{\mu^{J, +}} (T^J
  (t) \pi_0) \right]^{\lambda} \right)  \label{def-A}\\
  & = & \mathbbm{E} \left( \left[ \int_{\Sigma} \left[ (T^J (t) \pi_0) (\rho)
  - \mu^{J, +} (\sigma_0) \right]^2 d \mu^{J, +} (\rho) \right]^{\lambda}
  \right) \nonumber\\
  & = & \mathbbm{E} \left\| T^J (t) \pi_0 - \mu^{J, +} (\sigma_0)
  \right\|_{L^2 (\mu^{J, +})}^{2 \lambda} \nonumber
\end{eqnarray}
where $\pi_0: \Sigma \rightarrow \mathbbm{R}$ is the function that, to the
spin configuration $\sigma$, associates $\pi_0 (\sigma) = \sigma_0$, and
$\lambda > 0$ is an arbitrary positive number.

Although our work aims primary at describing the asymptotics of the averaged
autocorrelation, we also derive some upper bounds on the relaxation and mixing
times in finite volume. To this aim, we introduce the Dirichlet form
\begin{eqnarray*}
  \mathcal{E}^{J, \rho}_{\Lambda} (f, f) & = & \frac{1}{2} \sum_{\sigma, x}
  \mu_{\Lambda}^{J, \rho} (\sigma) c^J (x, \sigma) (f (\sigma^x) - f
  (\sigma))^2
\end{eqnarray*}
and the spectral gap
\begin{eqnarray}
  \tmop{gap} L^{J, \Lambda, \rho} & = & \inf \left\{ \frac{\mathcal{E}^{J,
  \rho}_{\Lambda} (f, f)}{\tmop{Var}_{\Lambda}^{J, \rho}(f)} ; f \text{ with }
  \tmop{Var}_{\Lambda}^{J, \rho}(f) \neq 0 \right\} .  \label{eq-def-gap}
\end{eqnarray}
The inverse of the spectral gap is the relaxation time
\begin{eqnarray}
  \mathcal{T}_{\tmop{rel}}^{J, \Lambda, \rho} & = & 1 / \tmop{gap} L^{J,
  \Lambda, \rho} .  \label{eq-def-Trel}
\end{eqnarray}
One fundamental property of the spectral gap is that, for any $f$,
\begin{eqnarray}
  \tmop{Var}_{\Lambda}^{J, \rho} (T_{\Lambda}^{J, \rho}(t) f) & \leqslant & \exp
  \left( - 2 t / \mathcal{T}_{\tmop{rel}}^{J, \Lambda, \rho} \right)
  \tmop{Var}_{\Lambda}^{J, \rho} (f)  \label{eq-decay-var-gap}
\end{eqnarray}
where $T_{\Lambda}^{J, \rho}$ is the semi-group corresponding to the Glauber dynamics restricted to $\Lambda$ with boundary condition $\rho$.

Finally, we define the total variation distance between two probability
measures $\mu$ and $\nu$, as
\begin{eqnarray*}
  \left\| \mu - \nu \right\|_{\tmop{TV}} & = & \sup_A \mu (A) - \nu (A) .
\end{eqnarray*}
The mixing time is
\begin{eqnarray}
  \mathcal{T}_{\tmop{mix}}^{J, \Lambda, \rho} & = & \inf \left\{ t > 0:
  \sup_{\sigma_0: \sigma_0 = \rho \text{ on } \Lambda^c} \left\|
  \tmmathbf{P}^{J, \Lambda}_{\sigma_0} (\sigma_t \in .) - \mu^{J,
  \rho}_{\Lambda} \right\|_{\tmop{TV}} \leqslant e^{- 1} \right\} .
  \label{eq-def-Tmix}
\end{eqnarray}
Given
any function $f$, we have
\begin{eqnarray}
  \sup_{\sigma_0: \sigma_0 = \rho \text{ on } \Lambda^c} \left|
  \tmmathbf{E}^{J, \Lambda}_{\sigma_0} (f (\sigma_t)) - \mu^{J,
  \rho}_{\Lambda} (f) \right| & \leqslant & \exp \left( - \left\lfloor
  {t}/{\mathcal{T}_{\tmop{mix}}^{J, \Lambda, \rho}} \right\rfloor \right)
  \|f\|_{\infty} .  \label{eq-decay-Tmix}
\end{eqnarray}

\section{Main Results}

\label{sec:results}\subsection{Slow dynamics in infinite volume}

The main result of the paper is a rigorous lower bound on the averaged
autocorrelation which confirm some of the claims of the foreseeing paper
{\cite{HF87PhysRevB}}. We recall that the set of (metastable) initial configurations IC is
defined at (\ref{IC}), while the surface energy $\mathcal{F}^r$, the cost of
initial dilution $\mathcal{I}^r (u_0)$ and the gap in surface energy
$\mathcal{K}^r (u_0)$ associated to $(u_0, \tau^r) \in \tmop{IC}$ are defined
at (\ref{eq-def-Fr}), (\ref{eq-def-Ir}) and (\ref{eq-def-Kr-disc}),
respectively. Now we define the exponent
\begin{equation}
  \mathcal{X}_{\lambda} = \inf_{(u_0, \tau^r) \in \tmop{IC}: \mathcal{K}^r
  (u_0) > 0} \frac{\mathcal{I}^r (u_0) + \lambda \mathcal{F}^r
  (u_0)}{\mathcal{K}^r (u_0)} \label{eq-def-Xl}
\end{equation}
for any $\lambda \geqslant 0$. When no initial configuration $(u_0, \tau^r)
\in \tmop{IC}$ leads to a positive surface energy gap $\mathcal{K}^r (u_0) >
0$ we adopt the convention that $\mathcal{X}_{\lambda} = + \infty$. Under mild
conditions the exponent $\mathcal{X}_{\lambda}$ is finite and even bounded
from above:

\begin{proposition}
  \label{prop:gap}Assume that $\tau^{\min} (\tmmathbf{n}) < \tau^q
  (\tmmathbf{n})$ for all $\tmmathbf{n} \in \mathcal{S}^{d - 1}$. Then:
\begin{enumerate}  
\item
  For
  all $(u_0, \tau^r) \in \tmop{IC}$ such that the boundary of $u_0$ is
  $\mathcal{C}^1$, the gap in surface energy $\mathcal{K}^r (u_0)$ is strictly
  positive. 
  \item If $0 <\mathbbm{P}(J_e = 0) < 1 - p_c (d)$, for
  every $\lambda \geqslant 0$ there is $C > 0$ such that, for any $\beta$
  large enough,
  \begin{eqnarray}
    \mathcal{X}_{\lambda} & \leqslant & \frac{C}{\beta} .  \label{eq-upb-Xl}
  \end{eqnarray}
  \end{enumerate}
\end{proposition}
Note that the above proposition is a corollary of Theorem~\ref{thm:Kr:pos} which is to be found, together with its proof, in Section~\ref{sec:lwb:gap}.

Now we present our main result, which relates the decay of the autocorrelation of the infinite volume Glauber dynamics in the plus phase with the exponent $\mathcal{X}_{\lambda}$ (we recall that $\mathcal{N}$ and $\mathcal{N}_I$ are defined respectively at (\ref{N}) and (\ref{NI})).

\begin{theorem}
  \label{thm-lwb-A}For any $\beta > \hat{\beta}_c$ such that $\beta \notin
  \mathcal{N} \cup \mathcal{N}_I$, for any $\lambda \geqslant 1$, for any $\delta >
  0$, for any $t > 0$ large enough,
  \begin{eqnarray}
    A^{\lambda} (t) & \geqslant & t^{-\mathcal{X}_{\lambda} - \delta} . 
    \label{eq-lwb-A}
  \end{eqnarray}
\end{theorem}

\begin{remark}
  In general we have not been able to compute $\mathcal{K}^r (u_0)$ nor
  $\mathcal{X}_{\lambda}$. Note however that we give another formulation of
  $\mathcal{K}^r (u_0)$ in Theorem \ref{thm-Kdisc-Kcont}. This alternate
  formulation is the key for computing $\mathcal{K}^r (u_0)$ in two particular
  cases, see Sections \ref{sec-Kr-circle} and \ref{sec-Kr-square}. These
  computations also point out the fact that the constraint of continuous
  evolution in the definition of $\mathcal{K}^r (u_0)$ at
  (\ref{eq-def-Kr-disc}) makes, in some cases, the gap in surface energy
  bigger than if we only take into account the constraint of continuous
  evolution of the overall magnetization (see Lemma \ref{lem:Kr:mag}).
\end{remark}

Theorem~\ref{thm-lwb-A} above was established some time ago during the PhD
Thesis {\cite{Wou07PhD}} directed by Thierry Bodineau. Later on the author
received the indications by Fabio Martinelli of how a corresponding upper bound
could be established. The result is complementary to the former Theorem and
confirms that the autocorrelation indeed decays as a power of time when $d =
2$. Furthermore, when $0 <\mathbbm{P}(J_e = 0) < 1 - p_c (d)$, the exponent in both the upper and the lower bounds match up to a multiplicative constant as $\beta \rightarrow \infty$.

\begin{theorem}
  \label{thm-upb-A}Assume that $d = 2$. For any $\lambda > 0$, there is $c >
  0$ such that, for all $\beta > \hat{\beta}_c$, for all $t > 0$ large enough,
  \begin{eqnarray}
    A^{\lambda} (t) & \leqslant & t^{- c / \beta} .  \label{Aupb1}
  \end{eqnarray}
\end{theorem}

\subsection{Slow dynamics in finite volume}

Our strategy for providing a lower bound on the autocorrelation also yields lower bounds
on typical relaxation and mixing times in finite volume. Given $d
\geqslant 2$, the law of interactions $\mathbbm{P}$ and the inverse
temperature $\beta$, we call
\begin{eqnarray*}
  \kappa & = & d \sup_{(u_0, \tau^r) \in \tmop{IC}} \frac{\mathcal{K}^r
  (u_0)}{\mathcal{I}^r (u_0)}\\
  & = & d /\mathcal{X}_0
\end{eqnarray*}
We recall that the relaxation time $\mathcal{T}_{\tmop{rel}}$ was defined at
(\ref{eq-def-Trel}), while the mixing time $\mathcal{T}_{\tmop{mix}}$ was
defined at (\ref{eq-def-Tmix}).

\begin{theorem}
  \label{thm-Trel}Assume $d \geqslant 2$ and $\beta > \hat{\beta}_c$ with
  $\beta \notin \mathcal{N} \cup \mathcal{N}_I$. For any $\delta > 0$,
  \[ \lim_N \mathbbm{P} \left( \mathcal{T}^{J, \Lambda_N, +}_{\tmop{rel}}
     \text{ } \geqslant \text{ } N^{\kappa - \delta} \right) = \lim_N
     \mathbbm{P} \left( \mathcal{T}^{J, \Lambda_N, +}_{\tmop{mix}} \text{ }
     \geqslant \text{ } N^{\kappa - \delta} \right) = 1. \]
\end{theorem}

It is instructive to compare this bound with the asymptotics of the relaxation
time in the pure Ising model. When $d = 2$,
$\mathcal{T}^{J \equiv 1, \Lambda_N, +}_{\tmop{rel}} = N$ apart from
logarithmic corrections, see {\cite{BM02JSP}}. For dilute models with $0
<\mathbbm{P}(J_e = 0) < 1 - p_c (d)$, Proposition \ref{prop:gap} states that
$\kappa \geqslant c \beta$ for $\beta$ large enough: this is another
illustration of the fact that dilution makes the relaxation time much larger.

\section{Proofs}\label{sec:proofs}

\subsection{Heuristics and organization of the proofs}
\label{sec:heur} 
The object of Section~\ref{sec:upb} is the proof of the upper bound on the autocorrelation  when the dimension is $d=2$ (Theorem~\ref{thm-upb-A}). The strategy for the proof, which we owe to Fabio Martinelli, relies on a uniform lower bound on the spectral gap in a square box with uniform plus boundary condition (cf.~Theorem 6.4 in \cite{Mar99LNM}  or (\ref{eq:Ma:lgap}) below). Some extra but classical work is then required to use that estimate for the infinite volume Glauber dynamics started from the plus phase.

Section~\ref{sec-autocor} concentrates the most complex part of the paper and is dedicated to the proofs of Theorems~\ref{thm-lwb-A} and \ref{thm-Trel}. The corresponding heuristics are derived from \cite{HF87PhysRevB}, where it was already suggested that dilution and reduction of surface tension, which has a surface cost, could trigger metastability for initial spin configurations corresponding to the minus phase in the region surrounded by the diluted surface (i.e., with our notation, the region where $u_0\equiv-1$, up to an appropriate scaling). In the present work, we have formalized that idea of reducing the surface tension along an initial contour with the set of initial configurations IC defined at (\ref{IC}). The concept of energy barrier has lead to the definition of the gap in surface energy $\mathcal{K}^r (u_0)$ at (\ref{eq-def-Kr-disc}). The reader will note that we have introduced, in that concept of energy barrier, the requirement that evolution of the phase profile is (almost) continuous in the $L^1$ norm (with initial configuration $u_0$ and final configuration ${\mathbf 1}$, corresponding to uniform plus phase). Once these concepts are defined, remains a substantial work for relating the concepts and definitions to the physical phenomena. This is done as follows. First, in Sections~\ref{sec-covering} and \ref{sec-dilution}, we define a microscopic counterpart for the $(u_0,\tau^r)$ dilution, the event ${\mathcal G}(({\mathcal R}_i)_{i=1..n})$, which depends only on the interaction strength $J_e$. We prove that this event has the expected probability (the cost for dilution) and also we establish, conditionally on ${\mathcal G}(({\mathcal R}_i)_{i=1..n})$, upper and lower bounds on phase coexistence. In Section~\ref{sec-bottleneck} we introduce a certain bottleneck set, that corresponds with the phase profiles with the highest cost along all relaxation paths from $u_0$ to $\mathbf 1$. Then we express the probability, under the equilibrium measure, that the actual phase profile is in the bottleneck set in terms of its reduced surface energy, and in turn we relate the infimum of the reduced surface energy in the bottleneck set to the energy barrier $\mathcal{K}^r (u_0)$. In Section~\ref{sec:int:meta} we relate the former estimates on the equilibrium measure to the Glauber dynamics and establish an important result on the dynamics in a finite box (namely, Proposition \ref{prop-mt}). That Proposition states that, conditionally on the event of dilution, the average magnetization remains significantly different from the plus phase magnetization (i.e.~a droplet of minus phase remains),
with a probability corresponding to the reduced surface energy, until a time determined by the energy barrier. We conclude the proof of Theorem~\ref{thm-lwb-A} and Theorem~\ref{thm-Trel} in Section~\ref{sec-lwb-final}. For the proof of the first Theorem we relate the autocorrelation to the evolution of the magnetization in a finite volume. For the proof of the second Theorem, we decompose the box of interest into many smaller boxes, so that the dilution event must occur in one of these boxes with a probability close to one.

Finally, Section~\ref{sec:geom} is concerned with geometrical estimates. By decomposing the difference between the initial droplet and the current droplet into many small droplets, we are able to use the assumption that the initial droplet has a smooth boundary and prove that the energy barrier is strictly positive. We also prove that, under appropriate assumptions, the exponent $\mathcal{X}_\lambda$ is bounded by $C/\beta$ (Proposition~\ref{prop:gap}). Then in Section \ref{sec-contevol} we derive another formulation of the energy barrier, which shows that the energy barrier can be computed assuming not only the $L^1$ continuity of the droplet evolution with time, but also the ${\mathcal H}^{d-1}$ continuity of the part of the droplet surface that corresponds with the initial droplet contour (where dilution takes place). We establish that alternative and informative formulation by providing an interpolation between any two phase profiles that is continuous for both volume and surface measures, with the additional property that along the interpolation, the surface energy does not exceed the maximum of the surface energy among the two interpolated phase profiles. Finally, in Sections~\ref{sec-Kr-circle} and \ref{sec-Kr-square}, we use that alternative formulation to compute the energy barrier in specific cases. These computations put in evidence the fact that the energy barrier, as defined by $\mathcal{K}^r (u_0)$, is higher than the energy barrier computed with a constraint on the overall magnetization.

\subsection{Upper bound on the autocorrelation}

\label{sec:upb}Here we give the proof of Theorem~\ref{thm-upb-A}. As said
above, the scheme of proof was suggested by Fabio Martinelli. We call
\begin{eqnarray*}
  f^J_t (\sigma) & = & \left( T^J (t) \pi_0 \right) (\sigma)\\
  f^J_{\Lambda, t} (\sigma) & = & \left( T^{J, +}_{\Lambda} (t) \pi_0 \right)
  (\sigma)\\
  m^J & = & \mu^{J, +} (\sigma_0)
\end{eqnarray*}
therefore $f^J_t (\sigma)$ is the mean value of the spin at the origin under
the Glauber dynamics performed until time $t$, initiated with the
configuration $\sigma$, while $f^J_{\Lambda, t} (\sigma)$ is the corresponding
quantity for the dynamics restricted to $\Lambda$, with plus boundary
condition on $\Lambda^c$. Using these notations we can write the averaged
autocorrelation $A^{\lambda} (t)$ defined at (\ref{def-A}) as
\begin{eqnarray}
  A^{\lambda} (t) & = & \mathbbm{E} \left( \left[ \tmop{Var}_{\mu^{J, +}}
  (f^J_t) \right]^{\lambda} \right) .  \label{def-A2}
\end{eqnarray}
\begin{lemma}
  \label{lemF1}One has
  \begin{eqnarray*}
    \frac{1}{2} \tmop{Var}_{\mu^{J, +}} \left( f^J_t \right) & \leqslant &
    f_t^J (+) - m^J
  \end{eqnarray*}
  where $+$ means the constant plus spin configuration.
\end{lemma}

\begin{proof}
  The attractivity of the dynamics implies that $f_t^J (\sigma)
  \leqslant f_t^J (+)$ for any $\sigma$. Then:
  \begin{eqnarray*}
    \frac{1}{2} \tmop{Var}_{\mu^{J, +}} (f^J_t) & = & \frac{1}{4} \int \int
    \left( f^J_t \left( \sigma \right) - f^J_t \left( \eta \right) \right)^2
    \mathd \mu^{J, +} \left( \sigma \right) \mathd \mu^{J, +} \left( \eta
    \right)\\
    & = & \frac{1}{2} \int \int_{f^J_t \left( \sigma \right) \geqslant f^J_t
    \left( \eta \right)_{}} \left( f^J_t \left( \sigma \right) - f^J_t \left(
    \eta \right) \right)^2 \mathd \mu^{J, +} \left( \sigma \right) \mathd
    \mu^{J, +} \left( \eta \right)\\
    & \underset{(x^2 \leqslant 2 x, \forall x \in [0, 2])}{\leqslant} & \int
    \int_{f^J_t \left( \sigma \right) \geqslant f^J_t \left( \eta \right)_{}}
    \left( f^J_t \left( \sigma \right) - f^J_t \left( \eta \right) \right)
    \mathd \mu^{J, +} \left( \sigma \right) \mathd \mu^{J, +} \left( \eta
    \right)\\
    & \underset{\tmop{attractivity}}{\leqslant} & \int_{_{}} \left( f^J_t
    \left( + \right) - f^J_t \left( \eta \right) \right) \mathd \mu^{J, +}
    \left( \eta \right) .
  \end{eqnarray*}
  And the last term equals $f_t^J (+) - m^J$.
\end{proof}

We apply then standard controls on the spectral gap:

\begin{lemma}
  \label{lemF2}There exist positive and
  finite constants $C_1, C_2, C_3$ such that, for all $J \in {\mathcal J}$, all $t>0$ and $\beta>0$,
  \begin{eqnarray*}
    f_{\hat{\Lambda}_N, t}^J (+) - f_{\hat{\Lambda}_N, t}^J (-) & \leqslant &
    e^{C_1 \beta N^d - C_3 t N^{- d} e^{- C_2 \beta N^{d - 1}}} .
  \end{eqnarray*}
\end{lemma}

\begin{proof}
  The difference $f_{\hat{\Lambda}_N, t}^J (+) - f_{\hat{\Lambda}_N, t}^J (-)$
  is not larger than twice the $\|.\|_{\infty}$ norm of $f_{\hat{\Lambda}_N,
  t}^J$, which according to (3.15) in {\cite{Mar99LNM}} does not exceed
  \begin{eqnarray*}
    \left\| f_{\hat{\Lambda}_N, t}^J \right\|_{\infty} & \leqslant & \left[
    \inf_{\sigma} \mu_{\hat{\Lambda}_N}^{J, +} (\sigma) \right]^{- 1 / 2} \exp
    \left( - t \tmop{gap} (L^{J, \hat{\Lambda}_N, +}) \right) \left\| \pi_0
    \right\|_{L^2 \left( \mu_{\hat{\Lambda}_N}^{J, +} \right)} .
  \end{eqnarray*}
  The conclusion comes then from the general lower bound
  \begin{eqnarray}
    \tmop{gap} (L^{J, \hat{\Lambda}_N, +}) & \geqslant & c_m (2 N + 1)^{- d}
    e^{- C_2 \beta N^{d - 1}}, \label{eq:Ma:lgap}
  \end{eqnarray}
  cf.~Theorem~6.4 in {\cite{Mar99LNM}}.
\end{proof}

An easy consequence of monotonicity is the next Lemma:

\begin{lemma}
  \label{lemF3}For any $J \in {\mathcal J}$, one has
  \begin{eqnarray*}
    f_{\Lambda, t}^J (-) - m^J & \leqslant & \mu^{J, +}_{\Lambda} (\sigma_0) -
    m^J\\
    & \leqslant & 2 \mu^{J, +} \left( \mathcal{C}_{\Lambda}^c \right)
  \end{eqnarray*}
  where $\mathcal{C}_{\Lambda}$ is the event that there exists a contour of
  plus spins in$\text{ } \Lambda$, around the origin.
\end{lemma}

\begin{proof}
  The inequality $f_{\Lambda, t}^J (-) \leqslant \mu^{J, +}_{\Lambda}
  (\sigma_0)$ is immediate. Then we remark that, conditionally on
  $\mathcal{C}_{\Lambda}$, the expectation of $\sigma_0$ under $\mu^J$ is at
  least $\mu^{J, +}_{\Lambda} \left( \sigma_0 \right)$. Otherwise it is at
  least $- 1$. Hence
  \begin{eqnarray*}
    m^J & \geqslant & \mu^{J, +} \left( \mathcal{C}_{\Lambda}^{} \right)
    \times \mu^{J, +}_{\Lambda} \left( \sigma_0 \right) + \mu^{J, +} \left(
    \mathcal{C}_{\Lambda}^c \right) \times (- 1)\\
    & \geqslant & \mu^{J, +}_{\Lambda} \left( \sigma_0 \right) - 2 \mu^{J, +}
    \left( \mathcal{C}_{\Lambda}^c \right) .
  \end{eqnarray*}
\end{proof}

The average of $\mu^{J, +} \left( \mathcal{C}_{\Lambda}^c \right)$ decays
exponentially fast with the size of $\Lambda$ under the slab percolation
assumption when $d = 2$ (see (\ref{eq-betasp}) for the definition of
$\hat{\beta}_c$). When $\beta$ is large, we can also give quantitative
estimates.

\begin{lemma}
  \label{lemF4}Assume $d = 2$ and $\beta > \hat{\beta}_c$, then there is $c >
  0$ such that, for $N$ large enough,
  \begin{eqnarray}
    \mathbbm{E} \mu^{J, +} \left( \mathcal{C}_{\hat{\Lambda}_N}^c \right) &
    \leqslant & \exp (- cN) .  \label{muC1}
  \end{eqnarray}
\end{lemma}

\begin{proof}
  We use the FK representation of the spin model and the renormalization
  framework of {\cite{Wou08SPA}}. The event $\mathcal{C}_{\hat{\Lambda}_N}$ is
  realized with conditional probability one if both of the following occur in
  the edge configuration:
  \begin{enumerate}[i.]
    \item There is an infinite $\omega$-open path issued from
    $\hat{\Lambda}_{N / 2}$
    
    \item There is a surface that lies in $\text{$\hat{\Lambda}_N$} \setminus
    \hat{\Lambda}_{N / 2}$ for which all points are $\omega$-connected.
  \end{enumerate}
  Since $d = 2$, the second point reduces to finding a circuit of open edges
  inside $ \text{$\hat{\Lambda}_N$} \setminus \hat{\Lambda}_{N / 2}$. We can
  cover $ \text{$\hat{\Lambda}_N$} \setminus \hat{\Lambda}_{N / 2}$ with $16$
  blocks of side-length $N / 2$. When all these blocks are good in the sense
  of Theorem 2.1 in {\cite{Wou08SPA}}, the second point is realized, and this
  occurs with a probability $1 - \exp (- cN)$ as we assume that $\beta >
  \hat{\beta}_c$. Similarly, the first point is realized as well with a
  probability $1 - \exp (- cN)$ and (\ref{muC1}) follows.
\end{proof}

\begin{remark}
  In larger dimensions, under the assumption that $\mathbbm{P}(J_e \geqslant
  \varepsilon) = 1$ for some $\varepsilon > 0$, one can use Peierls estimates
  to prove that $\mathbbm{E} \mu^{J, +} \left( \mathcal{C}_{\hat{\Lambda}_N}^c
  \right) \leqslant \exp (- c \beta N)$ for some $c > 0$, for any $\beta$
  large. Still, this is not useful for generalizing Theorem~\ref{thm-upb-A} as can be
  seen in the final optimization below.
\end{remark}

Now we give the proof of Theorem~\ref{thm-upb-A}.

\begin{proof}
  According to Lemmas~\ref{lemF1}, \ref{lemF2}, \ref{lemF3} and to the
  inequality $f_t^J (+) \leqslant f_{\hat{\Lambda}_N, t}^J (+)$ due to the
  monotonicity of the dynamics, we have
  \begin{eqnarray*}
    \tmop{Var}_{\mu^{J, +}} \left( f^J_t \right) & \leqslant & 2 e^{C_1 \beta
    N^d - C_3 t N^{- d} e^{- C_2 \beta N^{d - 1}}} + 4 \mu^{J \text{}, +} \left(
    \mathcal{C}_{\hat{\Lambda}_N}^c \right) .
  \end{eqnarray*}
  We recall the assumption that $d = 2$ and $\beta > \hat{\beta}_c$. According
  to Lemma~\ref{lemF4} there is $c > 0$ such that, for all $N$ large enough:
  \begin{eqnarray*}
    \mathbbm{E} \mu^{J, +} \left( \mathcal{C}_{\hat{\Lambda}_N}^c \right) &
    \leqslant & \exp (- cN),
  \end{eqnarray*}
  hence Markov's inequality implies that
  \begin{eqnarray}
    \mathbbm{P}\left( \mu^{J, +} \left( \mathcal{C}_{\hat{\Lambda}_N}^c \right)
    \geqslant \exp (- c N/ 2) \right) & \leqslant & \exp (- c N/2) . 
    \label{Cmarkov}
  \end{eqnarray}
  Now we define $N$ by the relation
  \begin{eqnarray}
    C_2 \beta N^{d - 1} & = & (1 - \delta) \log t.  \label{Nt}
  \end{eqnarray}
  For large enough $t$, we have
  \begin{eqnarray*}
    \tmop{Var}_{\mu^{J, +}} \left( f^J_t \right) & \leqslant & 4 \left( \exp
    (- c N / 2) + \mu^{J \text{}, +} \left( \mathcal{C}_{\hat{\Lambda}_N}^c
    \right) \right) .
  \end{eqnarray*}
  Combining this with (\ref{def-A2}) and (\ref{Cmarkov}) gives
  \begin{eqnarray*}
    A^{\lambda} (t) & \leqslant & 8^{\lambda} \left[ \exp (- c \lambda N / 2)
    + \exp (- c N / 2) \right]
  \end{eqnarray*}
  which, according to the assumption that $d = 2$ and to the definition
  (\ref{Nt}) of $N$, establishes the claim (\ref{Aupb1}).
\end{proof}

\subsection{Lower bound on the autocorrelation}

\label{sec-autocor}The subject of the present Section is the proof of
Theorems~\ref{thm-lwb-A} and \ref{thm-Trel}. It is organized as follows. In
Section~\ref{sec-covering}, we define the notion of covering of the border
of a magnetization profile by rectangles. Then, in Section~\ref{sec-dilution}
we define the event of dilution according to some initial profile $(u_0,
\tau^r) \in \tmop{IC}$ and prove that it has the expected probability. We show
how previous result from {\cite{Wou09CMP}} apply for the probability of phase
coexistence, given the event of dilution. In Section~\ref{sec-bottleneck} we
show that phase profiles evolve continuously in $L^1$ and relate this property
to the bottleneck and to the gap in free energy $\mathcal{K}^r (u_0)$. Finally
in~\ref{sec-lwb-final} we conclude the proof of Theorems~\ref{thm-lwb-A} and
\ref{thm-Trel}.

\subsubsection{Covering of the boundary of phase profiles}

\label{sec-covering}As in the work {\cite{Wou09CMP}}, the coverings of the
boundary of macroscopic phase profile are a fundamental tool 
for relating the macroscopic shape of the magnetization to the microscopic
spin system. The definition that we present here is more restrictive than in
{\cite{Wou09CMP}} since it takes into account the set $\tmop{IC}$, together
with a new parameter $\gamma$ (fourth line in (v) below).

\begin{definition}
  \label{def-g-adapted}Let $(u_0 = \chi_{U_0}, \tau^r) \in \tmop{IC}$ and $u
  \in \tmop{BV}$, together with $\delta, \gamma > 0$. We say that a
  rectangular parallelepiped $\mathcal{R} \subset [0, 1]^d$ is $(u_0, \tau^r,
  \delta, \gamma)$-adapted to $u$ at $x \in \partial^{\star} u$ if:
  \begin{enumerate}[i.]
    \item If $\tmmathbf{n}=\tmmathbf{n}_x^u$ is the outer normal to $u$ at
    $x$, there are $\mathcal{S} \in \mathbbm{S}_{\tmmathbf{n}}$ and $h \in (0,
    \delta]$ such that, either $\mathcal{R} \subset (0, 1)^d$ (we say that
    $\mathcal{R}$ is interior) and
    \[ \mathcal{R}= x + h\mathcal{S}+ [\pm \delta h]\tmmathbf{n}, \]
    either $\mathcal{R} \cap \partial [0, 1]^d \neq \emptyset$ (we say that
    $\mathcal{R}$ is on the border) and $x \in \partial [0, 1]^d$,
    $\tmmathbf{n}$ is also the outer normal to $[0, 1]^d$ at $x$ and
    \[ \mathcal{R}= x + h\mathcal{S}+ [- \delta h, 0]\tmmathbf{n}. \]
    \item We have
    \[ \mathcal{H}^{d - 1} \left( \partial^{\star} u \cap \partial \mathcal{R}
       \right) = 0, \]
    \[ \left| 1 - \frac{1}{h^{d - 1}} \mathcal{H}^{d - 1} \left(
       \partial^{\star} u \cap \mathcal{R} \right) \right| \leqslant \delta,
    \]
    and
    \[ \left| \tau^q (\tmmathbf{n}) - \frac{1}{h^{d - 1}} \int_{\partial^{\star}
       u \cap \mathcal{R}} \tau^q (\tmmathbf{n}_.^u) d\mathcal{H}^{d - 1}
       \right| \leqslant \delta . \]
    \item If $\chi: \mathbbm{R}^d \rightarrow \{\pm 1\}$ is the
    characteristic function of the half-space above the center of $ \mathcal{R}$, namely
    \[ \chi (z) = \left\{ \begin{array}{ll}
         + 1 & \text{if } (z - x) \cdot \tmmathbf{n} \geqslant 0\\
         - 1 & \text{else}
       \end{array}, \forall z \in \mathbbm{R}^d, \right. \]
    then
    \[ \frac{1}{2 \delta h^d} \int_{\mathcal{R}} | \chi - u|d\mathcal{H}^d
       \leqslant \delta . \]
    \item If $\mathcal{R} \cap \partial [0, 1]^d \neq \emptyset$, then
    $\mathcal{R}$ does not intersect $\partial U_0$.
    
    \item If $\mathcal{R}$ intersects $\partial U_0$, then $x \in \partial
    U_0$ and
    \[ \left| \frac{1}{h^{d - 1}} \mathcal{H}^{d - 1} \left( \left(
       \partial^{\star} u \Delta \partial^{\star} u_0 \right) \cap \mathcal{R}
       \right) \right| \leqslant \delta, \]
    \[ \left| \tau^r (x) - \frac{1}{h^{d - 1}} \int_{\partial^{\star} u_0 \cap
       \mathcal{R}} \tau^r d\mathcal{H}^{d - 1} \right| \leqslant \delta, \]
    \[ \left| I_{\tmmathbf{n}} (\tau^r (x)) - \frac{1}{h^{d - 1}}
       \int_{\partial^{\star} u_0 \cap \mathcal{R}} I_{\tmmathbf{n}_z^{u_0}}
       (\tau^r (z)) d\mathcal{H}^{d - 1} (z) \right| \leqslant \delta, \]
    and
    \[ \left| I_{\tmmathbf{n}} (\tau^r (x) - \gamma) - \frac{1}{h^{d - 1}}
       \int_{\partial^{\star} u_0 \cap \mathcal{R}} I_{\tmmathbf{n}_z^{u_0}}
       (\tau^r (z) - \gamma) d\mathcal{H}^{d - 1} (z) \right| \leqslant \delta
       . \]
    \item If $\mathcal{R}$ intersects $\partial U_0$, then the enlarged volume
    \[ \mathcal{R}' =\mathcal{R}+ B (0, 2 \sqrt{d} h^2) = \left\{ z \in
       \mathbbm{R}^d: d (z, \mathcal{R}) \leqslant 2 \sqrt{d} h^2 \right\} \]
    satisfies
    \begin{equation} \frac{1}{h^{d - 1}} \mathcal{H}^{d - 1} \left( \partial^{\star} u_0
       \cap \mathcal{R}' \setminus \mathcal{R} \right) \leqslant \delta . \label{eq:enlarged}
       \end{equation}
  \end{enumerate}
\end{definition}

Note that conditions {\tmem{i}} to {\tmem{iii}} above mean exactly that
$\mathcal{R}$ is $\delta$-adapted to $\partial^{\star} u$ at $x \in
\partial^{\star} u$ in the sense of Definition 3.1 of {\cite{Wou09CMP}}. Also, by 
 $\Delta$ we mean the symmetric difference.

\begin{definition}
  \label{def-g-cover}Let $(u_0, \tau^r) \in \tmop{IC}$ and $u \in \tmop{BV}$,
  $\delta, \gamma > 0$. A finite sequence $(\mathcal{R}_i)_{i = 1 \ldots n}$
  of disjoint rectangular parallelepipeds included in $[0, 1]^d$ is said to be
  a $(u_0, \tau^r, \delta, \gamma)$-covering for $\partial^{\star} u$ if each
  $\mathcal{R}_i$ is $(u_0, \tau^r, \delta, \gamma)$--adapted to $u$ and if
  \begin{equation}
    \mathcal{H}^{d - 1} \left( \partial^{\star} u \setminus \bigcup_{i = 1}^n
    \mathcal{R}_i \right) \leqslant \delta .
  \end{equation}
\end{definition}

The main result of this Section is:

\begin{proposition}
  \label{prop-g-cover}Let $(u_0, \tau^r) \in \tmop{IC}$ and $u \in \tmop{BV}$,
  together with $\gamma, \delta > 0$. There exists a $(u_0, \tau^r, \delta,
  \gamma)$-covering for $\partial^{\star} u$.
\end{proposition}

\begin{proof}
As the proof follows a classical argument, we will only give here the main steps in the proof.
First, we claim that, for any $u, u_0 \in \tmop{BV}$, for $\mathcal{H}^{d -
  1}$ almost all $x \in \partial^{\star} u \cap \partial^{\star} u_0$,
  \[ \lim_{r \rightarrow 0^+} \frac{1}{r^{d - 1}} \mathcal{H}^{d - 1} \left(
    ( \partial^{\star} u \Delta \partial^{\star} u_0 ) \cap B(x,r)\right) = 0. \]
This can be proven by applying, for instance, the Besicovitch derivation Theorem (Theorem 2.22 in {\cite{AFP00Book}}) to the Borel measurable function
  \[ f: x \in \mathbbm{R}^d \mapsto \left\{ \begin{array}{ll}
       1 & \text{if } x \in \partial^{\star} u_0\\
       2 & \text{else},
     \end{array} \right. \]
     with the result that for $\mathcal{H}^{d - 1}$-almost all $x \in \partial^{\star} u \cap
  \partial^{\star} u_0$,
  \[ \lim_{r \rightarrow 0^+} \frac{1}{\alpha_{d - 1} r^{d - 1}}
     \int_{\partial^{\star} u \cap B (x, r)} f d\mathcal{H}^{d - 1} = f (x) \]
  where $\alpha_{d - 1} =\mathcal{H}^{d - 1} \left( \left\{ x \in B (0, 1): x
  \cdot \tmmathbf{e}_d = 0\} \right) \right.$. Therefore,
  \[ \lim_{r \rightarrow 0^+} \left[ \frac{\mathcal{H}^{d - 1} \left(
     \partial^{\star} u \cap B (x, r) \right)}{\alpha_{d - 1} r^{d - 1}} +
     \frac{\mathcal{H}^{d - 1} \left( (\partial^{\star} u \setminus
     \partial^{\star} u_0) \cap B (x, r) \right)}{\alpha_{d - 1} r^{d - 1}}
     \right] = 1. \]
  As the first term goes to $1$ already as $r \rightarrow 0$, we conclude that
  \[ \lim_{r \rightarrow 0^+} \frac{1}{\alpha_{d - 1} r^{d - 1}}
     \mathcal{H}^{d - 1} \left(( \partial^{\star} u \setminus \partial^{\star}
     u_0) \cap B (x, r) \right) = 0 \]
  for $\mathcal{H}^{d - 1}$-almost all $x \in \partial^{\star} u \cap
  \partial^{\star} u_0$. The claim follows as $u$ and $u_0$ play a symmetric
  role.
  
Second, we remark that more generally, all the absolute values in point (v) of Definition \ref{def-g-cover}, and also left-hand side of inequality (\ref{eq:enlarged}), have a zero limit as  $h \rightarrow 0$ for $\mathcal{H}^{d - 1}$-almost all $x \in \partial^{\star} u \cap  \partial U_0$. This is a consequence of the strong form of the Besicovitch derivation theorem (Theorem 5.52 in {\cite{AFP00Book}}).

The two above facts enable, like in the proof of Theorem~3.4 in {\cite{Wou09CMP}}, the use of the Vitali covering Theorem (cf.~Definition~13.2 and Theorem~13.3 in {\cite{Cer06LNM}}) to conclude to the existence of the desired covering.
\end{proof}

\subsubsection{The event of dilution}

\label{sec-dilution}The purpose of this Section is to define the event of
(microscopic) dilution and to prove several important properties about phase
coexistence, given the event of dilution. We recall a notation from \cite{Wou09CMP}: when $\mathcal{R}$ is a rectangle and $N>0$, we let
\[ \mathcal{R}^N = N \mathcal{R} + z_N(\mathcal{R}) \]
where $z_N(\mathcal{R}) \in (-1/2,1/2]^d$ is such that the center of $\mathcal{R}^N$ belongs to $\mathbb{Z}^d$.

\begin{definition}
  \label{def-dilution}Let $(u_0, \tau^r) \in \tmop{IC}$, $\delta, \gamma > 0$
  and let $(\mathcal{R}_i)_{i = 1 \ldots n}$ be a $(u_0, \tau^r, \delta,
  \gamma)$-covering of $\partial^{\star} u_0$. The event of dilution on this
  covering is
  \begin{equation}
    \mathcal{G} \left( (\mathcal{R}_i)_{i = 1 \ldots n} \right) = \left\{ J
    \in \mathcal{J}: \tau^J_{\mathcal{R}^N_i} \leqslant \tau^r (x_i), \forall
    i = 1 \ldots n \right\} \label{eq-def-GN}.
  \end{equation}
\end{definition}

The reader will check that the event of dilution affects
only the random variables $J_e$ with $e$ at distance at most $N \delta$ from
the boundary $N \partial^{\star} u_0$.

The first property of the event of dilution is that it happens at the expected
rate. Namely,

\begin{lemma}
  \label{lem-proba-GN}Let $(u_0, \tau^r) \in \tmop{IC}$ and $\xi > 0$. For
  $\delta > 0$ small enough, for $\gamma > 0$ arbitrary, if
  $(\mathcal{R}_i)_{i = 1 \ldots n}$ is a $(u_0, \tau^r, \delta,
  \gamma)$-covering of $\partial^{\star} u_0$, then
  \begin{eqnarray*}
    \mathbbm{P} \left( \mathcal{G} \left( (\mathcal{R}_i)_{i = 1 \ldots n}
    \right) \right) & \geqslant & \exp \left( - N^{d - 1}  \left(
    \mathcal{I}^r (u_0) + \xi \right) \right)
  \end{eqnarray*}
  for any $N$ large enough.
\end{lemma}

\begin{proof}
  We recall that $\mathcal{I}^r (u_0)$ was defined at (\ref{eq-def-Ir}), see
  also (\ref{def-In}). Since the $\mathcal{R}^N_i$ are disjoint for large
  enough $N$, we have
  \begin{eqnarray*}
    \liminf_{N \rightarrow \infty} \frac{1}{N^{d - 1}} \log \mathbbm{P} \left(
    \mathcal{G} \left( (\mathcal{R}_i)_{i = 1 \ldots n} \right) \right) & = &
    \liminf_{N \rightarrow \infty} \sum_{i = 1}^n \frac{1}{N^{d - 1}} \log
    \mathbbm{P} \left( \tau^J_{\mathcal{R}^N_i} \leqslant \tau^r (x_i)
    \right)\\
    & = & - \sum_{i = 1}^n h_i^{d - 1} I_{\tmmathbf{n}_i} (\tau^r (x_i))
  \end{eqnarray*}
  in view of (\ref{def-In}). Note the role played by the assumption $\tau^r
  (x) > \tau^{\min} (\tmmathbf{n}_x^{u_0})$ in the definition of $\tmop{IC}$
  at (\ref{IC}). The properties of the covering (Definition \ref{def-g-cover}
  and point (iii) in Definition \ref{def-g-adapted}) imply the claim for
  $\delta > 0$ small enough.
\end{proof}

Then we show that the dilution has the expected impact on the probability for
phase coexistence. This is expressed with the two complementary
Propositions~\ref{prop-cond-lwb-phco} and \ref{prop-cond-upb-phco}. We have to recall here 
the definition of the magnetization profile
\begin{equation}
  \begin{array}{cccc}
    \mathcal{M}_K : & [0, 1]^d & \longrightarrow & [- 1, 1]\\
    & x & \longmapsto & \frac{1}{K^d} \sum_{z \in \Lambda_N \cap \Delta_{i
    (x)}} \sigma_z
  \end{array} \label{eq-def-MK}
\end{equation}
where  $K \in \mathbbm{N}^{\star}$ is the mesoscopic scale and
\begin{equation}
  i (x) = \left( \left[ \frac{Nx_1}{K} \right], \ldots, \left[ \frac{Nx_d}{K}
  \right] \right) \text{ \ and \ } \Delta_i = Ki +\{1, \ldots, K\}^d .
\end{equation}
Hence, unless $x$ is too close to the border of $[0, 1]^d$, $\mathcal{M}_K
(x)$ is the magnetization in a block of side-length $K$ that contains $Nx$.
Theorem 5.7 in {\cite{Wou08SPA}} provides a strong stochastic control on
$\mathcal{M}_K$ when $\beta > \hat{\beta}_c$. In particular, when $K$ is large
enough, at every $x$ the probability that $\mathcal{M}_K (x)$ is close to
either $m_{\beta}$ or $- m_{\beta}$ is close to one under the averaged measure
$\mathbbm{E} \mu^{J, +}_{\Lambda_N}$. The event $ \mathcal{M}_K/m_{\beta} \in
     \mathcal{V}(u_0, \varepsilon)$ means therefore that the system is close to plus (resp. minus) phase at $N x$ when $u_0(x)=1$ (resp. $u_0(x)=-1$).
\begin{proposition}
  \label{prop-cond-lwb-phco}Assume that $\beta > \hat{\beta}_c$ and $\beta
  \notin \mathcal{N}$. Let $(u_0, \tau^r) \in \tmop{IC}$ and $\varepsilon, \xi
  > 0$. Let $\delta, \gamma > 0$ and consider $(\mathcal{R}_i)_{i = 1 \ldots
  n}$ a $(u_0, \tau^r, \delta, \gamma)$-covering of $\partial^{\star} u_0$.
  Then, if $\delta > 0$ is small enough (and $\gamma$ arbitrary), for $K$
  large enough,
  \[ \lim_{N \rightarrow \infty} \mathbbm{P} \left( \left. \mu^{J,
     +}_{\Lambda_N} \left( \frac{\mathcal{M}_K}{m_{\beta}} \in
     \mathcal{V}(u_0, \varepsilon) \right) \geqslant \exp \left( - N^{d - 1}
     \left( \mathcal{F}^r (u_0) + \xi \right) \right) \right| \mathcal{G}
     \left( (\mathcal{R}_i)_{i = 1 \ldots n} \right) \right) = 1. \]
\end{proposition}

\begin{proof}
  We use the notations of Proposition~3.9 in {\cite{Wou09CMP}}. We recall that
  $\mathcal{D}_{U_0}^{N, \delta}$ is the event of $\omega$-disconnection
  around $N \partial^{\star} u_0$ and that $\mathcal{E}_{U_0}^{N, \delta}$ is
  the set of edges close to $N \partial^{\star} u_0$. We let then
  \[ F^J_N = \inf_{\pi \in \mathcal{D}_{U_0}^{N, \delta}} \Psi^{J, w,
     +}_{\Lambda_N} \left( \left. \frac{\mathcal{M}_K}{m_{\beta}} \in
     \mathcal{V}(u_0, \varepsilon) \right| \omega = \pi \text{ on }
     \mathcal{E}_{U_0}^{N, \delta} \right) . \]
  Proposition~3.9 in {\cite{Wou09CMP}} and the definition of the covering
  imply
  \begin{enumerate}[i.]
    \item For $\delta > 0$ small enough,
    \begin{equation}
      \liminf_{N \rightarrow \infty} \inf_{J \in \mathcal{G} \left(
      (\mathcal{R}_i)_{i = 1 \ldots n} \right)} \frac{1}{N^{d - 1}} \log
      \Phi^{J, w}_{\Lambda_N} \left( \mathcal{D}_{U_0}^{N, \delta} \right)
      \geqslant -\mathcal{F}^r (u_0) - \xi / 2. \label{eq-lwb-pDU0}
    \end{equation}
    \item For $\delta > 0$ small enough, for $K$ large enough,
    \begin{equation}
      \lim_{N \rightarrow \infty} \mathbbm{P} \left( F^J_N < \frac{1}{3}
      \right) = 0. \label{eq-meanF}
    \end{equation}
  \end{enumerate}
  The definition of $F^J_N$ yields
  \[ \mu^{J, +}_{\Lambda_N} \left( \frac{\mathcal{M}_K}{m_{\beta}} \in
     \mathcal{V}(u_0, \varepsilon) \right) \geqslant F^J_N \Phi^{J,
     w}_{\Lambda_N} \left( \mathcal{D}_{U_0}^{N, \delta} \right) \]
  hence, using (\ref{eq-lwb-pDU0}) we obtain: for large enough $N$,
  \[ \begin{array}{l}
       \mathbbm{P} \left( \left. \mu^{J, +}_{\Lambda_N} \left(
       \frac{\mathcal{M}_K}{m_{\beta}} \in \mathcal{V}(u_0, \varepsilon)
       \right) \geqslant \exp \left( - N^{d - 1} \left( \mathcal{F}^r (u_0) +
       \xi \right) \right) \right| \mathcal{G} \left( (\mathcal{R}_i)_{i = 1
       \ldots n} \right) \right)\\
       \hspace{6cm} \geqslant \mathbbm{P} \left( \left. F^J_N \geqslant
       \frac{1}{3} \right| \mathcal{G} \left( (\mathcal{R}_i)_{i = 1 \ldots n}
       \right) \right)
     \end{array} \]
  Yet, the variable $F^J_N$ is \tmtextit{independent} of the $J_e$ with $e \in
  \mathcal{E}_{U_0}^{N, \delta}$. Thus it is as well independent of the
  dilution $\mathcal{G} \left( (\mathcal{R}_i)_{i = 1 \ldots n} \right)$, and
  (\ref{eq-meanF}) yields the conclusion.
\end{proof}

We also recall from \cite{Wou09CMP} the definition of the compact set
\begin{equation}
  \tmop{BV}_a = \left\{ u = \chi_U \in \tmop{BV} : \mathcal{P}(U) \leqslant a
  \right\} \label{eq-def-BVa}
\end{equation} where $\mathcal{P}(U)$ is the perimeter of $U$.

\begin{proposition}
  \label{prop-cond-upb-phco}Assume $\beta > \hat{\beta}_c$ and $\beta \notin
  \mathcal{N} \cup \mathcal{N}_I$. Let $(u_0, \tau^r) \in \tmop{IC}$, $a > 0$,
  $\xi > 0$. There exists $\gamma > 0$ such that, for any $u \in \tmop{BV}_a$
  there is $\varepsilon > 0$ such that, for any $\delta_0 > 0$ small enough,
  for any $(u_0, \tau^r, \delta_0, \gamma)$-covering $(\mathcal{R}_i)_{i = 1
  \ldots n}$ of $\partial^{\star} u_0$, for any $K$ large enough,
  \[ \lim_{N \rightarrow \infty} \mathbbm{P} \left( \left. \mu^{J,
     +}_{\Lambda_N} \left( \frac{\mathcal{M}_K}{m_{\beta}} \in \mathcal{V}(u,
     \varepsilon) \right) \leqslant \exp \left( - N^{d - 1} \left(
     \mathcal{F}^r (u) - \xi \right) \right) \right| \mathcal{G} \left(
     (\mathcal{R}_i)_{i = 1 \ldots n} \right) \right) = 1. \]
\end{proposition}

The proof of Proposition \ref{prop-cond-upb-phco} is based on Proposition~3.11
in {\cite{Wou09CMP}} that relates the probability $\left. \mu^{J,
+}_{\Lambda_N} (\mathcal{M}_K / m_{\beta} \in \mathcal{V}(u, \varepsilon))
\right.$ to the $L^1$-notion of surface tension
\begin{equation}
  \tilde{\tau}^{J, \delta, K}_{N\mathcal{R}} = - \frac{1}{(h N)^{d - 1}} \log
  \sup_{\bar{\sigma} \in \Sigma^+_{\widehat{N\mathcal{R}}}} \mu^{J,
  \bar{\sigma}}_{\widehat{N\mathcal{R}}} \left( \left\|
  \frac{\mathcal{M}_K}{m_{\beta}} - \chi \right\|_{L^1 \left( \mathcal{R}
  \right)} \leqslant 2 \delta \mathcal{L}^d \left( \mathcal{R} \right) \right)
  \label{eq-def-tauJL1}
\end{equation}
together with Propositions~3.12 and 3.13 in {\cite{Wou09CMP}} that compare the
two definitions of surface tensions $\tilde{\tau}^{J, \delta,
K}_{N\mathcal{R}}$ and $\tau^J_{N\mathcal{R}}$. Before we complete the proof
of \ Proposition \ref{prop-cond-upb-phco}, we examine the typical value of
$\tilde{\tau}^{J, \delta, K}_{N\mathcal{R}}$ given the event of dilution.

\begin{lemma}
  \label{lem-cond-lwb-tauJt}Assume $\beta > \hat{\beta}_c$ and $\beta \notin
  \mathcal{N} \cup \mathcal{N}_I$. Let $(u_0, \tau^r) \in \tmop{IC}$ and
  $\gamma > 0$. For $\delta > 0$ small enough, the following holds: for any $u
  \in \tmop{BV}$ and any $\mathcal{R}$ that is $(u_0, \tau^r, \delta,
  \gamma)$-adapted to $\partial^{\star} u$ at $x \in \partial^{\star} u \cap
  \partial^{\star} u_0$, for any $(u_0, \tau^r, \delta_0, \gamma)$-covering
  $(\mathcal{R}_i)_{i = 1 \ldots n}$ of $\partial^{\star} u_0$ such that
  $\delta_0 \in (0, h^2)$, for any $K$ large enough, then
  \[ \lim_{N \rightarrow \infty} \frac{1}{N^{d - 1}} \log \mathbbm{P} \left(
     \left. \tilde{\tau}^{J, \delta, K}_{N\mathcal{R}} < \tau^r (x) - c'_{d,
     \delta} - \gamma \right| \mathcal{G} \left( (\mathcal{R}_i)_{i = 1 \ldots
     n} \right) \right) < 0 \]
  where $c'_{d, \delta} = c_d \delta + c_{d, \delta}$ is the sum of the
  constants that appear in Propositions~3.12 and 3.13 in {\cite{Wou09CMP}},
  and is arbitrary small as $\delta \rightarrow 0$.
\end{lemma}

Note that $h$ in the statement of Lemma~\ref{lem-cond-lwb-tauJt} refers to the
largest dimension of $\mathcal{R}$, see (i) in Definition~\ref{def-g-adapted}.

\begin{figure}[h!]
\begin{center}\includegraphics[width=10cm]{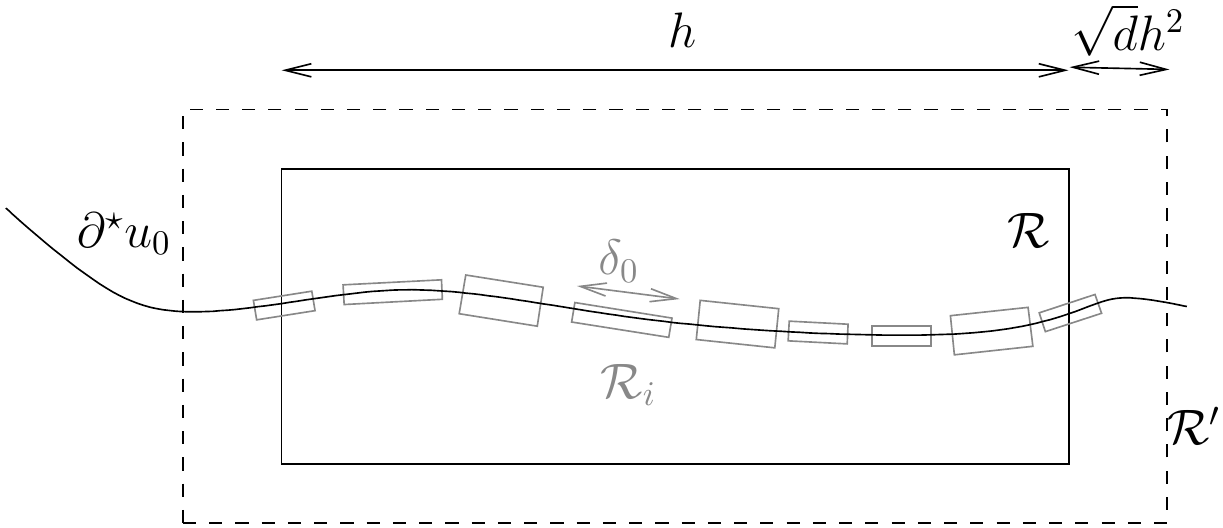}\end{center}
  \caption{\label{fig-dilution}The scale of dilution $\delta_0$.}
\end{figure}

\begin{proof}
  (Lemma \ref{lem-cond-lwb-tauJt}). We consider $(\mathcal{R}_i)_{i = 1 \ldots
  n}$ a $(u_0, \tau^r, \delta_0, \gamma)$-covering for $\partial^{\star} u_0$.
  Thanks to the product structure of $\mathbbm{P}$, for large enough $N$ the
  surface tension $\tilde{\tau}^{J, \delta, K}_{\mathcal{R}^N}$ is independent
  of the $\tau^J_{\mathcal{R}^N_i}$ such that $\mathcal{R}_i \cap \mathcal{R}=
  \emptyset$. Hence, for large enough $N$, the conditional probability
  \[ \mathbbm{P} \left( \left. \tilde{\tau}^{J, \delta, K}_{N\mathcal{R}} <
     \tau^r (x) - c'_{d, \delta} - \gamma \right| \mathcal{G} \left(
     (\mathcal{R}_i)_{i = 1 \ldots n} \right) \right) \]
  equals
  \[ \mathbbm{P} \left( \left. \tilde{\tau}^{J, \delta, K}_{N\mathcal{R}} <
     \tau^r (x) - c'_{d, \delta} - \gamma \right| \tau^J_{\mathcal{R}^N_i}
     \leqslant \tau^r (x_i), \forall i: \mathcal{R}_i \cap \mathcal{R} \neq
     \emptyset \right), \]
  which is not larger than
  \[ p_N = \frac{\mathbbm{P} \left( \tilde{\tau}^{J, \delta, K}_{N\mathcal{R}}
     \leqslant \tau^r (x) - c'_{d, \delta} - \gamma \right)}{\mathbbm{P}
     \left( \tau^J_{\mathcal{R}^N_i} \leqslant \tau^r (x_i), \forall i:
     \mathcal{R}_i \cap \mathcal{R} \neq \emptyset \right)} . \]
  According to the definition of $c'_{d, \delta}$, to Propositions~3.12 and
  3.13 in {\cite{Wou09CMP}} (here we use the assumption that $\beta >
  \hat{\beta}_c$ and $\beta \nin \mathcal{N}$) and to the definition of
  $I_{\tmmathbf{n}}$ at (\ref{def-In}), for large enough $K$ we have
  \begin{eqnarray}
    \limsup_N \frac{1}{N^{d - 1}} \log p_N & \leqslant & - h^{d - 1}
    I_{\tmmathbf{n}} \left( \tau^r \left( x \right) - \gamma \right) + \sum_{i
  : \mathcal{R}_i \cap \mathcal{R} \neq \emptyset} h_i^{d - 1}
    I_{\tmmathbf{n}_i} \left( \tau^r \left( x_i \right) \right) . \nonumber\\
    &  &  \label{eq-upb-pN}
  \end{eqnarray}
  We will show that the right-hand side of (\ref{eq-upb-pN}) is negative, for
  small enough $\delta > 0$ and $\delta_0 \leqslant h^2$. Thanks to (vi) in
  Definition \ref{def-g-adapted}, one sees (Figure \ref{fig-dilution}) that
  the right-hand side of (\ref{eq-upb-pN}) is not larger, for $\delta
  \leqslant 1 / 2$ and $\delta_0 \leqslant h^2$, than
  \[ \int_{\partial^{\star} u_0 \cap \mathcal{R}} \left[ I_{\tmmathbf{n}_z}
     \left( \tau^r \left( z \right) \right) - I_{\tmmathbf{n}_z} \left( \tau^r
     \left( z \right) - \gamma \right) \right] d\mathcal{H}^{d - 1} (z) + M
     \delta h^{d - 1} \]
  where
  \[ M = 3 + 2 \sup_{x \in \partial^{\star} u_0} I_{\tmmathbf{n}_x^{u_0}}
     \left( \tau^r \left( x \right) \right) \]
  is finite thanks to the definition of $\tmop{IC}$ at (\ref{IC}). Now we give
  an upper bound on the integral. Since $I_{\tmmathbf{n}}$ is convex, its
  slope is non-increasing and therefore
  \begin{eqnarray*}
    \int_{\partial^{\star} u_0 \cap \mathcal{R}} \left[ I_{\tmmathbf{n}_z}
    \left( \tau^r \left( z \right) \right) - I_{\tmmathbf{n}_z} \left( \tau^r
    \left( z \right) - \gamma \right) \right] d\mathcal{H}^{d - 1} (z) &
    \leqslant & - \int_{\partial^{\star} u_0 \cap \mathcal{R}}
    I_{\tmmathbf{n}_z} \left( \tau^q \left( \tmmathbf{n}_z \right) - \gamma
    \right) d\mathcal{H}^{d - 1} (z)\\
    & \leqslant & - (1 / 2) h^{d - 1} \inf_{\tmmathbf{n} \in S^{d - 1}}
    I_{\tmmathbf{n}} \left( \tau^q \left( \tmmathbf{n} \right) - \gamma
    \right) .
  \end{eqnarray*}
  Now we show that $\inf_{\tmmathbf{n} \in S^{d - 1}} I_{\tmmathbf{n}} \left(
  \tau^q \left( \tmmathbf{n} \right) - \gamma \right) > 0$. If not, we can
  extract a converging sequence $\tmmathbf{n}_k \rightarrow \tmmathbf{n}$ in
  the compact set $S^{d - 1}$ with $I_{\tmmathbf{n}_k} \left( \tau^q \left(
  \tmmathbf{n}_k \right) - \gamma \right) \rightarrow 0$ as $k \rightarrow
  \infty$. The Fenchel-Legendre transform $\tau^{\lambda} (\tmmathbf{n}) =
  \inf_{\tau > \tau^{\min} (\tmmathbf{n})} \{\lambda \tau + I_{\tmmathbf{n}}
  (\tau)\}$ satisfies therefore
  \[ \tau^{\lambda} (\tmmathbf{n}_k) \leqslant \lambda \tau^q (\tmmathbf{n}_k)
     - \lambda \gamma + \underset{k \rightarrow \infty}{o} (1) \text{, \ \ \ }
     \forall \lambda > 0. \]
  Since $\tau^{\lambda}$ and $\tau^q$ are continuous (Proposition 2.5 in
  {\cite{Wou09CMP}}), we obtain in the limit $\tau^{\lambda} (\tmmathbf{n})
  \leqslant \lambda \tau^q (\tmmathbf{n}) - \lambda \gamma$, and therefore, by
  duality of the Fenchel-Legendre transform $I_{\tmmathbf{n}} \left( \tau^q
  \left( \tmmathbf{n} \right) - \gamma \right) = 0$, which contradicts the
  assumption $\beta \notin \mathcal{N}_I$. The claim follows for any
  $\delta < 1 / 2$ with $\delta < \inf_{\tmmathbf{n} \in S^{d - 1}} I_{\tmmathbf{n}}
  \left( \tau^q \left( \tmmathbf{n} \right) - \gamma \right) / (2 M)$.
\end{proof}

\begin{proof}
  (Proposition \ref{prop-cond-upb-phco}). Let $\gamma < \xi / (2 a)$. We take
  $\delta > 0$ small enough so that we can use Lemma \ref{lem-cond-lwb-tauJt}
  with $c'_{d, \delta} \leqslant \gamma$. Let $u \in \tmop{BV}_a$ and consider
  $(\mathcal{R}_i^u)_{i = 1 \ldots n (u)}$ a $(u_0, \tau^r, \delta,
  \gamma)$-covering of $\partial^{\star} u$. Proposition 3.11 in
  {\cite{Wou09CMP}} states that
  \begin{eqnarray*}
    \frac{1}{N^{d - 1}} \log \mu^{J, +}_{\Lambda_N} \left(
    \frac{\mathcal{M}_K}{m_{\beta}} \in \mathcal{V}(u, \varepsilon) \right) &
    \leqslant & - \sum_{i = 1}^{n (u)} \left( h_i^u \right)^{d - 1}
    \tilde{\tau}^{J, \delta, K}_{N\mathcal{R}_i^u}
  \end{eqnarray*}
  for $\varepsilon > 0$ small enough. According to the definitions of $\gamma$
  and $\delta$ and to the properties of the coverings, provided $\delta$ was
  chosen small enough (still independently of $u \in \tmop{BV}_a$) we have
  \begin{eqnarray*}
    \sum_{i = 1}^{n (u)} (h_i^u)^{d - 1} \left( \tau^r (x_i) - c'_{d, \delta}
    - \gamma \right) & \geqslant & \mathcal{F}^r (u) - \xi .
  \end{eqnarray*}
  Finally when we take $\delta_0 < \min_{i = 1}^{n (u)} (h_i^u)^2$ and $K$
  large enough, Lemma \ref{lem-cond-lwb-tauJt} concludes the proof.
\end{proof}

Finally we remark that dilution has little influence on the overall
magnetization
\begin{eqnarray}
  m_{\Lambda} & = & \frac{1}{| \Lambda |} \sum_{x \in \Lambda} \sigma_x . 
  \label{mLambda}
\end{eqnarray}
\begin{proposition}
  \label{prop-meq}Assume that $\beta > \hat{\beta}_c$ and $\beta \notin
  \mathcal{N}$. Let $(u_0, \tau^r) \in \tmop{IC}$ and $\varepsilon > 0$. Let
  $\delta, \gamma > 0$ and consider $(\mathcal{R}_i)_{i = 1 \ldots n}$ a
  $(u_0, \tau^r, \delta, \gamma)$-covering of $\partial^{\star} u_0$. Let $N >
  0$ and
  \begin{eqnarray*}
    \mathcal{E}_N & = & E \left( \left( \bigcup_{i = 1}^n \mathcal{R}_i^N
    \right) \bigcup \left( \Lambda_N + \left( \mathbbm{Z}^d \setminus \{0\}
    \right) \left[ N + \left[ \sqrt{N} \right] \right] \right) \right)
  \end{eqnarray*}
  the set of edges in some $\mathcal{R}_i^N$, or in some of the translates of
  $\Lambda_N$ by $z (N + [ \sqrt{N}])$, for $z \in {\mathbb Z}^d \setminus \{0\}$. Then, if $\delta > 0$
  is small enough,
  \begin{eqnarray*}
    \lim_N \mathbbm{P} \left( \mu^{J, +} \left( m_{\Lambda_N} \right)
    \geqslant m_{\beta} - \varepsilon |J_e = 0, \forall e \in \mathcal{E}_N
    \right) & = & 1.
  \end{eqnarray*}
\end{proposition}

Note that $\mu^{J, +} \left( m_{\Lambda_N} \right)$ increases with every
$J_e$, therefore the condition that $J_e = 0, \forall e \in \mathcal{E}_N$ is
the worse condition that one can consider.

\begin{proof}
  The proof is based on the renormalization procedure established in Theorem
  5.10 in {\cite{Wou08SPA}}. As it is similar to that of Proposition 3.9 in
  {\cite{Wou09CMP}} we only sketch the argument. We
  cover $E (\mathbbm{Z}^d) \setminus \mathcal{E}_N$ with blocks with
  side-length $L_N = [ \sqrt{N}]$. To these blocks is associated a family of
  independent random variables $\varphi_i \in \{- 1, 0, 1\}$ (independent of
  the $J_e, e \in \mathcal{E}_N$), the local $\varepsilon$-phase, that is zero
  with probability less than $e^{- c \sqrt{N}}$. Write $\left. \mu^{J, +}
  (m_{\Lambda_N}) = \lim_M \mu_{\hat{\Lambda}_M}^{J, +} (m_{\Lambda_N})
  \right.$ and take some finite $M$. A simple Peierls estimate shows that,
  given the condition $J_e = 0, \forall e \in \mathcal{E}_N$, the
  $\varepsilon$-local phase associated to any block in $\Lambda_N$ is $+ 1$
  with a probability going to one as $N \rightarrow \infty$, uniformly in $N$.
  As described in Theorem 5.10 in {\cite{Wou08SPA}}, this event implies that
  \begin{eqnarray*}
    m_{\Lambda_N} & \geqslant & \left( m_{\beta} - \varepsilon \right) \left(
    1 - \sum_{i = 1}^n \tmop{Vol} (\mathcal{R}_i) \right) .
  \end{eqnarray*}
  Therefore we just need to take $\varepsilon > 0$ small enough (named
  $\delta$ in Theorem 5.10 in {\cite{Wou08SPA}}) and $\delta > 0$ small enough
  so that $\sum_{i = 1}^n \tmop{Vol} (\mathcal{R}_i) \simeq \delta
  \mathcal{H}^{d - 1} (\partial^{\star} u_0)$ is negligible.
\end{proof}

\subsubsection{The bottleneck}

\label{sec-bottleneck}Here we focus on the bottleneck in the dynamics. We
recall that the set $\mathcal{C}_{\varepsilon} (u_0)$, the set of sequences of
profiles that evolve from $u_0$ to $\tmmathbf{1}$ with jumps in $L^1$-norm
less than $\varepsilon$, was introduced at (\ref{Cepsu0}). Given $v = (v_i)_{i
= 0 \ldots k} \in \mathcal{C}_{\varepsilon} (u_0)$, we call
\begin{eqnarray*}
  \tmop{argmax}^{\varepsilon} (v) & = & \min \left\{ i: \inf_{u \in
  \mathcal{V}(v_i, \varepsilon)} \mathcal{F}^r (u) = \max_{l \leqslant k}
  \inf_{u \in \mathcal{V}(v_l, \varepsilon)} \mathcal{F}^r (u) \right\} .
\end{eqnarray*}
Now, given $(u_0, \tau^r) \in \tmop{IC}$, we define the
$\varepsilon$-bottleneck set as
\begin{eqnarray}
  \mathcal{B}_{\varepsilon} (u_0, \tau^r) & = & \bigcup_{v \in
  \mathcal{C}_{\varepsilon} (u_0)} \mathcal{V}(v_{\tmop{argmax}^{2
  \varepsilon} (v)}, \varepsilon) .  \label{Bottleneck}
\end{eqnarray}
When $(u_0, \tau^r)$ are clear from the context, we simply write
$\mathcal{B}_{\varepsilon}$ for $\mathcal{B}_{\varepsilon} (u_0, \tau^r)$.
Note that our motivation for the above definition is that the
$\varepsilon$-enlargement of $\mathcal{B}_{\varepsilon}$ is
$\mathcal{V}(\mathcal{B}_{\varepsilon}, \varepsilon) = \bigcup_{v \in
\mathcal{C}_{\varepsilon} (u_0)} \mathcal{V}(v_{\tmop{argmax}^{2 \varepsilon}
(v)}, 2 \varepsilon)$, a fact that helps in the proof of Lemma~\ref{lem-FB}
below.

Now we state three Lemmas related to the bottleneck set.
Lemma~\ref{lem-MK-bottleneck} gives the asymptotics of the probability that
the phase profile $\mathcal{M}_K / m_{\beta}$ belongs to
$\mathcal{B}_{\varepsilon}$. In Lemma~\ref{lem-Flsci} we show that
$\mathcal{F}^r$ is lower semi-continuous, a requisite for the proof of
Lemma~\ref{lem-FB} in which we show how the gap in surface energy
$\mathcal{K}^r (u_0)$ defined at (\ref{eq-def-Kr-disc}) is related to the
$\varepsilon$-enlargement of $\mathcal{B}_{\varepsilon} (u_0, \tau^r)$.

\begin{lemma}
  \label{lem-MK-bottleneck}Assume $\beta > \hat{\beta}_c$ and $\beta \notin
  \mathcal{N} \cup \mathcal{N}_I$. Let $(u_0, \tau^r) \in \tmop{IC}$ and $\xi,
  \varepsilon > 0$. There exists $\gamma > 0$ such that, for any $(u_0,
  \tau^r, \delta, \gamma)$-covering $(\mathcal{R}_i)_{i = 1 \ldots n}$ of
  $\partial^{\star} u_0$ with $\delta > 0$ small enough, for any $K$ large
  enough,
  \begin{eqnarray*}
    \lim_N \mathbbm{P} \left( \left. \mu^{J, +}_{\Lambda_N} \left(
    \frac{\mathcal{M}_K}{m_{\beta}} \in \mathcal{B}_{\varepsilon} \right)
    \leqslant \exp \left( - N^{d - 1} \left(
    \inf_{\mathcal{V}(\mathcal{B}_{\varepsilon}, \varepsilon)} \mathcal{F}^r -
    \xi \right) \right) \right| \mathcal{G}((\mathcal{R}_i)_{i = 1 \ldots n})
    \right) & = & 1.
  \end{eqnarray*}
  
\end{lemma}

\begin{proof} The exponential tightness property (Proposition~3.15 in {\cite{Wou09CMP}}) tells that there exists $C > 0$ such that, for every $\varepsilon'
  > 0$, for any $K$ large enough,
  \begin{eqnarray*}
    \limsup_N \frac{1}{N^{d - 1}} \log \mathbbm{E} \mu^{J, +}_{\Lambda_N}
    \left( \frac{\mathcal{M}_K}{m_{\beta}} \nin \mathcal{V}(\tmop{BV}_a,
    \varepsilon') \right) & \leqslant & - Ca
  \end{eqnarray*}
 (the set $BV_a$ was defined at (\ref{eq-def-BVa})). An immediate application of Markov's inequality shows that
  \begin{eqnarray*}
    \frac{1}{N^{d - 1}} \log \mu^{J, +}_{\Lambda_N} \left(
    \frac{\mathcal{M}_K}{m_{\beta}} \nin \mathcal{V}(\tmop{BV}_a,
    \varepsilon') \right) & \leqslant & - \frac{Ca}{2}
  \end{eqnarray*}
  with a probability at least $1 - \exp (- CaN^{d - 1} / 3)$, for large $N$.
  Since the surface cost of $\mathcal{G}((\mathcal{R}_i)_{i = 1 \ldots n})$ is
  bounded, we have therefore, for any $\varepsilon' > 0$, for any $a$ large
  enough, for $K$ large enough:
  \begin{eqnarray}
    \lim_N \mathbbm{P} \left( \left. \mu^{J, +}_{\Lambda_N} \left(
    \frac{\mathcal{M}_K}{m_{\beta}} \nin \mathcal{V}(\tmop{BV}_a,
    \varepsilon') \right) \leqslant \exp \left( - \frac{Ca}{2} N^{d - 1}
    \right) \right| \mathcal{G}((\mathcal{R}_i)_{i = 1 \ldots n}) \right) & =
    & 1.  \label{exptight}
  \end{eqnarray}
  Now we take some $\xi, \varepsilon > 0$ and $a > 0$ large. We take for
  $\gamma > 0$ the one given by Proposition \ref{prop-cond-upb-phco} (it does not
  depend on $\varepsilon$), and for any $u \in \tmop{BV}_a$ we denote by
  $\varepsilon_{\xi} (u) > 0$ the parameter $\varepsilon$ given by the same
  Proposition, and $\varepsilon (u) = \min (\varepsilon, \varepsilon_{\xi}
  (u))$. Since the set $\tmop{BV}_a$ is compact, it is covered with a finite
  number of balls
  \[ \tmop{BV}_a \subset \bigcup_{i = 1}^n \mathcal{V}(u_i, \varepsilon (u_i))
  \]
  where $u_i \in \tmop{BV}_a$, for all $i \in \{1, \ldots, n\}$. Since the
  right-hand set is open while $\tmop{BV}_a$ is compact, there is
  $\varepsilon' > 0$ such that
  \[ \mathcal{V}(\tmop{BV}_a, \varepsilon') \subset \bigcup_{i = 1}^n
     \mathcal{V} \left( u_i, \varepsilon (u_i) \right) \]
  and we can write
  \begin{eqnarray*}
    \mathcal{B}_{\varepsilon} \cap \mathcal{V}(\tmop{BV}_a, \varepsilon') &
    \subset & \bigcup_{i \leqslant n: u_i \in
    \mathcal{V}(\mathcal{B}_{\varepsilon}, \varepsilon)} \mathcal{V} \left(
    u_i, \varepsilon (u_i) \right) .
  \end{eqnarray*}
  It follows then from Proposition \ref{prop-cond-upb-phco} that, for any
  $\delta > 0$ small enough, for any $(u_0, \tau^r, \delta, \gamma)$-covering
  $(\mathcal{R}_i)_{i = 1 \ldots n}$ of $\partial^{\star} u_0$ and any $K$
  large enough,
  \begin{eqnarray}
    \lim_N \mathbbm{P} \left( \left. 
    \begin{array}{l}
    \mu^{J, +}_{\Lambda_N} \left(
    \frac{\mathcal{M}_K}{m_{\beta}} \in \mathcal{B}_{\varepsilon} \cap
    \mathcal{V}(\tmop{BV}_a, \varepsilon') \right) \leqslant \\
    n \exp \left( -
    N^{d - 1} \left( \min_{i \leqslant n: u_i \in
    \mathcal{V}(\mathcal{B}_{\varepsilon}, \varepsilon)} \mathcal{F}^r (u_i) -
    \xi \right) \right)
    \end{array}    
     \right| \mathcal{G}((\mathcal{R}_i)_{i = 1 \ldots n})
    \right) & = & 1.  \label{mkB}
  \end{eqnarray}
  Combining (\ref{mkB}) with (\ref{exptight}) proves the Lemma, provided that
  $a > 0$ was chosen large enough.
\end{proof}

\begin{lemma}
  \label{lem-Flsci}For any $(u_0, \tau^r) \in \tmop{IC}$, the functional
  $\mathcal{F}^r$ defined at (\ref{eq-def-Fr}) is lower semi-continuous.
\end{lemma}

\begin{proof}
  We show the lower semi-continuity as an application of the covering Proposition
  (Proposition \ref{prop-g-cover}). Let $u \in \tmop{BV}$ and $\delta, \gamma > 0$,
  and consider a $(u_0, \tau^r, \delta, \gamma)$-covering $(\mathcal{R}_i)_{i
  = 1 \ldots n}$ for $\partial^{\star} u$. Since the $\mathcal{R}_i$ are
  disjoint, for any $v \in \tmop{BV}$ we have
  \begin{eqnarray*}
    \mathcal{F}^r (v) & \geqslant & \sum_{i = 1}^n \int_{\dot{\mathcal{R}}_i
    \cap \partial^{\star} u_0 \cap \partial^{\star} v} \tau^r (x)
    d\mathcal{H}^{d - 1} (x) + \int_{( \dot{\mathcal{R}}_i \setminus
    \partial^{\star} u_0) \cap \partial^{\star} v} \tau^q (\tmmathbf{n}_x^v)
    d\mathcal{H}^{d - 1} (x)\\
    & \geqslant & \sum_{i: \mathcal{R}_i \cap \partial U_0 \neq \emptyset}
    \mathcal{H}^{d - 1} \left( \dot{\mathcal{R}}_i \cap \partial^{\star} v
    \right) \inf_{x \in \mathcal{R}_i} \tau^r + \sum_{i: \mathcal{R}_i \cap
    \partial U_0 = \emptyset} \int_{\dot{\mathcal{R}}_i \cap \partial^{\star}
    v} \tau^q (\tmmathbf{n}_x^v) d\mathcal{H}^{d - 1} (x)
  \end{eqnarray*}
  since $\tau^r (x) \leqslant \tau^q (\tmmathbf{n}_x)$. Thanks to the lower
  semi-continuity of the surface energy in open sets (Chapter 14 in
  {\cite{Cer06LNM}}), the quantities $\mathcal{H}^{d - 1} (
  \dot{\mathcal{R}}_i \cap \partial^{\star} v)$ and $\int_{\dot{\mathcal{R}}_i
  \cap \partial^{\star} v} \tau^q (\tmmathbf{n}_x^v) d\mathcal{H}^{d - 1} (x)$
  become not smaller than their value at $u$ when $v$ converges to $u$ in
  $L^1$ norm. Hence:
  \begin{eqnarray*}
    \lim_{\varepsilon \rightarrow 0} \inf_{v \in \mathcal{V}(u, \varepsilon)}
    \mathcal{F}^r (v) & \geqslant & \sum_{i: \mathcal{R}_i \cap \partial U_0
    \neq \emptyset} (1 - \delta) h_i^{d - 1} \inf_{x \in \mathcal{R}_i} \tau^r
    + \sum_{i: \mathcal{R}_i \cap \partial U_0 = \emptyset} (1 - \delta)
    h_i^{d - 1} \tau^q (\tmmathbf{n}_i^u)
  \end{eqnarray*}
  which is arbitrary close to $\mathcal{F}^r (u)$ for small $\delta$, thanks
  to the uniform continuity of $\tau^r$.
\end{proof}

\begin{lemma}
  \label{lem-FB}For any $(u_0, \tau^r) \in \tmop{IC}$,
  \begin{equation}
    \mathcal{K}^r (u_0) = \lim_{\varepsilon \rightarrow 0^+} \inf_{u \in
    \mathcal{V}(\mathcal{B}_{\varepsilon} (u_0, \tau^r), \varepsilon)}
    \mathcal{F}^r (u) -\mathcal{F}^r (u_0) . \label{eq-Kr-L8e}
  \end{equation}
\end{lemma}

\begin{proof}
  We prove first that $\mathcal{K}^r (u_0)$ is larger or equal to the
  right-hand side. Let $v \in \mathcal{C}_{\varepsilon} (u_0)$. Then,
  \begin{eqnarray*}
    \max_i \mathcal{F}^r (v_i) & \geqslant & \max_i \inf_{\mathcal{V}(v_i, 2
    \varepsilon)} \mathcal{F}^r\\
    & \geqslant & \inf_{u \in \mathcal{V}(\mathcal{B}_{\varepsilon} (u_0,
    \tau^r), \varepsilon)} \mathcal{F}^r (u)
  \end{eqnarray*}
  and, when we optimize over $v \in \mathcal{C}_{\varepsilon} (u_0)$ and let
  $\varepsilon \rightarrow 0$ we obtain the inequality
  \begin{eqnarray*}
    \mathcal{K}^r (u_0) +\mathcal{F}^r (u_0) & \geqslant & \lim_{\varepsilon
    \rightarrow 0^+} \inf_{u \in \mathcal{V}(\mathcal{B}_{\varepsilon} (u_0,
    \tau^r), \varepsilon)} \mathcal{F}^r (u) .
  \end{eqnarray*}
  Now we prove the opposite inequality. Given any $\varepsilon > 0$, there is
  $u \in \mathcal{V}(\mathcal{B}_{\varepsilon} (u_0, \tau^r), \varepsilon)$
  such that
  \begin{eqnarray}
    \mathcal{F}^r (u) & \leqslant & \inf_{w \in
    \mathcal{V}(\mathcal{B}_{\varepsilon} (u_0, \tau^r), \varepsilon)}
    \mathcal{F}^r (w) + \varepsilon .  \label{eq-KL-1}
  \end{eqnarray}
  According to the definition of $\mathcal{B}_{\varepsilon} (u_0, \tau^r)$ at
  (\ref{Bottleneck}), there is $v = (v_i)_{i = 0 \ldots k}$ in
  $\mathcal{C}_{\varepsilon} (u_0)$, that is an $\varepsilon$-continuous
  evolution such that $v_0 = u_0$, $v_k =\tmmathbf{1}$, that satisfies
  \begin{eqnarray}
    \mathcal{F}^r (u) & \geqslant & \max_{i = 0}^k \inf_{w \in
    \mathcal{V}(v_i, 2 \varepsilon)} \mathcal{F}^r (w) .  \label{eq-KL-2}
  \end{eqnarray}
  Now, for each $i = 0, \ldots, k$ we consider $v'_{i + 1} \in
  \mathcal{V}(v_i, 2 \varepsilon)$ such that
  \begin{eqnarray}
    \mathcal{F}^r (v'_{i + 1}) & \leqslant & \inf_{w \in \mathcal{V}(v_i, 2
    \varepsilon)} \mathcal{F}^r (w) + \varepsilon .  \label{def-vp}
  \end{eqnarray}
  We let also $v'_0 = u_0$ and $v'_{k + 2} =\tmmathbf{1}$. Clearly, the
  evolution $v' = (v'_i)_{i = 0 \ldots k + 2}$ is $5 \varepsilon$-continuous.
  Therefore, we have
  \begin{eqnarray*}
    \inf_{v'' \in \mathcal{C}_{5 \varepsilon} (u_0)} \max_i \mathcal{F}^r
    (v''_i) & \leqslant & \max_{i = 0}^{k + 2} \mathcal{F}^r (v'_i) .
  \end{eqnarray*}
  The maximum in the right-hand side does not occur at $i = k + 2$ since
  $\mathcal{F}^r (\tmmathbf{1}) = 0$. We also have the bound
  \begin{eqnarray*}
    \mathcal{F}^r (v'_0) -\mathcal{F}^r (v'_1) & \leqslant & \mathcal{F}^r
    (u_0) - \inf_{w \in \mathcal{V}(u_0, 2 \varepsilon)} \mathcal{F}^r (w)
  \end{eqnarray*}
  therefore,
  \begin{eqnarray*}
    \inf_{v'' \in \mathcal{C}_{5 \varepsilon} (u_0)} \max_i \mathcal{F}^r
    (v''_i) & \leqslant & \max_{i = 1}^{k + 1} \mathcal{F}^r (v'_i) + \left(
    \mathcal{F}^r (u_0) - \inf_{w \in \mathcal{V}(u_0, 2 \varepsilon)}
    \mathcal{F}^r (w) \right)\\
    & \leqslant & \max_{i = 0}^k \inf_{w \in \mathcal{V}(v_i, 2 \varepsilon)}
    \mathcal{F}^r (w) + \varepsilon + \left( \mathcal{F}^r (u_0) - \inf_{w \in
    \mathcal{V}(u_0, 2 \varepsilon)} \mathcal{F}^r (w) \right)\\
    & \leqslant & \mathcal{F}^r (u) + \varepsilon + \left( \mathcal{F}^r
    (u_0) - \inf_{w \in \mathcal{V}(u_0, 2 \varepsilon)} \mathcal{F}^r (w)
    \right)\\
    & \leqslant & \inf_{w \in \mathcal{V}(\mathcal{B}_{\varepsilon} (u_0,
    \tau^r), \varepsilon)} \mathcal{F}^r (w) + 2 \varepsilon + \left(
    \mathcal{F}^r (u_0) - \inf_{w \in \mathcal{V}(u_0, 2 \varepsilon)}
    \mathcal{F}^r (w) \right)
  \end{eqnarray*}
  where the second line is due to the definition of $v'_{i + 1}$ at
  (\ref{def-vp}), the third one to (\ref{eq-KL-2}) and the last one to
  (\ref{eq-KL-1}). The lower semi-continuity of $\mathcal{F}^r$
  (Lemma~\ref{lem-Flsci}) imply that the last term goes to $0$ as $\varepsilon
  \rightarrow 0$. Therefore taking $\varepsilon \rightarrow 0$ ends the proof.
\end{proof}

\subsubsection{Intermediate formulation of the metastability}
\label{sec:int:meta}
The aim of this Section is the proof of an intermediate and useful formulation
of the metastability:

\begin{proposition}
  \label{prop-mt}Assume $\beta > \hat{\beta}_c$ and $\beta \notin \mathcal{N}
  \cup \mathcal{N}_I$. Let $(u_0, \tau^r) \in \tmop{IC}$, $\xi > 0$ and
  $\varepsilon > 0$ small enough. Then, there exists $\gamma > 0$ such that,
  for any $\delta > 0$ small enough, for any $(u_0, \tau^r, \delta,
  \gamma)$-covering $(\mathcal{R}_i)_{i = 1 \ldots n}$ of $\partial^{\star}
  u_0$,
  \begin{eqnarray*}
    \lim_N \inf_{t \leqslant K_N} \mathbbm{P} \left( \left. \mu^{J,
    +}_{\Lambda_N} \left( T^{J, +}_{\Lambda_N} \left( t \right) m_{\Lambda_N}
    \leqslant m_{\beta} - 2 \varepsilon \right) \geqslant e^{- N^{d - 1}
    \left( \mathcal{F}^r (u_0) + \xi \right)} \right| \mathcal{G} \left(
    (\mathcal{R}_i)_{i = 1 \ldots n} \right) \right) & = & 1
  \end{eqnarray*}
  where $K_N = \exp (N^{d - 1} \left( \mathcal{K}^r (u_0) - \xi \right))$.
\end{proposition}

Important keys for the proof of Proposition~\ref{prop-mt} are
Lemmas~\ref{lem-MK-bottleneck} and \ref{lem-FB} about the bottleneck of the
dynamics, together with Lemma~\ref{lem:contevol} below on the so-to-say
``continuous evolution'' of the magnetization profile.

\begin{lemma}
  \label{lem:contevol}Let $\beta > 0$ and $\varepsilon > 0$. There is $c > 0$
  such that, for any $N$ large enough,
  \begin{eqnarray*}
    \sup_{t \in [0, \varepsilon / (4 c_M)]} \sup_J \tmmathbf{P}^{J, \Lambda_N,
    +}_{\mu^{J, +}_{\Lambda_N}} \left( \left\| \mathcal{M}_K (\sigma (t))
    -\mathcal{M}_K (\sigma (0)) \right\|_{L^1} \geqslant \varepsilon \right) &
    \leqslant & \exp \left( - cN^d \right) 
  \end{eqnarray*}
where $c_M$ is a uniform upper bound on the rates of the Glauber dynamics.
\end{lemma}

\begin{proof}
  (Lemma \ref{lem:contevol}). The $L^1$ distance $\left\| \mathcal{M}_K
  (\sigma (t)) -\mathcal{M}_K (\sigma (0)) \right\|_{L^1}$ is bounded by $2 /
  N^d$ times the number of jumps of the Glauber dynamics. These jumps occur at
  a rate bounded by $c_M$. Therefore,
  \begin{eqnarray*}
    \tmmathbf{P}^{J, \Lambda_N, +}_{\mu^{J, +}_{\Lambda_N}} \left( \left\|
    \mathcal{M}_K (\sigma (t)) -\mathcal{M}_K (\sigma (0)) \right\|_{L^1}
    \geqslant \varepsilon \right) & \leqslant & P \left( X \geqslant
    \frac{\varepsilon N^d}{2} \right)
  \end{eqnarray*}
  where $X$ is a Poisson variable with parameter $c_M tN^d \leqslant 
  \varepsilon N^d / 4$. Cram\'er's Theorem imply the claim.
\end{proof}

\begin{proof}
  (Proposition~\ref{prop-mt}). In view of Proposition~\ref{prop-cond-lwb-phco}
  it suffices to prove that
  \begin{eqnarray}
    \lim_N \inf_{t \leqslant K_N} \mathbbm{P} \left( p^J \left. \leqslant
    \frac{1}{2} e^{- N^{d - 1} \left( \mathcal{F}^r (u_0) + \xi \right)}
    \right| \mathcal{G} \left( (\mathcal{R}_i)_{i = 1 \ldots n} \right)
    \right) & = & 1  \label{eq:claim:pJ}
  \end{eqnarray}
  where
  \begin{eqnarray*}
    p^J & = & \mu^{J, +}_{\Lambda_N} \left( \begin{array}{l}
      \mathcal{M}_K / m_{\beta} \in \mathcal{V}(u_0, \varepsilon) \text{
      and}\\
      T^{J, +}_{\Lambda_N} (t) m_{\Lambda_N} > m_{\beta} - 2 \varepsilon
    \end{array} \right) .
  \end{eqnarray*}
  So we focus on the proof of (\ref{eq:claim:pJ}). We remark that, for any initial
  configuration $\rho$,
  \begin{eqnarray*}
    \tmmathbf{P}_{\rho}^{J, \Lambda_N, +} \left( m_{\Lambda_N} (\sigma (t))
    \geqslant m_{\beta} - 3 \varepsilon \right) < \varepsilon & \Rightarrow &
    (T^{J, +}_{\Lambda_N} \left( t \right) m_{\Lambda_N}) (\rho) < m_{\beta} -
    2 \varepsilon
  \end{eqnarray*}
  since $(T^{J, +}_{\Lambda_N} \left( t \right) m_{\Lambda_N}) (\rho)
  =\tmmathbf{E}_{\rho}^{J, \Lambda_N, +} \left( m_{\Lambda_N} (\sigma (t))
  \right)$ and $m_{\Lambda_N} (\sigma) \leqslant 1$, for any $\sigma$.
  Therefore
  \begin{eqnarray}
    p^J & \leqslant & \mu^{J, +}_{\Lambda_N} \left(  \begin{array}{l}
      \mathcal{M}_K / m_{\beta} \in \mathcal{V}(u_0, \varepsilon) \text{
      and}\\
      \tmmathbf{P}_.^{J, \Lambda_N, +} \left( m_{\Lambda_N} (\sigma (t))
      \geqslant m_{\beta} - 3 \varepsilon \right) \geqslant \varepsilon
    \end{array} \right) \nonumber\\
    & \leqslant & \frac{1}{\varepsilon} \tmmathbf{P}_{\mu^{J,
    +}_{\Lambda_N}}^{J, \Lambda_N, +} \left(  \begin{array}{l}
      \mathcal{M}_K (\sigma (0)) / m_{\beta} \in \mathcal{V}(u_0, \varepsilon)
      \text{ and}\\
      m_{\Lambda_N} (\sigma (t)) \geqslant m_{\beta} - 3 \varepsilon
    \end{array} \right)  \label{eq:upb:pJ}
  \end{eqnarray}
  where the second inequality is a consequence of Markov's inequality. Now we
  fix $k \in \mathbbm{N}^{\star}$ and $a > 0$ and call
  \begin{eqnarray*}
    C & = & \left\{ \left\| \frac{\mathcal{M}_K}{m_{\beta}} \left( \sigma
    \left( \frac{(i + 1) t}{k} \right) \right) -
    \frac{\mathcal{M}_K}{m_{\beta}} \left( \sigma \left( \frac{it}{k} \right)
    \right) \right\|_{L^1} \leqslant \varepsilon \text{, \ \ } \forall i < k
    \right\}
  \end{eqnarray*}
  the event of continuity, and
  \begin{eqnarray*}
    D_a & = & \left\{ \forall i \leqslant k \text{, \ } \exists v_i \in
    \tmop{BV}_a: \frac{\mathcal{M}_K}{m_{\beta}} \left( \sigma \left(
    \frac{it}{k} \right) \right) \in \mathcal{V}(v_i, \varepsilon) \right\}
  \end{eqnarray*}
  the event that the magnetization profile is close to $\tmop{BV}_a$ at any
  time $it / k$. Then we remark that
  \begin{eqnarray}
    \left. \begin{array}{l}
      \frac{\mathcal{M}_K}{m_{\beta}} \left( \sigma \left( 0 \right) \right)
      \in \mathcal{V}(u_0, \varepsilon),\\
      m_{\Lambda_N} (\sigma (t)) \geqslant m_{\beta} - 3 \varepsilon\\
      \text{and } C \cap D_a \text{ occurs}
    \end{array} \right\} & \Rightarrow & \begin{array}{l}
      \exists v = (v_i)_{i = 0 \ldots k} \in \mathcal{C}_{8 \varepsilon /
      m_{\beta}} (u_0),\\
      \frac{\mathcal{M}_K}{m_{\beta}} \left( \sigma \left( \frac{it}{k}
      \right) \right) \in \mathcal{V}(v_i, \varepsilon) .
    \end{array}  \label{imp:vk}
  \end{eqnarray}
  Indeed, the definition of $D_a$ yields a sequence $v = (v_i)_{i = 0 \ldots
  k}$. We can take $v_0 = u_0$ since we know that $\mathcal{M}_K \left( \sigma
  \left( 0 \right) \right) / m_{\beta} \in \mathcal{V}(u_0, \varepsilon)$.
  This sequence is $3 \varepsilon$-continuous according to the properties of
  the $v_i$ and to the definition of $C$. Finally, we have
  \begin{eqnarray}
    \|\tmmathbf{1}- v_k \|_{L^1} & = & \int_{[0, 1]^d} (1 - v_k (x))
    d\mathcal{L}^d (x) \nonumber\\
    & \leqslant & \int_{[0, 1]^d} \left( 1 -\mathcal{M}_K (\sigma (t)) (x) /
    m_{\beta} \right) d\mathcal{L}^d (x) + \varepsilon \nonumber\\
    & \leqslant & 1 - m_{\Lambda_N} (\sigma_t) / m_{\beta} + 2 \varepsilon
    \nonumber\\
    & \leqslant & 5 \varepsilon / m_{\beta}  \label{eq-d1vk}
  \end{eqnarray}
  and therefore $\mathcal{M}_K \left( \sigma \left( t \right) \right) /
  m_{\beta} \in \mathcal{V}(\tmmathbf{1}, 6 \varepsilon / m_{\beta})$. So we
  can fix $v_k =\tmmathbf{1}$, which makes the evolution $8 \varepsilon /
  m_{\beta}$-continuous.
  
  Now we conclude the proof of (\ref{eq:claim:pJ}). As a consequence of
  (\ref{eq:upb:pJ}) and (\ref{imp:vk}) we have the inequality
  \begin{eqnarray*}
    \varepsilon p^J & \leqslant & \tmmathbf{P}_{\mu^{J, +}_{\Lambda_N}}^{J,
    \Lambda_N, +} \left( \begin{array}{l}
      \exists v = (v_i)_{i = 0 \ldots k} \in \mathcal{C}_{8 \varepsilon /
      m_{\beta}} (u_0),\\
      \frac{\mathcal{M}_K}{m_{\beta}} \left( \sigma \left( \frac{it}{k}
      \right) \right) \in \mathcal{V}(v_i, \varepsilon)
    \end{array} \right) +\tmmathbf{P}_{\mu^{J, +}_{\Lambda_N}}^{J, \Lambda_N,
    +} (C^c \cup D_a^c)
  \end{eqnarray*}
  and therefore, according to the invariance of $\mu^{J, +}_{\Lambda_N}$ for
  the Glauber dynamics, and to the definition (\ref{Bottleneck}) of the
  bottleneck set $\mathcal{B}_{\varepsilon} (u_0, \tau^r)$,
  \begin{eqnarray*}
    p^J & \leqslant & p_1^J + p_2^J + p_3^J
  \end{eqnarray*}
  where
  \begin{eqnarray*}
    p_1^J & = & \frac{k + 1}{\varepsilon} \mu^{J, +}_{\Lambda_N} \left(
    \frac{\mathcal{M}_K}{m_{\beta}} \in \mathcal{B}_{8 \varepsilon /
    m_{\beta}} \right)\\
    p_2^J & = & \frac{k + 1}{\varepsilon} \mu^{J, +}_{\Lambda_N} \left(
    \frac{\mathcal{M}_K}{m_{\beta}} \nin \mathcal{V}(\tmop{BV}_a, \varepsilon)
    \right)\\
    p_3^J & = & \frac{k}{\varepsilon} \tmmathbf{P}^{J, \Lambda_N, +}_{\mu^{J,
    +}_{\Lambda_N}} \left( \left\| \mathcal{M}_K \left( \sigma \left(
    \frac{t}{k_N} \right) \right) -\mathcal{M}_K (\sigma (0)) \right\|_{L^1}
    \geqslant \varepsilon \right) .
  \end{eqnarray*}
  Finally we bound each contribution separately. First, according to
  Lemmas~\ref{lem-MK-bottleneck} and \ref{lem-FB} we have
  \begin{eqnarray*}
    \lim_N \mathbbm{P} \left( \left. \mu^{J, +}_{\Lambda_N} \left(
    \frac{\mathcal{M}_K}{m_{\beta}} \in \mathcal{B}_{\frac{8
    \varepsilon}{m_{\beta}}} \right) \leqslant e^{- N^{d - 1} \left(
    \mathcal{K}^r (u_0) +\mathcal{F}^r (u_0) - 2 \xi \right)} \right|
    \mathcal{G}((\mathcal{R}_i)_{i = 1 \ldots n}) \right) & = & 1
  \end{eqnarray*}
  for $\varepsilon > 0$ small enough, for some $\gamma > 0$, and for any
  $(u_0, \tau^r, \delta, \gamma)$-covering $(\mathcal{R}_i)_{i = 1 \ldots n}$
  of $\partial^{\star} u_0$ with $\delta > 0$ small enough, for any $K$ large
  enough. So for
  \[ k = k_N = [\exp (N^{d - 1} \left( \mathcal{K}^r (u_0) - \xi / 2 \right))]
  \]
  we have, under the same conditions,
  \begin{eqnarray}
    \lim_N \mathbbm{P} \left( \left. p_1^J \leqslant \frac{1}{6} e^{- N^{d -
    1} \left( \mathcal{F}^r (u_0) - \xi \right)} \right|
    \mathcal{G}((\mathcal{R}_i)_{i = 1 \ldots n}) \right) & = & 1. 
  \end{eqnarray}
  Then, the exponential tightness property (see (\ref{exptight}) above)
  implies that, for $a > 0$ large enough,
  \begin{eqnarray}
    \left. \lim_N \mathbbm{P} \left( p_2^J \leqslant \frac{1}{6} e^{- N^{d -
    1} \left( \mathcal{F}^r (u_0) - \xi \right)} \right|
    \mathcal{G}((\mathcal{R}_i)_{i = 1 \ldots n}) \right) & = & 1 
  \end{eqnarray}
  and finally Lemma~\ref{lem:contevol} implies that, for large $N$, uniformly
  over $J$ and $t \leqslant K_N$,
  \begin{eqnarray}
    p_3^J & \leqslant & \exp \left( - cN^d \right) . 
  \end{eqnarray}
  Summing the last three displays yields (\ref{eq:claim:pJ}).
\end{proof}

\subsubsection{Proofs of Theorems~\ref{thm-lwb-A} and \ref{thm-Trel}}

\label{sec-lwb-final}Here we conclude the proof of Theorems~\ref{thm-lwb-A}
and \ref{thm-Trel}. We state one more Lemma that relates the averaged
autocorrelation $A^{\lambda} (t)$ defined at (\ref{def-A}) to the dynamics of
the overall magnetization $m_{\Lambda_N}$ defined at (\ref{mLambda}).

\begin{lemma}
  \label{lem-A-finvol}For any $\varepsilon > 0$, $N \in
  \mathbbm{N}^{\star}$ and $\lambda \geqslant 1$ one has
  \begin{eqnarray}
    A^{\lambda} (t) & \geqslant & \varepsilon^{2 \lambda} \mathbbm{E} \left[
    \tmmathbf{1}_{\left\{ \mu^{J, +} (m_{\Lambda_N}) \geqslant m_{\beta} -
    \varepsilon \right\}} \times \left( \mu^{J, +}_{\Lambda_N} \left( T^{J,
    +}_{\Lambda_N} (t) m_{\Lambda_N} \leqslant m_{\beta} - 2 \varepsilon
    \right) \right)^{\lambda} \right]  \label{eq-lwb-autocor}
  \end{eqnarray}
\end{lemma}

\begin{proof}
  Minkowski's inequality implies, for any $\nu \geqslant 1$, that
  \[ \frac{1}{N^d} \sum_{x \in \Lambda_N} \left[ \int_{\Sigma} \left| T^J (t)
     \pi_x - \mu^{J, +} \pi_x \right|^{\nu} d \mu^{J, +} \right]^{1 / \nu} \]
  \[ \geqslant \left[ \int_{\Sigma} \left| T^J (t) m_{\Lambda_N} - \mu^{J, +}
     (m_{\Lambda_N}) \right|^{\nu} d \mu^{J, +} \right]^{1 / \nu} \]
  where $\pi_x: \Sigma \rightarrow \mathbbm{R}$ is the function which
  associates, to the spin configuration $\sigma \in \Sigma$, the spin at $x$,
  $\sigma (x)$. Taking $\nu=1/\lambda$, the translation invariance of $\mathbbm{E}$ and of the Glauber
  dynamics implies that
  \begin{eqnarray*}
    A^{\lambda} (t) & \geqslant & \mathbbm{E} \left[ \left( \int_{\Sigma}
    \left| T^J (t) m_{\Lambda_N} - \mu^{J, +} (m_{\Lambda_N}) \right|^{2} d
    \mu^{J, +} \right)^{\lambda} \right] .
  \end{eqnarray*}
  Hence, for any $\varepsilon > 0$ we have
  \begin{eqnarray*}
    A^{\lambda} (t) & \geqslant & \varepsilon^{2 \lambda} \mathbbm{E}
    \left[ \tmmathbf{1}_{\left\{ \mu^{J, +} (m_{\Lambda_N}) \geqslant
    m_{\beta} - \varepsilon \right\}} \times \left( \mu^{J, +} \left( T^J (t)
    m_{\Lambda_N} \leqslant m_{\beta} - 2 \varepsilon \right)
    \right)^{\lambda} \right]
  \end{eqnarray*}
  and we conclude by using the attractivity of the Glauber dynamics:
  \begin{eqnarray}
    \mu^{J, +} \left( T^J (t) m_{\Lambda_N} \leqslant m_{\beta} - 2
    \varepsilon \right) & \geqslant & \mu^{J, +} \left( T_{\Lambda_N}^{J, +}
    (t) m_{\Lambda_N} \leqslant m_{\beta} - 2 \varepsilon \right) \nonumber\\
    & \geqslant & \mu^{J, +}_{\Lambda_N} \left( T^{J, +}_{\Lambda_N} (t)
    m_{\Lambda_N} \leqslant m_{\beta} - 2 \varepsilon \right) 
    \label{eq-muT-infvol}
  \end{eqnarray}
\end{proof}

\begin{proof}
  (Theorem \ref{thm-lwb-A}). Let $\delta > 0$ and $\lambda \geqslant 1$. We assume that
  $\mathcal{X}_{\lambda} < \infty$, otherwise there is nothing to prove. We
  fix $(u_0, \tau^r) \in \tmop{IC}$ and $\xi \in (0, \mathcal{K}^r (u_0))$
  such that
  \begin{eqnarray*}
    \frac{\mathcal{I}^r (u_0) + \xi + \lambda (\mathcal{F}^r (u_0) +
    \xi)}{\mathcal{K}^r (u_0) - \xi} (1 + \xi) & \leqslant &
    \mathcal{X}_{\lambda} + \delta / 2.
  \end{eqnarray*}
  Then, for any $t > 0$ we call $N (t)$ the smallest integer $N$ such that $t
  \leqslant \exp (N^{d - 1} \left( \mathcal{K}^r (u_0) - \xi \right))$.
  According to Lemma \ref{lem-A-finvol}, to Propositions \ \ref{prop-meq} and
  \ref{prop-mt}, to Lemma~\ref{lem-proba-GN}, for any $\varepsilon > 0$ small
  enough we can find $\delta, \gamma > 0$ such that, provided that $N (t)$ is
  large enough,
  \begin{eqnarray*}
    A^{\lambda} (t) & \geqslant & \frac{\varepsilon^{\lambda}}{4} \exp \left(
    - \lambda N (t)^{d - 1} \left( \mathcal{F}^r (u_0) + \xi \right) - N
    (t)^{d - 1} (\mathcal{I}^r (u_0) + \xi) \right) .
  \end{eqnarray*}
  The definition of $N (t)$ implies finally that
  \begin{eqnarray*}
    A^{\lambda} (t) & \geqslant & \frac{\varepsilon^{\lambda}}{4} \exp \left(
    - (\log t) \frac{\lambda \left( \mathcal{F}^r (u_0) + \xi \right) +
    (\mathcal{I}^r (u_0) + \xi)}{\mathcal{K}^r (u_0) - \xi}  \left( \frac{N
    (t)}{N (t) - 1} \right)^{d - 1} \right)
  \end{eqnarray*}
  for $t$ large enough, and this gives the claim as $N (t) \rightarrow +
  \infty$.
\end{proof}

\begin{proof}
  (Theorem \ref{thm-Trel}) Let $\delta > 0$ and assume that $\kappa > 0$.
  There exists $(u_0, \tau^r) \in \tmop{IC}$ and $\xi > 0$ such that
  \begin{eqnarray*}
    d \frac{\mathcal{K}^r (u_0) - \xi}{\mathcal{I}^r (u_0) + \xi} & \leqslant
    & \kappa - \delta / 2.
  \end{eqnarray*}
  Given $N \in \mathbbm{N}^{\star}$ large, we define an intermediate
  side-length $M = M (N)$ as the largest integer $M$ such that
  \begin{eqnarray*}
    N^d & \geqslant & \exp (M^{d - 1} (\mathcal{I}^r (u_0) + 2 \xi)) .
  \end{eqnarray*}
  First, we prove that the event of dilution, on scale $M$, occurs with large
  probability inside $\Lambda_M$. More precisely, we consider $L = M + [
  \sqrt{M]}$ and
  \begin{eqnarray*}
    X_N & = & \{0, \ldots, [N / L] - 1\}^d .
  \end{eqnarray*}
  We also consider $\delta > 0$ and a $\delta$-covering for $\partial u_0$,
  and for any $x \in X$ we call
  \begin{eqnarray*}
    \mathcal{G}_M^x & = & \{ \text{dilution occurs in } Lx + \Lambda_M \}
  \end{eqnarray*}
  the $L x$-translate of $\mathcal{G}_M$. According to Lemma \ref{lem-proba-GN}
  we have, for $\delta$ small enough,
  \begin{eqnarray}
    \mathbbm{P} \left( \bigcap_{x \in X} \left( \mathcal{G}^x_M \right)^c
    \right) & \leqslant & \left( 1 - \exp (- M^{d - 1} (\mathcal{I}^r (u_0) +
    \xi)) \right)^{[N / L]^d} \nonumber\\
    & \leqslant & \exp \left( - \left[ \frac{N}{L} \right]^d \exp (- M^{d -
    1} (\mathcal{I}^r (u_0) + \xi)) \right) \nonumber\\
    & \leqslant & \exp \left( - \left( \frac{1}{C \log N} \right)^d \exp
    (M^{d - 1} \xi) \right) \nonumber\\
    & \leqslant & \exp \left( - N^{ \frac{d \xi}{2 (\mathcal{I}^r (u_0) + 2
    \xi)}} \right) 
  \end{eqnarray}
  where we use the definition of $M$ and the fact that $L \leqslant C \log N$,
  for some $C < \infty$. We call now $\mathcal{D}_M^x =\mathcal{G}_M^x \bigcap
  \bigcap_{z \in X, z < x} \left( \mathcal{G}^z_M \right)^c$ the event that
  $x$ is the smallest box for which dilution occurs, where $X$ is ordered
  according to the lexicographic order. According to Proposition \ref{prop-mt}
  and to the attractivity of the dynamics, the conditional probability
  \[ \mathbbm{P} \left( \left. \mu^{J, +}_{\Lambda_N} \left( T^{J,
     +}_{\Lambda_N} \left( \exp (M^{d - 1} \left( \mathcal{K}^r (u_0) - \xi
     \right)) \right) m_{Lx + \Lambda_M} \leqslant m_{\beta} - 2 \varepsilon
     \right) \geqslant \exp \left( - M^{d - 1} \left( \mathcal{F}^r (u_0) +
     \xi \right) \right) \right| \mathcal{D}_M^x \right) \]
  goes to $1$ as $N \rightarrow \infty$ (and thus $M$). On the other hand,
  Proposition \ref{prop-meq} and the monotonicity of the system imply that, as
  well, the conditional probability
  \[ \mathbbm{P} \left( \left. \mu^{J, +}_{\Lambda_N} \left( m_{Lx +
     \Lambda_M}) \geqslant m_{\beta} - \varepsilon \right) \right|
     \mathcal{D}_M^x \right) \]
  goes to $1$ as $N \rightarrow \infty$. Therefore, we have shown that the
  $\mathbbm{P}$-probability that there exists $x \in X_N$ with both
  \begin{eqnarray*}
    \mu^{J, +}_{\Lambda_N} \left( T^{J, +}_{\Lambda_N} \left( \exp (M^{d - 1}
    \left( \mathcal{K}^r (u_0) - \xi \right)) \right) m_{Lx + \Lambda_M}
    \leqslant m_{\beta} - 2 \varepsilon \right) & \geqslant & \exp \left( - M^{d -
    1} \left( \mathcal{F}^r (u_0) + \xi \right) \right) \\
 \mu^{J, +}_{\Lambda_N} \left( m_{Lx + \Lambda_M} \right) & \geqslant  & m_{\beta} - \varepsilon 
  \end{eqnarray*} 
  goes to one as $N \rightarrow \infty$. Taking $m_{Lx + \Lambda_M}$ as a test
  function in (\ref{eq-decay-var-gap}) proves the lower bound on the
  relaxation time. For obtaining the lower bound on the mixing time one can
  consider $f = m_{Lx + \Lambda_M}$ in (\ref{eq-decay-Tmix}), and $\sigma_0$
  any initial configuration for which $T^{J, +}_{\Lambda_N} \left( \exp (M^{d
  - 1} \left( \mathcal{K}^r (u_0) - \xi \right)) \right) m_{Lx + \Lambda_M}
  \leqslant m_{\beta} - 2 \varepsilon$, which is possible as this event has
  positive probability.
\end{proof}

\subsection{The geometry of relaxation}

\label{sec:geom}The first part of this Section is dedicated to the proof of
Proposition~\ref{prop:gap}. Then we show that the gap in surface energy can be
computed on more restrictive evolutions, and finally we compute $\mathcal{K}^r
(u_0)$ for two simple initial configurations.

\subsubsection{A lower bound on the gap in surface energy}
\label{sec:lwb:gap}
The statement of theorems \ref{thm-lwb-A} and \ref{thm-Trel} is non-empty only
if $\mathcal{K}^r (u_0)$ happens to be strictly positive for some initial
profiles $(u_0, \tau^r) \in \tmop{IC}$. We state the next Theorem, which has
Proposition~\ref{prop:gap} as a Corollary:

\begin{theorem}
  \label{thm:Kr:pos}Let $(u_0, \tau^r) \in \tmop{IC}$. Assume that the
  boundary of $u_0$ is $\mathcal{C}^1$. There exists a non-decreasing function
  $H: \mathbbm{R}^+ \rightarrow \mathbbm{R}^+$, with $H (\delta) > 0$,
  $\forall \delta > 0$, that depends only on $(u_0, \tau^r)$, such that
  \begin{eqnarray}
    \mathcal{K}^r (u_0) & \geqslant & \sup_{\tmmathbf{n} \in S^{d - 1}} \tau^q
    (\tmmathbf{n}) \times H \left( \frac{\inf_{x \in \partial^{\star} u_0}
    (\tau^q (\tmmathbf{n}_x^{u_0}) - \tau^r (x))}{\max (\sup_{\tmmathbf{n} \in
    S^{d - 1}} \tau^q (\tmmathbf{n}), 1)} \right) .  \label{eq:Kr:lwb}
  \end{eqnarray}
\end{theorem}

\begin{proof}
  (Proposition~\ref{prop:gap}). When $\tau^{\min} (\tmmathbf{n}) < \tau^q
  (\tmmathbf{n})$ for all $\tmmathbf{n} \in \mathcal{S}^{d - 1}$ and the
  boundary of $u_0$ is $\mathcal{C}^1$, it follows at once from
  (\ref{eq:Kr:lwb}) that $\mathcal{K}^r (u_0) > 0$. Now we address the proof
  of (\ref{eq-upb-Xl}). We discuss first the consequences of the assumption
  $\mathbbm{P}(J_e = 0) > 0$. When $J = 0$ along a section of a rectangle
  $\mathcal{R}$ the surface tension in that rectangle is $\tau^J_{\mathcal{R}}
  = 0$, thus for all $\varepsilon > 0$, all $\beta > 0$,
  \begin{eqnarray}
    I_{\tmmathbf{n}} (\varepsilon) & \leqslant & -\|\tmmathbf{n}\|_1 \log
    \mathbbm{P}(J_e = 0) .  \label{upb:In}
  \end{eqnarray}
  Now we consider the initial configuration defined by
  \[ u_0 = \chi_{B (z_0, 1 / 4)} \text{ \ \ and \ \ } \tau^r (x) = \min (1,
     \tau^q (\tmmathbf{n}_x^{u_0}) / 2), \forall x \in \partial^{\star} u_0 \]
  where $z_0 = (1 / 2, \ldots, 1 / 2)$. According to (\ref{upb:In}) there is
  $C < + \infty$ not depending on $\beta$ such that the cost of dilution
  $\mathcal{I}^r (u_0)$ defined at (\ref{eq-def-Ir}) satisfies $\mathcal{I}^r
  (u_0) \leqslant C$. We also have $\mathcal{F}^r (u_0) \leqslant C'$ where
  $\mathcal{F}^r (u_0)$ is the initial surface energy defined at
  (\ref{eq-def-Fr}), and $C'$ is the perimeter of $u_0$. So we conclude
  already that the numerator $\mathcal{I}^r (u_0) + \lambda \mathcal{F}^r
  (u_0)$ in the definition (\ref{eq-def-Xl}) of $\mathcal{X}_{\lambda}$ is
  bounded by a constant which does not depend on $\beta$. Finally we recall
  that, thanks to the assumption that $\mathbbm{P}(J_e = 0) < 1 - p_c (d)$ and
  to Proposition 2.13 in {\cite{Wou09CMP}},
  \begin{eqnarray}
    \liminf_{\beta \rightarrow \infty} \inf_{n \in \mathcal{S}^{d - 1}}
    \frac{\tau^q_{\beta} (\tmmathbf{n})}{\beta} & > & 0,  \label{tauq:beta}
  \end{eqnarray}
  and therefore the inequality (\ref{eq:Kr:lwb}) implies that, for some $c >
  0$ and for all $\beta$ large enough,
  \begin{eqnarray*}
    \mathcal{K}^r (u_0) & \geqslant & \sup_{\tmmathbf{n} \in S^{d - 1}} \tau^q
    (\tmmathbf{n}) \times H \left( \frac{1}{2}  \frac{\inf_{\tmmathbf{n} \in
    S^{d - 1}} \tau^q (\tmmathbf{n})}{\sup_{\tmmathbf{n} \in S^{d - 1}} \tau^q
    (\tmmathbf{n})} \right) .
  \end{eqnarray*}
  The convexity and the lattice symmetries of the Wulff crystal $\mathcal{W}^q$ imply that the ratio
  $$\frac{ \inf_{\tmmathbf{n} \in S^{d - 1}} \tau^q (\tmmathbf{n}) }{  \sup_{\tmmathbf{n} \in S^{d - 1}} \tau^q (\tmmathbf{n})}$$
   is bounded from
  below by some constant $c_d$ not depending on $\beta$, hence $$\mathcal{K}^r
  (u_0) \geqslant \sup_{\tmmathbf{n} \in S^{d - 1}} \tau^q (\tmmathbf{n}) \times H
  (c_d) \geqslant C \beta$$ for large $\beta$. The claim (\ref{eq-upb-Xl})
  follows.
\end{proof}

The proof of Theorem~\ref{thm:Kr:pos} requires three Lemmas, that are stated
now together with their proof. We introduce the new surface energy
\begin{equation}
  \mathcal{F}^{r, -} (u) = \int_{\partial^{\star} u \setminus \partial^{\star}
  u_0} \tau^q (\tmmathbf{n}^u_.) d\mathcal{H}^{d - 1} - \int_{\partial^{\star}
  u \cap \partial^{\star} u_0} \tau^r d\mathcal{H}^{d - 1} \label{eq-def-Frm}
  \text{, \ \ } \forall u \in \tmop{BV} .
\end{equation}
\begin{lemma}
  \label{lem-lwb-Fu-Fu0}Let $(u_0 = \chi_{U_0}, \tau^r) \in \tmop{IC}$ and $u
  = \chi_U \in \tmop{BV}$. Then, the profiles $v = \chi_{U_0 \setminus U}$ and
  $w = \chi_{U \setminus U_0}$ satisfy
  \begin{eqnarray}
    \mathcal{F}^r (u) -\mathcal{F}^r (u_0) & = & \mathcal{F}^{r, -} (v)
    +\mathcal{F}^{r, -} (w) .  \label{eq-lwb-Fu-Fu0}
  \end{eqnarray}
\end{lemma}

\begin{proof}
  We remark first that
  \begin{eqnarray}
    \mathcal{F}^r (u) -\mathcal{F}^r (u_0) & = & \int_{\partial^{\star} u
    \setminus \partial^{\star} u_0} \tau^q (\tmmathbf{n}^u_.) d\mathcal{H}^{d
    - 1} - \int_{\partial^{\star} u_0 \setminus \partial^{\star} u} \tau^r
    d\mathcal{H}^{d - 1} .  \label{eq:mtr}
  \end{eqnarray}
  As stated in Theorem 3.61 in {\cite{AFP00Book}}, given any $u = \chi_U \in
  \tmop{BV}$ the local density of $U$ at $x$ is either $0, 1 / 2$ or $1$, for
  $\mathcal{H}^{d - 1}$-almost all $x \in \mathbbm{R}^d$, and the set of
  points at which the local density is $1 / 2$ coincides, up to a
  $\mathcal{H}^{d - 1}$-negligible set, to the reduced boundary
  $\partial^{\star} u$ (denoted $\mathcal{F}U$ in {\cite{AFP00Book}}). This
  implies the two equalities
  \begin{eqnarray}
    \partial^{\star} u \setminus \partial^{\star} u_0 & = & \left(
    \partial^{\star} v \setminus \partial^{\star} u_0 \right) \sqcup \left(
    \partial^{\star} w \setminus \partial^{\star} u_0 \right), 
    \label{eq:b1}\\
    \partial^{\star} u_0 \setminus \partial^{\star} u & = & (\partial^{\star}
    w \cap \partial^{\star} u_0) \sqcup (\partial^{\star} v \cap
    \partial^{\star} u_0)  \label{eq:b2}
  \end{eqnarray}
  up to $\mathcal{H}^{d - 1}$-negligible sets, where $\sqcup$ stands for the
  disjoint union (again, up to $\mathcal{H}^{d - 1}$-negligible sets).
  Furthermore, the outer normal at $x \in \partial^{\star} w \setminus
  \partial^{\star} u_0$ (resp.~$x \in \partial^{\star} v \setminus
  \partial^{\star} u_0$) corresponds ($\mathcal{H}^{d - 1}$-a.s.) to the outer
  normal at $x \in \partial^{\star} u \setminus \partial^{\star} u_0$
  (resp.~to the opposite of the former). Thus, in conjunction with
  (\ref{eq:mtr}), equations (\ref{eq:b1}) and (\ref{eq:b2}) imply respectively
  the $\tau^q$ and the $\tau^r$ part of (\ref{eq-lwb-Fu-Fu0}).
\end{proof}

\begin{lemma}
  \label{lem-dec-vwi}Let $(u_0, \tau^r) \in \tmop{IC}$ and $u = \chi_U \in
  \tmop{BV}$. For any $h > 0$, there is a finite collection of droplets $(v_i
  = \chi_{V_i})_{i \in I}$ such that
  \begin{enumerate}[i.]
    \item The $V_i$ are disjoint and their union is $U$,
    
    \item Each $V_i$ is included in a box of side-length at most $h$,
    
    \item And
    \begin{eqnarray}
      \mathcal{F}^{r, -} (u) & \geqslant & \sum_{i \in I} \mathcal{F}^{r, -}
      (v_i) -\|\tmmathbf{1}- u\|_{L^1} \frac{d \tau^q (\tmmathbf{e}_1)}{h} . 
      \label{eq:Fu:Fvi}
    \end{eqnarray}
  \end{enumerate}
\end{lemma}

\begin{proof}
  We define $I =\{0, \ldots, \lceil 1 / h \rceil\}$ and call
  \begin{eqnarray*}
    V_i & = & \{x \in U: 0 \leqslant x_k - z_k - i_k h \leqslant h \text{, \
    } \forall k = 1, \ldots, d\} \text{, \ \ } \forall i \in I
  \end{eqnarray*}
  where $z$ is some point of $(0, h)^d$. The first two properties hold
  trivially. Now we chose $z$ so that (\ref{eq:Fu:Fvi}) holds. For every $k
  \in \{1, \ldots, d\}$ we have
  \begin{eqnarray*}
    \int_0^h \mathcal{H}^{d - 1} \left( \left\{ x \in U: (x_k - z_k) / h \in
    \mathbbm{N} \right\} \right) \mathd z_k & = & \mathcal{L}^d (U)
  \end{eqnarray*}
  and therefore we can chose $z_k \in (0, h)$ such that $\mathcal{H}^{d - 1}
  \left( \left\{ x \in U: (x_k - z_k) / h \in \mathbbm{N} \right\} \right)
  \leqslant \mathcal{L}^d (U) / h =\|\tmmathbf{1}- u\|_{L^1} / (2 h)$.
  Equation (\ref{eq:Fu:Fvi}) follows as the new portions of interface in the
  $v_i$ that were not present in $v$ are exactly the $\left\{ x \in U: (x_k -
  z_k) / h \in \mathbbm{N} \right\}$ as $k = 1, \ldots, d$.
\end{proof}

\begin{lemma}
  \label{lem:Frm:pos}Assume that $u \in \tmop{BV}$, $\tmmathbf{n}_0 \in
  \mathcal{S}^{d - 1}$ and $\varepsilon > 0$ satisfy
  \begin{eqnarray}
    \tau^r (x) + \left( \varepsilon +\|\tmmathbf{n}_x^u -\tmmathbf{n}_0 \|_2
    \right) \sup_{n \in S^{d - 1}} \tau^q (\tmmathbf{n}) & \leqslant & \tau^q
    (\tmmathbf{n}_0)  \label{eq:no}
  \end{eqnarray}
  for all $x \in \partial^{\star} u \cap \partial^{\star} u_0 \setminus
  \mathcal{N}$, where $\mathcal{H}^{d - 1} (\mathcal{N}) = 0$. Then
  \begin{eqnarray}
    \mathcal{F}^{r, -} (u) & \geqslant & c_d \varepsilon \sup \tau^q \left(
    \|\tmmathbf{1}- u\|_{L^1} \right)^{\frac{d - 1}{d}} 
  \end{eqnarray}
  where $c_d > 0$ is a constant that depends only on $d$.
\end{lemma}

\begin{proof}
  We let
  \begin{eqnarray*}
    \mathcal{W}^q & = & \{z \in \mathbbm{R}^d: z \cdot \tmmathbf{n} \leqslant
    \tau^q (\tmmathbf{n}), \forall \tmmathbf{n} \in S^{d - 1} \}\\
    \text{and \ \ } \tilde{\mathcal{W}} & = & \left\{ z \in \mathcal{W}^q: z
    \cdot \tmmathbf{n}_x^u \leqslant - \tau^r (x), \forall x \in
    \partial^{\star} u \cap \partial^{\star} u_0 \setminus \mathcal{N}
    \right\} .
  \end{eqnarray*}
  The surface tension associated to $\tilde{\mathcal{W}}$ is
  \begin{eqnarray*}
    \tilde{\tau} (\tmmathbf{n}) & = & \sup_{z \in \tilde{\mathcal{W}}} z \cdot
    \tmmathbf{n} \text{, \ \ \ } \forall \tmmathbf{n} \in S^{d - 1}
  \end{eqnarray*}
  and it satisfies obviously the relations
  \begin{eqnarray*}
    \tilde{\tau} (\tmmathbf{n}) & \leqslant & \tau^q (\tmmathbf{n}), \forall
    \tmmathbf{n} \in S^{d - 1}\\
    \tilde{\tau} (\tmmathbf{n}_x^u) & \leqslant & - \tau^r (x), \forall x \in
    \partial^{\star} u \cap \partial^{\star} u_0 \setminus \mathcal{N}
  \end{eqnarray*}
  therefore
  \begin{eqnarray}
    \mathcal{F}^{r, -} (u) & \geqslant & \int \tilde{\tau} (\tmmathbf{n}_x^u)
    \mathd \mathcal{H}^{d - 1} (x) . 
  \end{eqnarray}
  The isoperimetric inequality for $\tilde{\tau}$ (see {\cite{KP94JSP}})
  implies in turn that
  \begin{eqnarray}
    \mathcal{F}^{r, -} (u) & \geqslant & d\mathcal{L}^d (
    \tilde{\mathcal{W}})^{1 / d} \left( \frac{\|\tmmathbf{1}- u\|_{L^1}}{2}
    \right)^{\frac{d - 1}{d}} 
  \end{eqnarray}
  so the claim will follows from an appropriate lower bound on the volume
  $\tilde{\mathcal{W}}$. There exists $z_0 \in \mathcal{W}^q$ such that $z_0
  \cdot \tmmathbf{n}_0 = \tau^q (\tmmathbf{n}_0)$. Now we take $z \in
  \mathcal{W}^q$, close to $- z_0$, and prove that it lies in
  $\tilde{\mathcal{W}}$: if $x \in \partial^{\star} u \cap \partial^{\star}
  u_0 \setminus \mathcal{N}$, then
  \begin{eqnarray*}
    z \cdot \tmmathbf{n}_x^u & = & (z + z_0) \cdot \tmmathbf{n}_x^u - z_0
    \cdot (\tmmathbf{n}_x^u -\tmmathbf{n}_0) - z_0 \cdot \tmmathbf{n}_0\\
    & \leqslant & \|z + z_0 \|_2 + \sup_{S^{d - 1}} \tau^q \|\tmmathbf{n}_x^u
    -\tmmathbf{n}_0 \|_2 - \tau^q (\tmmathbf{n}_0)
  \end{eqnarray*}
  as the Euclidean norm of $z_0$ is not larger than $\tau^q (z_0 /\|z_0 \|_2)
  \leqslant \sup_{S^{d - 1}} \tau^q$. According to our assumption
  (\ref{eq:no}) it follows that
  \begin{eqnarray*}
    \tilde{\mathcal{W}} & \supset & \mathcal{W}^q \cap B \left( - z_0,
    \varepsilon \sup_{S^{d - 1}} \tau^q \right) .
  \end{eqnarray*}
  The set $\mathcal{W}^q$ is convex and contains all the images of $z_0$ by
  the symmetries of $\mathbbm{Z}^d$, therefore it contains an hypercube which
  vertices are the images of $z_0$ by the above mentioned symmetries.
  Consequently the volume of $\tilde{\mathcal{W}}$ is at least $1 / 2^d$ of
  the volume of the ball $B \left( - z_0, \varepsilon \sup \tau^q \right)$ and
  the proof is over.
\end{proof}

\begin{proof}
  (Theorem \ref{thm:Kr:pos}). Given $\delta > 0$, the continuity of $\tau^r$
  and the smoothness of $\partial u_0$ imply that there exists $h (\delta) >
  0$ such that
  \begin{eqnarray*}
    x, y \in \partial^{\star} u_0, \|x - y\|_2 \leqslant h (\delta) &
    \Rightarrow & \left\{ \begin{array}{l}
      | \tau^r (x) - \tau^r (y) | \leqslant \delta\\
      \|\tmmathbf{n}_x^{u_0} -\tmmathbf{n}_y^{u_0} \|_2 \leqslant \delta .
    \end{array} \right.
  \end{eqnarray*}
  As we claimed, $h (\delta)$ depends only on $(u_0, \tau^r)$, and of course
  on $\delta$. Also, it is clear that one can chose the function $h$
  non-decreasing. Furthermore, to a given direction $\tmmathbf{n} \in S^{d -
  1}$ corresponds $z \in \mathcal{W}^q$ such that $z \cdot \tmmathbf{n}=
  \tau^q (\tmmathbf{n})$ and therefore
  \begin{eqnarray*}
    \tau^q (\tmmathbf{n}') & \geqslant & z \cdot \tmmathbf{n}'\\
    & \geqslant & \tau^q (\tmmathbf{n}) -\|\tmmathbf{n}-\tmmathbf{n}' \|_2
    \sup \tau^q .
  \end{eqnarray*}
  It follows that any $u \in \tmop{BV}$ with diameter not greater than $h
  (\delta)$ satisfies assumption (\ref{eq:no}) of Lemma \ref{lem:Frm:pos} when
  one takes
  \begin{eqnarray*}
    \varepsilon \text{ \ } = \text{ \ } \delta & = & \frac{\inf_{x \in
    \partial^{\star} u_0} (\tau^q (\tmmathbf{n}_x^{u_0}) - \tau^r (x))}{6 \max
    (\sup_{\tmmathbf{n} \in S^{d - 1}} \tau^q (\tmmathbf{n}), 1)} .
  \end{eqnarray*}
  Now we consider some phase profile $u \in \tmop{BV}$, call $v$ and $w$ the
  profiles given by Lemma \ref{lem-lwb-Fu-Fu0}, and then $(v_i)_{i \in I}$ the
  union of the droplets associated to both $v$ and $w$ by Lemma
  \ref{lem-dec-vwi}. In conjunction with Lemma \ref{lem:Frm:pos}, we obtain
  \begin{eqnarray*}
    \mathcal{F}^r (u) -\mathcal{F}^r (u_0) & \geqslant & c_d \varepsilon \sup
    \tau^q \sum_{i \in I} \left( \|\tmmathbf{1}- v_i \|_{L^1} \right)^{\frac{d
    - 1}{d}}\\
    &  & - \left( \|\tmmathbf{1}- v\|_{L^1} +\|\tmmathbf{1}- w\|_{L^1}
    \right) \frac{d \tau^q (\tmmathbf{e}_1)}{h (\delta)} .
  \end{eqnarray*}
  According to the definition of $v, w$ and of the $v_i$ we have
  \begin{eqnarray*}
    \|u - u_0 \|_{L^1} & = & \|\tmmathbf{1}- v\|_{L^1} +\|\tmmathbf{1}-
    w\|_{L^1}\\
    & = & \sum_{i \in I} \|\tmmathbf{1}- v_i \|_{L^1}
  \end{eqnarray*}
  therefore, the trivial inequality $x^{\alpha} + y^{\alpha} \geqslant (x +
  y)^{\alpha}$ for $\alpha \in (0, 1)$ and $x, y > 0$ implies that, for every
  $u \in \tmop{BV}$,
  \begin{eqnarray*}
    \mathcal{F}^r (u) -\mathcal{F}^r (u_0) & \geqslant & c_d \delta \sup
    \tau^q \left( \|u - u_0 \|_{L^1} \right)^{\frac{d - 1}{d}} -\|u - u_0
    \|_{L^1} \frac{d \tau^q (\tmmathbf{e}_1)}{h (\delta)} .
  \end{eqnarray*}
  But $x^{1 - 1 / d} - bx$ is maximized for $x = \left( \frac{d - 1}{bd}
  \right)^d$ and equals $c_d (1 / b)^{d - 1}$. So,
  taking $b=\frac{d}{c_d} 
    \frac{\tau^q (\tmmathbf{e}_1)}{\sup  \tau^q}
    \frac{1}{\delta h(\delta)}$,
   for any $u \in \tmop{BV}$
  at $L^1$-distance $\left( \frac{d - 1}{bd} \right)^d$ from $u$ (this distance is smaller than the $L^1$ distance between $u_0$ and $\mathbf{1}$ if we require that $\sup_{\delta > 0} h (\delta)$ be small), we
  get
  \begin{eqnarray*}
    \mathcal{F}^r (u) -\mathcal{F}^r (u_0) & \geqslant & c'_d   \delta \sup \tau^q
     \left(\delta h (\delta) \right)^{d - 1}
  \end{eqnarray*}
  and we have proved (\ref{eq:Kr:lwb}) with $H (6 \delta) = c'_d \delta^d h^{d
  - 1} (\delta)$, as $\varepsilon = \delta$.
\end{proof}

\subsubsection{Continuous evolution and continuous separation}

\label{sec-contevol}In this paragraph we give an alternative formulation of
the surface energy gap $\mathcal{K}^r (u_0)$. We introduce a set of continuous
evolution of the phase profile, for which the boundary splits continuously
from its original location:
\begin{equation}
  \mathcal{C}(u_0) = \left\{ \text{\begin{tabular}{l}
    $v: t \in [0, 1] \mapsto v_t \in \tmop{BV}$: $v_0 = u_0$, $v_1 \equiv
    1$,\\
    $t \mapsto v_t$ is continuous for the $L^1$-norm\\
    and $t \mapsto \tmmathbf{1}_{\partial^{\star} v_t \cap \partial^{\star}
    u_0}$ is continuous for the \\
    $L^1$-norm associated to the measure $\mathcal{H}^{d - 1}$
  \end{tabular}} \right\} \label{eq-def-cont-evol}
\end{equation}
and then define $\mathcal{K}^r_{\tmop{cs}}$, the gap in the free energy
associated to the optimal droplet removal in this class of evolutions:
\begin{equation}
  \mathcal{K}^r_{\tmop{cs}} (u_0) = \inf_{v \in \mathcal{C}(u_0)} \sup_{t \in
  [0, 1]} \mathcal{F}^r (v_t) -\mathcal{F}^r (u_0) . \label{eq-def-Kr}
\end{equation}
\begin{theorem}
  \label{thm-Kdisc-Kcont}For all $(u_0, \tau^r) \in \tmop{IC}$,
  \[ \mathcal{K}_{\tmop{cs}}^r (u_0) =\mathcal{K}^r (u_0) . \]
\end{theorem}

The inequality $\mathcal{K}_{\tmop{cs}}^r (u_0) \geqslant \mathcal{K}^r (u_0)$
is clear since the set $\mathcal{C}(u_0)$ is more restrictive that those
evolutions considered in the definition (\ref{eq-def-Kr-disc}) of
$\mathcal{K}^r (u_0)$. We prove the reverse inequality with the help of an
\tmtextit{interpolation}: given an evolution $v = (v_i)_{i = 0 \ldots k} \in
\mathcal{C}_{\varepsilon} (u_0)$ (see (\ref{Cepsu0})) we interpolate from $v$
a continuous evolution $v'$ with continuous separation from $\partial^{\star}
u_0$, at the price of a small increase in the maximal cost, negligible as
$\varepsilon \rightarrow 0$ (and uniform over $v$):
\[ v' \in \mathcal{C}(u_0) \text{: \ } v_{i / k}' = v_i \text{ \ and \ }
   \sup_{t \in [0, 1]} \mathcal{F}^r (v'_t) \leqslant \max_{i = 0 \ldots k}
   \mathcal{F}^r (v_i) + \underset{\varepsilon \rightarrow 0}{o} (1) . \]
The next lemma is one of the keys to the proof of Theorem
\ref{thm-Kdisc-Kcont}.

\begin{lemma}
  \label{lem-Ut}Let $(u_0, \tau^r) \in \tmop{IC}$ and $\delta > 0$. For any
  Borel set $\Delta \subset [0, 1]^d$ with volume $\mathcal{L}^d (\Delta)
  \leqslant \delta$, there exists a collection of measurable sets $U =
  (U_t)_{t \in \mathbbm{R}}$ such that:
  \begin{enumerate}[i.]
    \item $t \mapsto U_t$ is a non-decreasing function with
    \[ \lim_{t \rightarrow - \infty} U_t = \emptyset \text{ \ and \ \ }
       \lim_{t \rightarrow + \infty} U_t =\mathcal{T} \]
    where $\mathcal{T}$ is the tube $\mathcal{T}= \left[ 0, 1 \right]^d \times
    \mathbbm{R}$.
    
    \item The function $t \mapsto U_t - t\tmmathbf{e}_d$ is $1$-periodic.
    
    \item The volume $t \mapsto \mathcal{L}^d (U_t \cap [0, 1]^d)$ is a
    continuous function of $t$
    
    \item The area $t \mapsto \mathcal{H}^{d - 1} \left( U_t \cap
    \partial^{\star} u_0 \right)$ is a continuous function of $t$
    
    \item The portion of the boundary of $U_t$ that intersects $\Delta
    +\mathbbm{Z}\tmmathbf{e}_d$ in $\dot{\mathcal{T}}$ has a small area:
    \[ \sup_{t \in [0, 1]} \mathcal{H}^{d - 1} \left( \partial U_t \cap \left(
       \Delta +\mathbbm{Z}\tmmathbf{e}_d \right) \cap \dot{\mathcal{T}}
       \right) \leqslant 7 \sqrt{\delta} \]
    for $\delta > 0$ small enough.
  \end{enumerate}
\end{lemma}

\begin{proof}
  See Figure \ref{fig-Et-Delta} for an illustration of the proof. To begin
  with, we partition $[0, 1]^d$ in horizontal slabs: let $n = \lfloor 1 /
  \sqrt{\delta} \rfloor$ and call
  \[ A_i = \left\{ x \in [0, 1]^d: \frac{i - 1}{n} \leqslant x \cdot
     \tmmathbf{e}_d \leqslant \frac{i}{n} \right\} \]
  for $i \in \{1, \ldots, n\}$. Because each slab has a volume at least
  $\sqrt{\delta}$, for each $i \in \{0, \ldots n - 1\}$ there exists $z_i \in
  (i / n, (i + 1) / n)$ such that the density $\mathcal{L}^{d - 1} \left(
  \Delta \cap \left\{ z: z \cdot \tmmathbf{e}_d = z_i \right\} \right)$ of
  $\Delta$ at height $z_i$ is not larger than $\sqrt{\delta}$. We extend then
  the definition of $z_i$ by periodicity, letting
  \[ z_{i + n} = 1 + z_i, \forall i \in \mathbbm{Z}. \]
  Then, we let $\tmmathbf{n}= \cos \alpha \tmmathbf{e}_1 + \sin \alpha
  \tmmathbf{e}_d$ for some $\alpha \in [0, \pi / 3]$ such that
  $\partial^{\star} u_0$ has no face orthogonal to $\tmmathbf{n}$ and define
  \[ U_t = \left\{ x \in \mathcal{T}: \begin{array}{l}
       x \cdot \tmmathbf{e}_d \leqslant z_{\lfloor nt \rfloor} \text{ or }\\
       x \cdot \tmmathbf{e}_d \leqslant z_{\lceil nt \rceil} \text{ and } x
       \cdot \tmmathbf{n} \leqslant l_t
     \end{array} \right\}, \text{ \ \ } \forall t \in \mathbbm{R} \]
  where $l$ is the piecewise linear function defined by: $\forall i \in
  \mathbbm{Z}$,
  \begin{eqnarray*}
    l_{(i / n)^+} & = & z_i \tmmathbf{e}_d \cdot \tmmathbf{n}\\
    l_{(i + 1) / n} & = & \left( z_{i + 1} \tmmathbf{e}_d +\tmmathbf{e}_1
    \right) \cdot \tmmathbf{n}
  \end{eqnarray*}
  and $l$ linear on each interval $(i / n, (i + 1) / n]$. The set $U_t$
  evolves as follows: between times $i / n$ and $(i + 1) / n$, $U_t$ invades
  the region $\{x \in [0, 1]^d: z_i \leqslant x \cdot \tmmathbf{e}_d
  \leqslant z_{i + 1} \}$ by the mean of a front line normal to $\tmmathbf{n}$,
  that moves at a constant speed.
  
  It is immediate from the definition that $U_t - t\tmmathbf{e}_d$ is
  $1$-periodic and that $U_t$ increases continuously in volume. The
  $\mathcal{H}^{d - 1}$ measure of $U_t \cap \partial^{\star} u_0$ is
  non-decreasing and the assumption on $\tmmathbf{n}$ ensures that it
  increases continuously. We consider at last the portion of the surface of
  $U_t$ in $\dot{\mathcal{T}}$ that might intersect $\Delta
  +\mathbbm{Z}\tmmathbf{e}_d$. We just have to take into account the upper
  portion of $\partial U_t$, made of the two planes at height $z_i, z_{i +
  1}$, and of a portion of plane normal to $\tmmathbf{n}$. Recall that the
  $z_i$ have been chosen so that the density of $\Delta
  +\mathbbm{Z}\tmmathbf{e}_d$ at height $z_i$ does not exceed $\sqrt{\delta}$.
  Similarly, because $\alpha \leqslant \pi / 3$, the piece of plane orthogonal
  to $\tmmathbf{n}$ has a surface at most $4 / n \leqslant 5 / \sqrt{\delta}$
  for $\delta > 0$ small enough. The claim follows.
\end{proof}

\begin{figure}[h!]
\begin{center}
  \includegraphics[width=7cm]{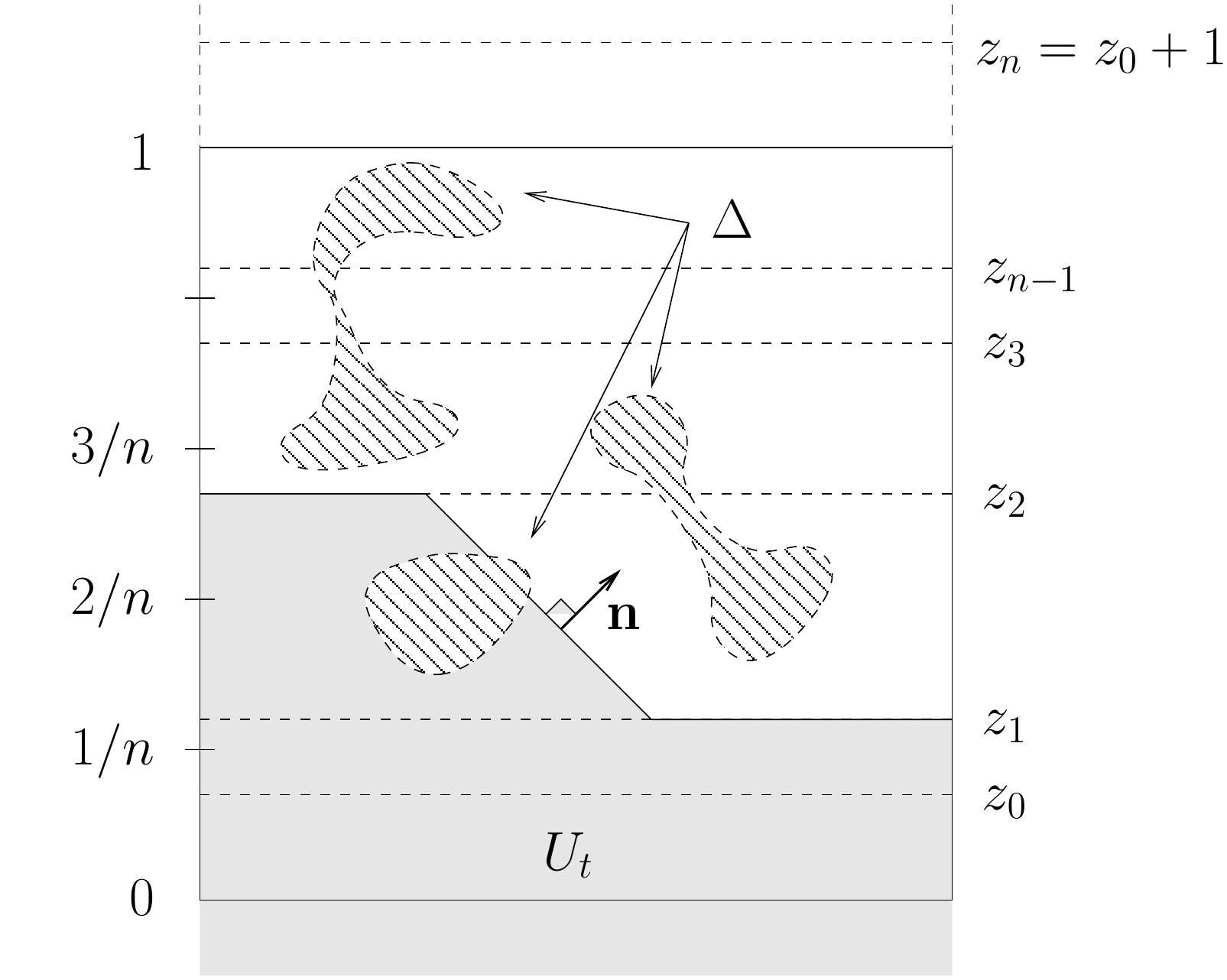}\hspace{2cm}\includegraphics[width=7cm]{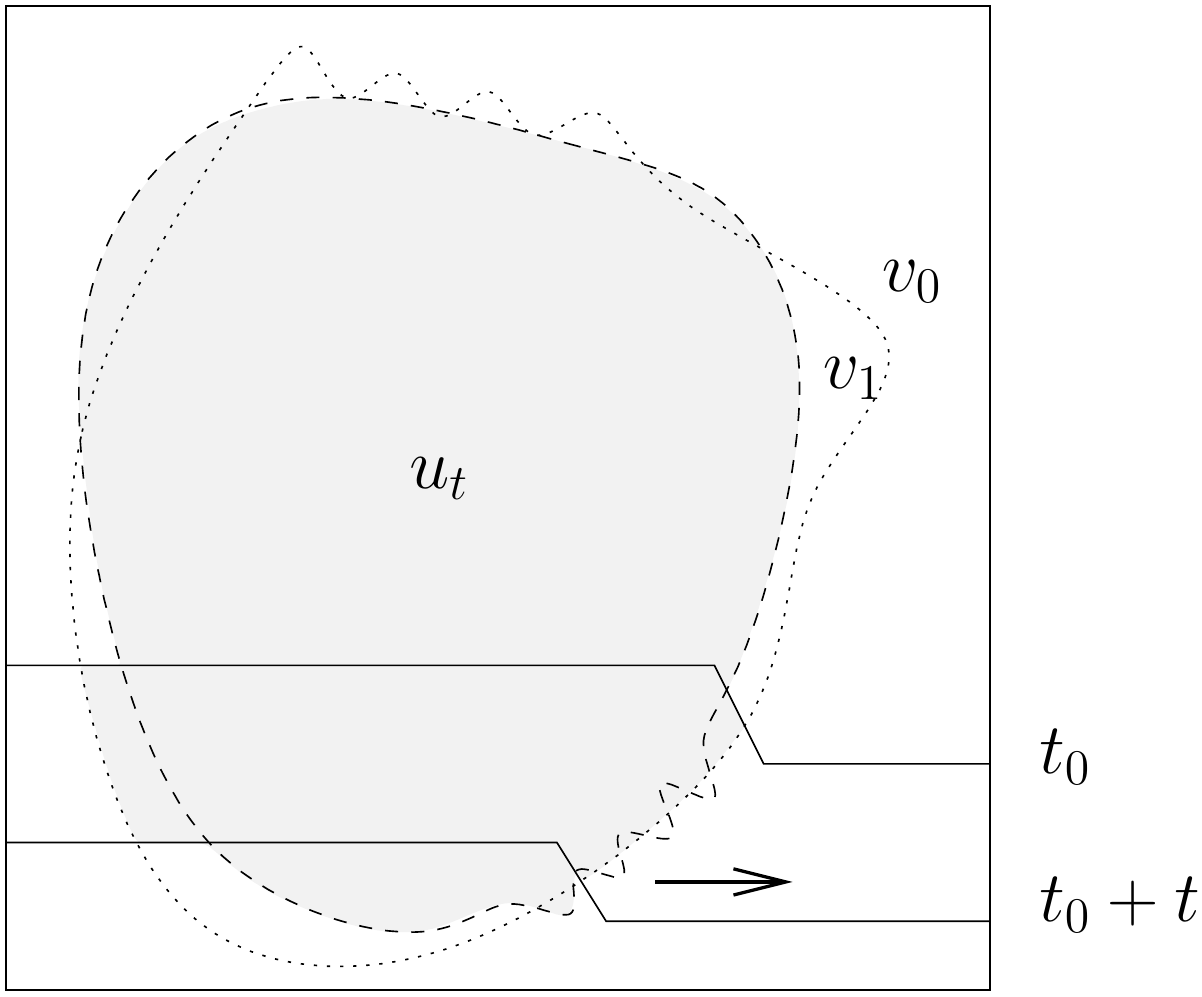}\end{center}
  \caption{\label{fig-Et-Delta}The construction of $U_t$, and the
  interpolation $u_t$ between $v_0$ and $v_1$.}
\end{figure}

The second key argument is periodicity: as seen on Figure \ref{fig-Et-Delta},
the interpolation between $v_0$ and $v_1$ has to choose first the region where
the cost of $v_1$ is smaller than that of $v_0$ -- which means that we have to
fix $t_0$ carefully.

\begin{proof}
  (Theorem \ref{thm-Kdisc-Kcont}). Let $\delta > 0$. There exists $\varepsilon
  \in (0, 2 \delta]$ and $k \in \mathbbm{N}$, together with $v \in
  \mathcal{C}_{\varepsilon} (u_0)$, such that
  \[ \max_{i = 0 \ldots k} \mathcal{F}^r (v) \leqslant \mathcal{K}^r (u_0) +
     \delta . \]
  Starting from $v$, we construct a continuous evolution $u \in
  \mathcal{C}(u_0)$ that has a maximal cost not much larger than that of $v$.
  It is enough to do the interpolation between two successive $v_i$, as one
  can paste together the successive interpolations to deduce the continuous
  evolution $u$.
  
  Hence we consider $v_0, v_1 \in \tmop{BV}$ and assume that $\|v_0 - v_1
  \|_{L^1} \leqslant 2 \delta$. We let
  \[ \Delta = \left\{ x: v_0 (x) \neq v_1 (x) \right\} \]
  which has a volume at most $\delta$. Lemma \ref{lem-Ut} applies and there is
  $U_t$ with properties (i)-(v). Given $t_0 \in \mathbbm{R}$ and $t \in [0,
  1]$ we let
  \[ G_{t_0, t} = \left\{ x \in [0, 1]^d: \exists k \in \mathbbm{Z} \text{: }
     x + k\tmmathbf{e}_d \in U_{t_0 + t} \setminus U_{t_0} \right\}, \]
  for any $t_0$ the set $G_{t_0, t}$ increases continuously from the empty to
  the full set in $[0, 1]^d$, makes the surface $\mathcal{H}^{d - 1}
  (\partial^{\star} u_0 \cap G_{t_0, t})$ a continuous function of $t$, and
  the area
  \[ \mathcal{H}^{d - 1} \left( \partial G_{t_0, t} \cap \Delta \cap
     \dot{\mathcal{T}} \right) \leqslant 14 \sqrt{\delta} \]
  small, for $\delta > 0$ small enough. Now we define
  \[ u_t (x) = \left\{ \begin{array}{ll}
       v_0 (x) & \text{if } x \notin G_{t_0, t}\\
       v_1 (x) & \text{if } x \in G_{t_0, t} .
     \end{array} \right. \]
  The cost of $u_t$ decomposes in the following way: it is the sum of the cost
  of $v_0$ in $G_{t_0, t}^c$, of the cost of $u_1$ in $G_{t_0, t}$, and of the
  cost of $\partial G_{t_0, t}$ in $\Delta$. In other words,
  \begin{equation}
    \mathcal{F}^r (u_t) \leqslant \mathcal{F}^r (v_0) -\mathcal{F}^r_{G_{t_0,
    t}} (v_0) +\mathcal{F}^r_{G_{t_0, t}} (v_1) + 14 \sqrt{\delta}
    \label{eq-Fut}
  \end{equation}
  where $\mathcal{F}^r_E (u)$ stands for
  \[ \mathcal{F}^r_E (u) = \int_{\partial^{\star} u \cap \partial^{\star} u_0
     \cap E} \tau^r d\mathcal{H}^{d - 1} + \int_{(\partial^{\star} u \setminus
     \partial^{\star} u_0) \cap E} \tau^q (\tmmathbf{n}_.^u) d\mathcal{H}^{d -
     1} . \]
  It is clear that the initial cost of $u_t$ is $\mathcal{F}^r (v_0)$ and that
  its final cost is $\mathcal{F}^r (v_1)$. Yet in the interval $(0, 1)$ it
  could be that $G_{t_0, t}$ selects first the region where $v_1$ has a larger
  cost than $v_0$, leading to a maximal cost larger than expected. We rule out
  this possibility with an appropriate choice for $t_0$ -- see Figure
  \ref{fig-Et-Delta} for an illustration of the discussion below.
  
  For $t_0 \in \mathbbm{R}$ and $t \in [0, 1)$ we consider
  \[ f (t_0, t) =\mathcal{F}^r_{G_{t_0, t}} (v_1) -\mathcal{F}^r_{G_{t_0, t}}
     (v_0) . \]
  Our aim is to extend $f$ to arbitrary values of $t \in \mathbbm{R}^+$. For
  $k \in \mathbbm{Z}$, we denote by $v^k_i$ the translated of $v_i$ by
  $k\tmmathbf{e}_d$, then for any $t_0 \in \mathbbm{R}$ and $t \in [0, 1)$ we
  have
  \[ f (t_0, t) = \sum_{k \in \mathbbm{Z}} \left( \mathcal{F}^r_{U_{t_0 + t}
     \setminus U_{t_0}} \left( v_1^k \right) -\mathcal{F}^r_{U_{t_0 + t}
     \setminus U_{t_0}} \left( v_0^k \right) \right) \]
  from the definition of $G_{t_0, t}$. The latter formula permits to extend
  $f$ to $\mathbbm{R} \times \mathbbm{R}^+$ and puts in evidence the existence
  of a function $g: \mathbbm{R} \rightarrow \mathbbm{R}$ such that
  \[ f (t_0, t) = g (t_0 + t) - g (t_0) \text{, \ \ \ \ } \forall (t_0, t) \in
     \mathbbm{R} \times \mathbbm{R}^+ . \]
  This function is, apart from a linear correction, $1$-periodic: for all $t
  \in \mathbbm{R}$,
  \[ g (t + 1) = g (t) +\mathcal{F}^r (v_1) -\mathcal{F}^r (v_0), \]
  in other words,
  \[ g (t) = h (t) + t \left( \mathcal{F}^r (v_1) -\mathcal{F}^r (v_0) \right)
  \]
  where $h$ is a $1$-periodic function. Now we fix $t_0$ such that
  \[ h (t_0) \geqslant \sup_t h (t) - \sqrt{\delta}, \]
  it is immediate that
  \begin{eqnarray*}
    f (t_0, t) & = & g (t_0 + t) - g (t_0)\\
    & = & h (t_0 + t) - h (t_0) + t \left( \mathcal{F}^r (v_1) -\mathcal{F}^r
    (v_0) \right)\\
    & \leqslant & t \left( \mathcal{F}^r (v_1) -\mathcal{F}^r (v_0) \right) +
    \sqrt{\delta} .
  \end{eqnarray*}
  Reporting into (\ref{eq-Fut}) we conclude that $u_t$ is a satisfactory
  interpolation between $v_0$ and $v_1$: provided that $\delta > 0$ is small
  enough,
  \[ \mathcal{F}^r (u_t) \leqslant (1 - t)\mathcal{F}^r \left( v_0 \right) +
     t\mathcal{F}^r \left( v_1 \right) + 15 \sqrt{\delta} \text{, \ \ }
     \forall t \in [0, 1) . \]
\end{proof}

\subsubsection{The gap in surface energy in the isotropic case}

\label{sec-Kr-circle}In this paragraph we use Theorem~\ref{thm-Kdisc-Kcont} to
compute the gap in surface energy in the two dimensional, isotropic case.

\begin{lemma}
  Assume $d = 2$, $\tau^q (\tmmathbf{n}) = 1$ for all $\tmmathbf{n} \in S^{d -
  1}$ and consider
  \[ u_0 = \chi_B \text{ \ \ and \ \ } \tau^r (x) = \lambda, \forall x \in
     \partial^{\star} u_0 \]
  where $B$ is the disk of radius $r < 1 / 2$ centered at $(1 / 2, 1 / 2)$,
  and $\lambda \in (0, 1)$. Then
  \[ \mathcal{K}^r (u_0) = 2 r \left[ \sqrt{1 - \lambda^2} - \lambda
     \tmop{acos} \lambda \right] . \]
\end{lemma}

An optimal continuous evolution in this setting is illustrated on Figure
\ref{fig-cde-ball1}.

\begin{figure}[h!]
\begin{center}\includegraphics[width=6cm]{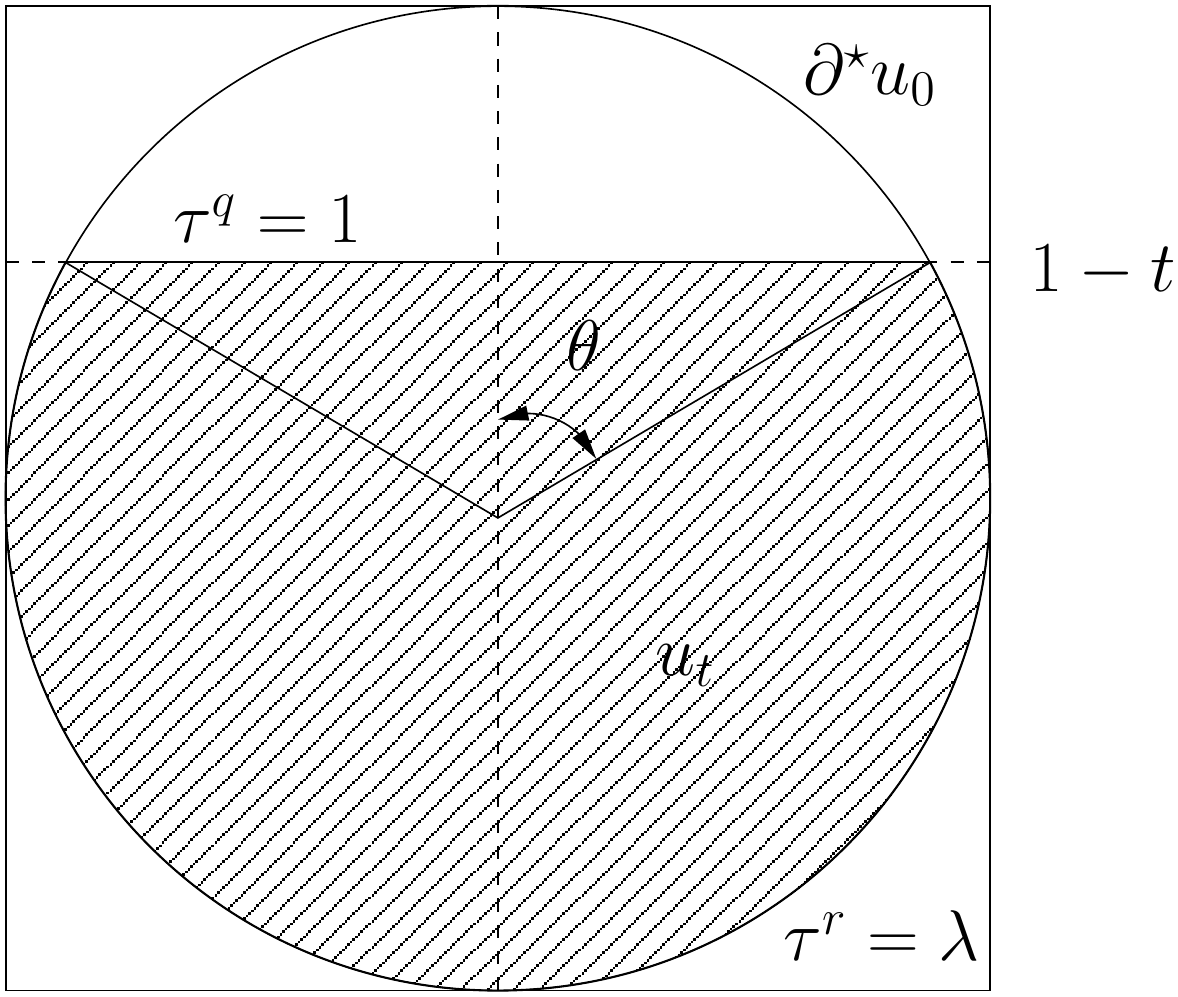}\end{center}
\caption{\label{fig-cde-ball1}A continuous evolution of minimal cost.}
\end{figure}

\begin{proof}
  In order to simplify the notation we will consider $r = 1 / 2$. The upper
  bound
  \[ \mathcal{K}^r (u_0) \leqslant \sup_{\theta \in [0, \pi]} (\sin \theta -
     \lambda \theta) \]
  is immediate if one considers the continuous evolution $(u_t)_{t \in [0,
  1]}$ defined by
  \[ u_t (x) = \left\{ \begin{array}{ll}
       - 1 & \text{if } x \in B \text{ and } x \cdot \tmmathbf{e}_2 \leqslant
       1 - t\\
       1 & \text{else}
     \end{array} \right. \]
  and $\theta$ satisfying $1 - t = 1 / 2 + \sin \theta$, as illustrated on
  Figure \ref{fig-cde-ball1}. The lower bound is scarcely more difficult to
  establish: given $(u_t)_{t \in [0, 1]} \in \mathcal{E}(u_0)$ a continuous
  evolution with continuous detachment, there is $t \in (0, 1)$ such that
  \[ \mathcal{H}^1 (\partial^{\star} u_0 \setminus \partial^{\star} u_t) =
     \tmop{acos} \lambda . \]
\begin{figure}[h!]
\begin{center}\includegraphics[width=\textwidth]{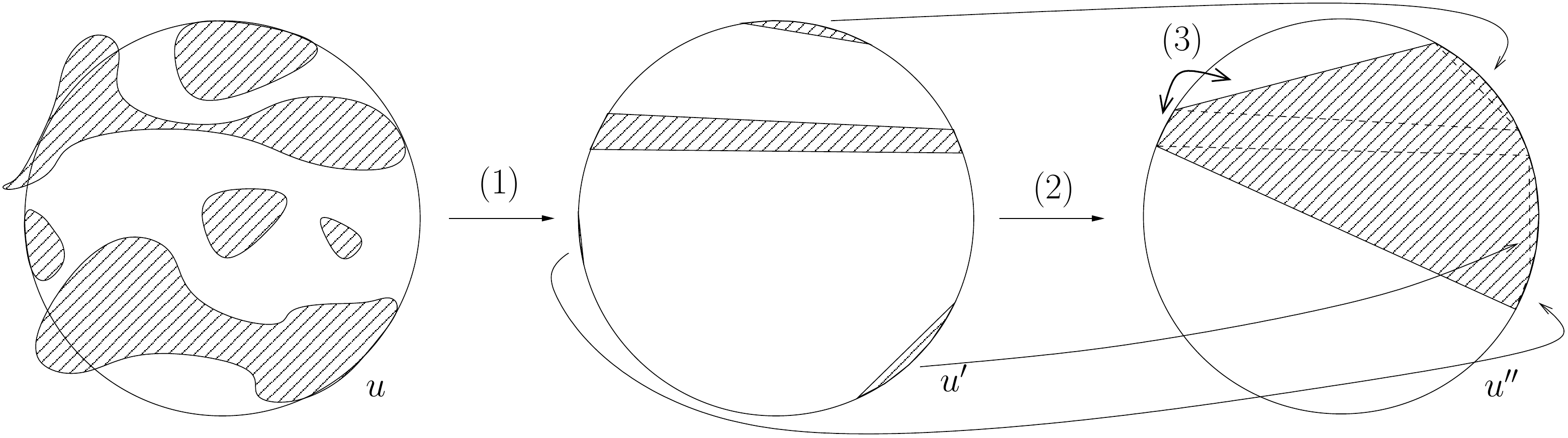}\end{center}
  \caption{\label{fig-cde-ball2}Reduction of $u$ to a portion of disk in three
  steps. The surface energy decreases, the length of contact is preserved.}
\end{figure}

  Optimizing the droplets of $u_t$ as in Figure \ref{fig-cde-ball2} -- we
  replace each portion of the interface not in $\partial^{\star} u_0$ with a
  segment (step (1)) -- we obtain a profile $u'_t$ with a lower cost, yet it still has
  the same contact length $\tmop{acos} \lambda$ with $\partial^{\star} u_0$.
  By isotropy of surface tension it is possible to aggregate the droplets
  together (step (2)) and obtain $u''_t$ with a unique droplet and a lower cost,
  preserving again the length of contact. At last, inverting the order of the
  segments and arcs and optimizing again we see that the profile of lower cost
  that satisfies $\mathcal{H}^1 (\partial^{\star} u_0 \setminus
  \partial^{\star} u) = \tmop{acos} \lambda$ coincides, apart from a rotation,
  with the profiles considered in the upper bound. The claim follows.
\end{proof}

\subsubsection{The gap in surface energy when the Wulff crystal is a square}

\label{sec-Kr-square}Here we compute the gap in surface energy for another
simple case, when the Wulff crystal $\mathcal{W}^q$ is a square. We also show
that the gap in surface energy is strictly bigger than the cost of the less
likely overall magnetization.

\begin{lemma}
  Assume $d = 2$, $\tau^q (\tmmathbf{n}) =\|\tmmathbf{n}\|_1$ for all
  $\tmmathbf{n} \in S^{d - 1}$ and
  \[ u_0 = \chi_C \text{ \ \ and \ \ } \tau^r (x) = \lambda, \forall x \in
     \partial^{\star} u_0 \]
  where $C = [1 / 2 - r, 1 / 2 + r]^2$, $r < 1 / 2$ and $\lambda \in [0, 1]$.
  Then, we have
  \[ \mathcal{K}^r (u_0) = 2 r \left[ 1 - \lambda \right] . \]
\end{lemma}

\begin{proof}
  Again we consider $r = 1 / 2$ in order to simplify the notations. The upper
  bound on the additional cost is immediate considering $u_t = \chi_{\{x \cdot
  \tmmathbf{e}_2 \leqslant 1 - t\}}$. For the lower bound we need a finer
  analysis. First, it is a consequence of the assumption on $\tau^q$ that for
  any open, connected $U \subset \mathbbm{R}^2$ with extension $h^1, h^2$ in
  the canonical directions, that is:
  \[ h (k) = \sup_{x \in U} x \cdot \tmmathbf{e}_k - \inf_{x \in U} x \cdot
     \tmmathbf{e}_k, \]
  we have
  \[ \mathcal{F}^q (\chi_U) \geqslant 2 h^1 + 2 h^2 . \]
  Then, we decompose a profile configuration $u$ into its droplets $(U_i)_{i
  \geqslant 0}$. We call $h_i^1, h_i^2$ the extension of $U_i$ in the
  canonical directions and let $l_i$ the length of contact between
  $\partial^{\star} \chi_{U_i}$ and $\partial^{\star} u_0$, so that, for all
  $i$:
  \begin{eqnarray*}
    \mathcal{F}^r (\chi_{U_i}) & \geqslant & 2 h^1 + 2 h^2 - (1 - \lambda) l_i
  \end{eqnarray*}
  If a droplet $U_i$ touches two opposite faces of $\partial^{\star} u_0$, say
  $h^1 = 1$, then its extension in the orthogonal direction is at least $h^2
  \geqslant (l_i - 1) / 2$ and the inequality
  \begin{eqnarray*}
    \mathcal{F}^r (\chi_{U_i}) & \geqslant & 1 + \lambda l_i
  \end{eqnarray*}
  follows. If on the opposite the droplet is in contact with at most two
  adjacent sides of $\partial^{\star} u_0$, we have $h^1 + h^2 \geqslant l_i$
  and hence
  \begin{eqnarray*}
    \mathcal{F}^r (\chi_{U_i}) & \geqslant & (1 + \lambda) l_i .
  \end{eqnarray*}
  Assume now that the total length of contact is $3$, i.e. that $\sum_i l_i =
  3$. A consequence of the former lower bounds is that, whether or not some
  droplet touches two opposite faces, the cost of $u$ is at least
  $\mathcal{F}^r (u) \geqslant 1 + 3 \lambda$. The claim follows.
\end{proof}

We conclude this work with a comparison between the bottleneck due to the
positivity of $\mathcal{K}^r (u_0)$ and the one due to the continuous
evolution of the magnetization:

\begin{lemma}
  \label{lem:Kr:mag}In the settings of the former lemma, with furthermore
  $\lambda \in (0, 1)$, we have
  \[ \mathcal{K}^r (u_0) > \sup_{m \in [- 1, 1]} \inf_{u \in \tmop{BV}:
     \int_{[0, 1]^d} u = m} \mathcal{F}^r (u) -\mathcal{F}^r (u_0) . \]
\end{lemma}

\begin{figure}[h!]
\begin{center}  \includegraphics[width=0.8\textwidth]{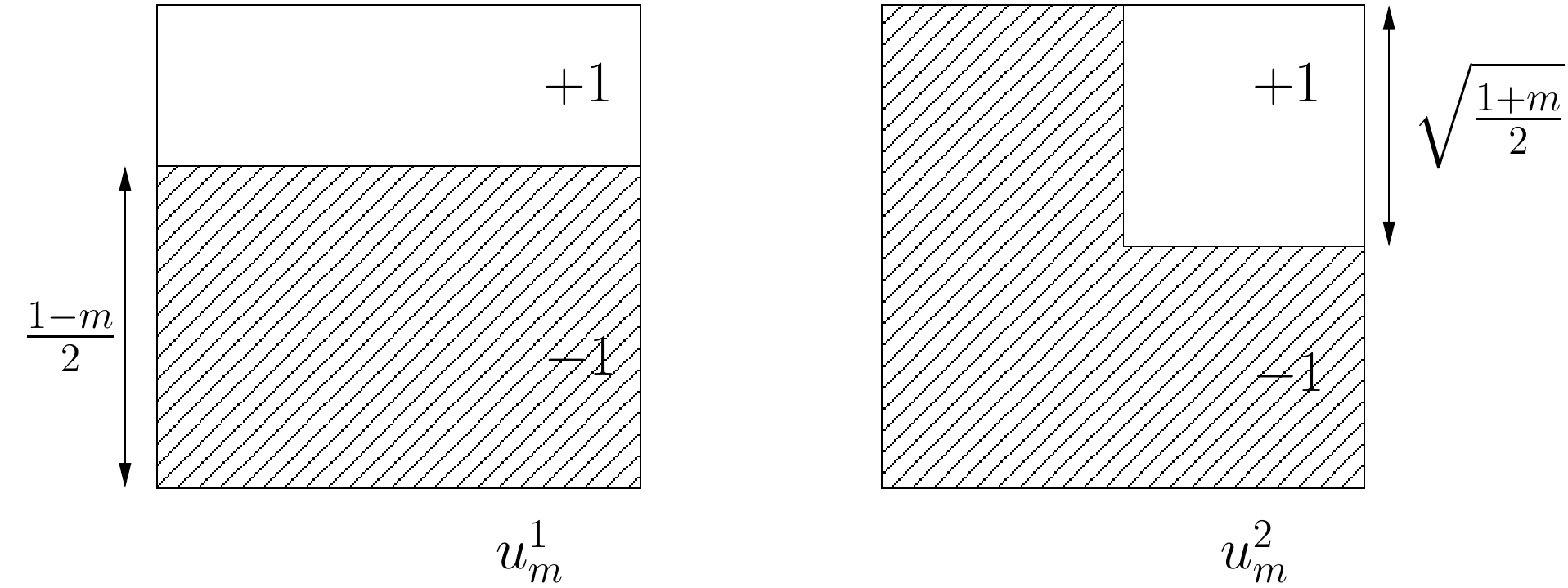}\end{center}
  \caption{\label{fig-opt-vol}The two profiles $u^1_m$ and $u^2_m$: for $m
  \simeq 1$ the first one is better, for $m \simeq - 1$ the second one has a
  smaller cost.}
\end{figure}

\begin{proof}
  We provide an upper bound for the right hand term, considering for a given
  $m$ the two profiles (see Figure \ref{fig-opt-vol})
  \begin{eqnarray*}
    u^1_m & = & \left\{ \begin{array}{ll}
      - 1 & \text{if } x \in [0, 1]^d \text{ and } x \cdot \tmmathbf{e}_2
      \leqslant \frac{1 - m}{2}\\
      1 & \text{else}
    \end{array} \right.
  \end{eqnarray*}
  and
  \begin{eqnarray*}
    u^2_m & = & \left\{ \begin{array}{ll}
      - 1 & \text{if } x \in [0, 1]^d \text{ and } \min (x 
      \cdot
      \tmmathbf{e}_1, x \cdot \tmmathbf{e}_2) \leqslant 1 - \sqrt{\frac{1 +
      m}{2}}\\
      1 & \text{else}
    \end{array} \right.
  \end{eqnarray*}
  that both satisfy the volume constraint $\int_{[0, 1]^d} u = m$. It is
  immediate that
  \[ \mathcal{F}^r (u^1_m) = 1 + (2 - m) \lambda \text{ \ and \ }
     \mathcal{F}(u^2_m) = 4 \lambda + 2 (1 - \lambda) \sqrt{\frac{1 + m}{2}},
  \]
  hence
  \[ \sup_{m \in [- 1, 1]} \inf_{u \in \tmop{BV}: \int_{[0, 1]^d} u = m}
     \mathcal{F}^r (u) \leqslant \sup_{m \in [- 1, 1]} \min \left(
     \mathcal{F}(u^1_m), \mathcal{F}(u^2_m) \right) . \]
  Note that $\mathcal{F}(u^1_m)$ decreases with $m$ while $\mathcal{F}(u^2_m)$
  increases with $m$. Because of their extremal values there exists some $m_0
  \in (0, 1)$ at which $\mathcal{F}(u^1_{m_0}) =
  \text{$\mathcal{F}(u^2_{m_0})$} = \sup_{m \in [- 1, 1]} \min \left(
  \mathcal{F}(u^1_m), \mathcal{F}(u^2_m) \right)$, and since $m_0 < 1$ we have
  in particular $\mathcal{F}(u^1_{m_0}) <\mathcal{F}(u^1_{- 1}) = 1 + 3
  \lambda$ which is the maximal cost of an optimal continuous detachment
  evolution.
\end{proof}

{\bf Acknowledgements.} It is a pleasure to thank Thierry Bodineau for his commitment and 
continuous support throughout the duration of this work. I am also grateful to Fabio Martinelli for pointing out the scheme of proof of Theorem~\ref{thm-upb-A}.

\end{document}